\documentclass[12pt]{article}
\usepackage{amsmath,amsthm, verbatim,amssymb,amsfonts,amscd,diagrams, graphicx, mathrsfs}
\usepackage[usenames,dvipsnames]{color}
\usepackage{amstext}
\usepackage{hyperref}
\usepackage{color}

\theoremstyle{plain}
\newtheorem{theorem}{Theorem}[section]
\newtheorem{corollary}[theorem]{Corollary}
\newtheorem{lemma}[theorem]{Lemma}
\newtheorem{proposition}[theorem]{Proposition}
\newtheorem*{cor}{Corollary}
\newtheorem*{theorem*}{Theorem}
\newtheorem*{proposition*}{Proposition}

\theoremstyle{definition}
\newtheorem{definition}[theorem]{Definition}

\theoremstyle{remark}
\newtheorem{remark}[theorem]{Remark}
\newtheorem{example}[theorem]{Example}

\newarrow{ul}---->
\newarrow{Backwards}<----

\newcommand{\into}{\hookrightarrow}
\newcommand{\Z}{\mathbb{Z}}
\newcommand{\Q}{\mathbb{Q}}
\newcommand{\R}{\mathbb{R}}

\newcommand{\C}{\mathbb{C}}

\newcommand{\bd}{\partial}

\newcommand{\mc}[1]{\mathcal{#1}}
\newcommand{\ms}[1]{\mathscr{#1}}

\newcommand{\mf}{\mathfrak}

\newarrow{Onto}----{>>}
\newarrow{Equals}=====

\newcommand{\Lk}{\text{Lk}}

\begin{document}

\title{Stratified and unstratified bordism of pseudomanifolds}
\author{Greg Friedman\thanks{
This material is based upon work supported by the National Science Foundation under Grant Number (DMS-1308306) to Greg Friedman
and by a grant from the Simons Foundation (\#209127 to Greg Friedman). }}

\date{May 11, 2015}

\maketitle

\begin{abstract}
We study bordism groups and bordism homology theories based on pseudomanifolds and stratified pseudomanifolds. The main seam of the paper demonstrates that when we uses classes of spaces determined by local link properties, the stratified and unstratified bordism theories are identical; this includes the known examples of pseudomanifold bordism theories, such as bordism of Witt spaces and IP spaces. Along the way, we relate the  stratified and unstratified  points of view for describing  various classes of pseudomanifolds. 
\end{abstract}



\tableofcontents

\section{Introduction}

In \cite{Si83}, Siegel introduced Witt spaces (or, more specifically, $\Q$-Witt spaces in more modern terminology) and computed their bordism groups. Siegel's $\Q$-Witt spaces are oriented piecewise linear (PL) pseudomanifolds that satisfy a certain local vanishing condition, the $\Q$-Witt condition, with respect to the lower middle perversity intersection homology groups of Goresky and MacPherson \cite{GM1}. While intersection homology is defined with respect to an appropriate stratification of a pseudomanifold, in the instances relevant for Witt spaces, both intersection homology and the $\Q$-Witt condition turn out to be independent of the choice of stratification. This  makes it reasonable to think of a $\Q$-Witt space purely in terms of the underlying PL  space without any reference to a particular stratification. One can then define bordism groups of such spaces by employing $\Q$-Witt spaces with boundary. In this way, building on a general  construction of bordisms by Akin \cite{Ak75}, Siegel defined and computed the oriented bordism groups that we will denote $\Omega^{|\Q-\text{Witt}|}_*$. These computations implied, via work of Sullivan \cite{Sul70}, that the resulting  bordism homology theory, when tensored with $\Z\left[\frac{1}{2}\right]$, computes connective $ko$-homology at odd primes. 

Building on Siegel's work  and using local conditions defined by  Goresky and Siegel \cite{GS83}, Pardon \cite{Pa90} defined and computed the bordism groups of a more restrictive class of pseudomanifolds called ``IP spaces'' and used these computations to formulate a characteristic variety theorem. Other studies of analogous bordism groups of pseudomanifolds satisfying various local intersection homology properties were performed in \cite{Go84, GP89, GBF21, GBF33, GBF34}. 

Given that intersection homology is defined directly in terms of \emph{stratified} pseudomanifolds, it is natural to ask what happens if  one attempts to define these various bordism groups in terms of spaces \emph{with}  their stratifications. For example, one could define $\Omega^{\Q-\text{Witt}}_*$ as the bordism group whose generators are oriented \emph{stratified} pseudomanifolds that satisfy the local $\Q$-Witt condition and whose relations are through bordisms by oriented \emph{stratified} pseudomanifolds satisfying the same condition. 
There is an obvious surjective forgetful map $\mf s: \Omega^{\Q-\text{Witt}}_*\to \Omega^{|\Q-\text{Witt}|}_*$ that forgets the choice of stratification, but it is not  obvious that this map is injective as well. Directly from the definitions, injectivity of $\mf s$ would require showing that if the underlying spaces $|X|$ and $|Z|$  of two oriented stratified $\Q$-Witt spaces  $X$ and $Z$  are bordant via some oriented $\Q$-Witt space $|W|$ without any particular choice of stratification, then there exists an oriented \emph{stratified} $\Q$-Witt space  $Y$ that provides  a \emph{stratified} bordism between $X$ and $Z$; in particular, the stratification of $Y$ must be compatible with the stratifications of $X$ and $Z$ at  the boundaries. A priori, it is not obvious that this can always be done. 

We will show here that the maps $\mf s:\Omega^{\mc C}_*\to \Omega^{|\mc C|}_*$ are, in fact, isomorphisms, where $\mc C$ stands for any one of a broad variety of  classes  of stratified pseudomanifolds that we term  \emph{intrinsic weak stratified bordism classes}, or IWS classes, and $|\mc C|$ consists of the underlying spaces of $\mc C$, without their stratifications. Examples of such pairs of classes  include the stratified and unstratified versions of Witt spaces, IP spaces, and the various other spaces whose bordism groups were studied by Goresky and Pardon in \cite{GP89}.
Thus, for example, the ``stratified'' and ``unstratified'' versions of the Witt and IP  bordism groups  of pseudomanifolds  are  identical. In fact, we will prove the  stronger assertion that if the underlying spaces of stratified pseudomanifolds  $X$ and $Z$ are bordant in the unstratified sense via some  $W$,  then $W$ itself can be stratified  to provide a stratified bordism between $X$ and $Z$.

In the later stages of the paper, we show that the isomorphism between stratified and unstratified bordism groups\footnote{We use the shorthand ``stratified bordism groups'' for what should properly be called ``bordism groups of stratified pseudomanifolds'', and similarly for ``unstratified bordism groups''.} extends to an isomorphism of bordism homology theories for IWS classes that are determined by local link properties.  Along the way, we also provide something of an axiomatization of the process of moving back and forth between the stratified and unstratified worlds, developing various classes of pseudomanifolds and studying the roles these classes play in the construction  of bordism theories. This section is heavily influenced by Siegel's work on $\Q$-Witt spaces \cite{Si83} and Akin's work on bordism theories \cite{Ak75}.

Throughout this paper, we will work entirely in the piecewise-linear (PL) category, and our results will depend strongly upon PL techniques. In particular, one of our critical tools will be the fact that a PL stratified pseudomanifold remains a PL stratified pseudomanifold in its intrinsic stratification. 
 While it seems possible that analogous bordism results might hold in the purely topological category (with all spaces being topological pseudomanifolds), for now this remains an open question. The reliance of the PL bordism theory on PL general position,  regular neighborhoods, and surgery techniques, and the difficulty of extending such techniques even to topological manifolds,   makes such a foray into the topological category forbidding. Furthermore, a topological  pseudomanifold with its intrinsic stratification is not necessarily known to be a stratified pseudomanifold. So, in the topological world, one would need very different techniques. 
 
 To streamline the exposition, we focus primarily on oriented pseudomanifolds and the corresponding ``oriented bordism groups,'' though all of our results hold in the unoriented setting via easy modification of our arguments. 

\paragraph{Acknowledgment.} The existence of this paper is due to Jim McClure, who asked me a very reasonable question concerning the proper definition of bordism groups for Witt and IP spaces. Unfortunately, these definitions are not always completely clear in the literature. In \cite{Si83}, Siegel quite explicitly defines Witt bordism in the unstratified sense. This also seems to be the choice in \cite{GP89}, though various classes of the spaces studied there are defined in terms of properties of links in a given stratification, and it is often not made completely clear  that these properties are independent of the choice of stratification so that unstratified bordism makes sense. We resolve these issues here both by explicitly showing that previously considered pseudomanifold class conditions from \cite{Si83, Pa90, GP89, GBF21, GBF33, GBF34} are stratification independent and by showing that, when this is the case, stratified and unstratified notions of bordism yield the same bordism theories. This work is then utilized in McClure's paper with Banagl and Laures \cite{BLM}.

\paragraph{Outline of results.} We now proceed to outline the paper and its main results.

Section \ref{S: background} contains basic background concerning pseudomanifolds and stratified pseudomanifolds, while Section \ref{S: susps} provides some reminders concerning various facts of PL topology that we will need. 

Section \ref{S: bordisms} contains our geometric construction of stratified bordisms from unstratified bordisms. The key point here, and essentially the heart of the entire paper, is the explicit construction of a stratified bordism from a stratified pseudomanifold $X$ to the stratified pseudomanifold $X^*$, which has the same underlying space as $X$ but is stratified by the intrinsic stratification\footnote{Here we make crucial use of the PL category; the intrinsic stratification of a topological pseudomanifold is not known, in general, to yield a pseudomanifold stratification.}. This leads to the following corollary, which is the principal conclusion of this section. Since every pseudomanifold is, by definition, the underlying space of a stratified pseudomanifold, this corollary implies that unstratified bordisms can be stratified to match given stratifications of their boundaries.  To explain the notation of the corollary, $|X|$ denotes the unstratified underlying space of the stratified pseudomanifold $X$, and a $\bd$-stratified pseudomanifold, defined in detail below, plays the role in pseudomanifold theory analogous to $\bd$-manifolds (also called, somewhat misleadingly, manifolds with boundary) in manifold theory.

\begin{cor}[Corollary \ref{C: given bord}]
Suppose that $X$, $Z$ are two  compact (orientable) PL stratified pseudomanifolds and that there exist (compatibly oriented) PL stratified pseudomanifolds $X'$ and $Z'$ such that $|X|\cong|X'|$,  $|Z|\cong |Z'|$, and there exists  a PL $\bd$-stratified pseudomanifold $Y'$ such that $\bd Y'\cong X'\amalg Z'$ (or, in the oriented case, $\bd Y'\cong X'\amalg -Z'$). Then there exists a stratification of $|Y'|$ as a   PL $\bd$-stratified pseudomanifold $Y$ such that $\bd Y\cong X\amalg Z$ (or, in the oriented case, $\bd Y\cong X\amalg -Z$). If none of $X, X', Z,Z', Y'$ have a codimension one stratum, then $Y$ can be chosen to have no codimension one strata. 
\end{cor}

In Section \ref{S: bordism groups}, we study stratified and unstratified bordism groups. It is here that we define IWS classes of pseudomanifolds; these are the classes of pseudomanifolds for which both stratified and unstratified bordism can be defined and for which it is reasonable to compare the two\footnote{See Definitions \ref{D: weak strat class} and \ref{D: underlying class}; essentially these classes are designed to allow for definitions of both stratified and unstratified bordism groups, so these are classes of pseudomanifolds that are required to be closed under stratified homeomorphisms, taking boundaries, taking products with $I$ (if the boundary is empty), gluing of bordisms, and changes of orientation and stratification.}. In particular, these spaces are determined by properties that are intrinsic to the spaces and not their stratifications. We prove the following:

\begin{theorem}[Theorem \ref{T: bord group}]\label{T: bord group0}
If $\mc C$ is an IWS class, the forgetful maps $\mf s:\Omega^{\mc C}_n\to \Omega^{|\mc C|}_n$  are well-defined isomorphisms. 
\end{theorem}

\noindent In Section \ref{S: sps}, we  define \emph{classes of stratified pseudomanifold singularities} (denoting a generic such class $\mc E$), which provide a way to  construct IWS classes by specifying conditions on a stratified pseudomanifold's links\footnote{See Definition \ref{D: sps}. Essentially these are classes of compact positive-dimensional stratified pseudomanifolds without boundary. Furthermore, the class must be closed under changes of stratification, homeomorphisms, and suspensions, and must include the positive-dimensional spheres. Given such an $\mc E$, the class of PL $\bd$-stratified pseudomanifolds whose links are all in $\mc E$ constitutes an IWS class, $\mc C_{\mc E}$, by Lemma \ref{L: prop P}.}.  As examples, we show in Section \ref{E: examples} that Witt spaces, IP spaces, and the pseudomanifold classes of \cite{GP89}, among others, are all IWS classes;  so, in particular, Theorem \ref{T: bord group} applies to such classes. \emph{At this point, we will have answered McClure's original motivating question, and thus the reader who is interested primarily in pseudomanifold bordism \emph{groups}, including the previously-studied examples (Witt spaces, etc.), can consider the first five chapters to be a self-contained treatment. }

Section \ref{S: bordism homology} contains our study of bordism as a homology theory. After defining stratified and unstratified bordism homology theories, respectively $\Omega_*^{\mc E}(\cdot)$ and $\Omega_*^{|\mc E|}(\cdot)$, based upon (stratified) pseudomanifolds with links in the class $\mc E$, 
our ultimate result comes in the following form, which is  analogous to our result concerning bordism \emph{groups} in Theorem \ref{T: bord group0}/Theorem \ref{T: bord group}:

\begin{theorem}[Theorem \ref{T: theory}]\label{T: theory0}
The natural transformation $\mf s: \Omega_*^{\mc E}(\cdot)\to \Omega_*^{|\mc E|}(\cdot)$  is an  isomorphism of homology theories.  
\end{theorem}

The casual reader can easily skip most of Section 6 to head directly to this result. However, in order to put Theorem \ref{T: theory} on a firm technical footing, we first need to construct our bordism homology theories rigorously, and this requires some detailed  technical work concerning different approaches to constructing classes of pseudomanifolds based on prescribing properties for either their links (in the stratified case) or their polyhedral links\footnotemark (in the unstratified case). Toward this end, in Section \ref{S: classes} we first define \emph{classes of pseudomanifold singularities}, which we denote generically by $\mc G$. In contrast to the classes $\mc E$, which are classes of  links of strata in \emph{stratified} pseudomanifolds, the $\mc G$ are classes of polyhedral links of points in \emph{unstratified} pseudomanifolds.  
The classes of pseudomanifold singularities $\mc G$ are instances of ``classes of singularities'' in the sense of Akin \cite{Ak75}, and so Akin's technology can be applied to construct \emph{unstratified} bordism homology theories of spaces whose polyhedral links are in $\mc G$. 
A class $\mc E$ of stratified pseudomanifold singularities determines a class $\mc G_{\mc E}$ of pseudomanifold singularities, which leads to the following useful technical proposition:

\footnotetext{This will be explained more fully below. Briefly: a link of a stratum in a stratified pseudomanifold $X$ is a stratified pseudomanifold $L$  such that some point $x\in X$ has a distinguished neighborhood stratified homeomorphic to $\R^i\times cL$; a polyhedral link $|\Lk(x)|$ in the unstratified pseudomanifold  $|X|$ is a PL space such that a point $x\in |X|$ has a neighborhood homeomorphic to the cone on $|\Lk(x)|$.}

\begin{proposition*}[Proposition \ref{P: pm links}]
Let $\mc E$ be a class of stratified pseudomanifold singularities, and let $X$ be a $\bd$-stratified pseudomanifold. Then the links of the strata of $X$ are contained in $\mc E$  if and only if the polyhedral links of points of $|X|-|\bd X|$ are contained in $\mc G_{\mc E}$.
\end{proposition*}

\noindent However, we show in Lemma  \ref{L: E and G} that the assignment $\mc E\to \mc G_{\mc E}$ is not a bijection, though it is surjective. 

In Section \ref{S: F}, we then look at the relation between IWS classes with links in $\mc E$ and the ``bordism sequences'' of Akin with polyhedral links in $\mc G_{\mc E}$. 
 We also verify in Section \ref{S: unstrat bord}   that, although Akin works in the broader class of polyhedra, the Akin unstratified bordism homology groups we denote $\Omega^{|\mc E|}_n(pt)$, determined by the class $\mc G_{\mc E}$, agree with the unstratified pseudomanifold bordism groups $\Omega^{|\mc E|}_n$ of Section \ref{S: bordism groups}. This is the content of Lemma \ref{L: akin pt is bordism}.

Finally, in Section \ref{S: theory}, we show how to build \emph{stratified} bordism homology theories with  links in the class of stratified pseudomanifold singularities $\mc E$ and demonstrate that these turn out to be the same homology theories as their unstratified counterparts, which is Theorem \ref{T: theory0}/Theorem \ref{T: theory}.

Section \ref{S: Siegel} contains an exploration of one additional interesting feature of  Siegel's application of Akin's bordism theory to construct the $\Q$-Witt bordism groups. The interesting point here is that, although Siegel works throughout most of \cite{Si83} with two particular sets of criteria for recognizing a $\Q$-Witt space $X$, one criterion for the links in a stratification of $X$ and one for the polyhedral links of the underlying $|X|$, when it comes time to set up a bordism homology theory, he utilizes a different recognition criterion for the polyhedral links. This different criterion has the interesting feature of depending upon what we call ``second order link properties'' in that there are conditions not just on links but on the links of links (though Siegel does not phrase his conditions in this way). We generalize this construction, arriving at classes of polyhedral links that we call \emph{Siegel classes}. Although in some ways more complex in definition, Siegel classes provide   a more efficient (in a sense we make precise below) description of the classes of spaces arising in pseudomanifold bordism theories.  
 
Section \ref{S: Q} collects some questions left unanswered by our study.

\begin{remark}
While many of the results in this paper appear to be negative, in the sense that we have a number of theorems of the form  ``stratified bordism groups are the same as unstratified bordism groups,'' these results do have their utility. In many situations, it is more natural to consider the stratified spaces,  not their underlying topological spaces, to be the natural objects, and we provide some evidence below that it is easier to verify that one has a class of stratified pseudomanifold singularities than a class of (unstratified) pseudomanifold singularities. In fact,  we will see that (except for one low-dimensional situation) the criteria are the same except for a weakening of one of the conditions in the stratified case! Furthermore, even in equivalent situations, it is often useful to have a dictionary between them.  The fact that our dictionary does not consist entirely of bijections indicates that there may be yet some more interesting work to do in, for example, understanding the lack of injectivity of the assignment $\mc E\to \mc G_{\mc E}$ or in understanding more generally when different classes of links (perhaps not even in a class $\mc G$) yield the same classes of spaces for a bordism theory. 
\end{remark}

The following diagram represents something of a schematic to the various classes of spaces and bordism theories we will study.  Each arrow is  labeled with a reference to where the relevant connection can be found, either in this paper or in Akin's \cite{Ak75}. The dashed horizontal arrow reflects the fact that, while we consider stratified bordism \emph{groups} based on arbitrary IWS classes $\mc C$, we will only treat stratified bordism \emph{homology theories} for classes of the form $\mc C_{\mc E}$, i.e. those IWS classes determined by classes of stratified pseudomanifold singularities. Similarly, the dashed vertical arrow reflects that we will only show that IWS classes of the form $\mc C_{\mc E}$ yield pseudomanifold bordism classes. 
We will see that this diagram, in some sense, commutes: if we  begin with a classes of stratified pseudomanifold singularities $\mc E$  and follow the two paths of constructions, we obtain equivalent bordism theories.

\begin{equation*}
\resizebox{.95\hsize}{!}{
\begin{diagram}
\parbox{1.5in}{\begin{flushleft}classes of stratified pseudomanifold singularities $\mc E$\end{flushleft}}&\rTo^{\S \ref{S: sps}}&\parbox{1in}{\begin{flushleft}IWS classes $\mc C$\end{flushleft}}&\rDashto^{\S \ref{S: all abord}, \S \ref{S: theory}}&\parbox{1.5in}{\begin{flushleft}stratified bordism theories $\Omega_*^{\mc C}(\cdot)$\end{flushleft}}\\
\\
\dTo^{\S \ref{S: classes}}&&\dDashto^{\S \ref{S: F}}&&\dTo^{\mf s}_{\S \ref{S: all abord}, \S \ref{S: theory}}\\
\\
\parbox{1.5in}{\begin{flushleft}classes of pseudomanifold singularities $\mc G$\end{flushleft}}&\rTo^{\text{Akin} \cite[p. 354]{Ak75}}&\parbox{1.5in}{\begin{flushleft}pseudomanifold bordism classes $\mc F$\end{flushleft}}&\rTo^{\text{Akin} \cite[Prop. 7]{Ak75}}&\parbox{1.5in}{\begin{flushleft}(unstratified) bordism theories $\Omega_*^{|\mc F|}(\cdot)$\end{flushleft}}
\end{diagram} }
\end{equation*}

\section{Background definitions and working assumptions}\label{S: background}

In this section, we provide some background definitions to orient the reader to our language. We make no attempt to be comprehensive, referring the reader to \cite{GBF35} for more details concerning pseudomanifolds and stratifications.

We work throughout in the category of piecewise linear (PL) spaces. The symbol $\cong$ will always mean PL homeomorphism and all maps are PL maps. All dimension indices will correspond to topological dimensions.

\paragraph{PL pseudomanifolds and $\bd$-pseudomanifolds.}

A PL filtered space is a PL space $Y$ equipped with a family of closed PL subspaces
$$Y=Y^n\supseteq Y^{n-1}\supseteq \cdots \supseteq 
Y^0\supseteq Y^{-1}=\emptyset.$$ 
We let $c(Y)$ denote the open cone on $Y$ with filtration 
$(c(Y))^i=c(Y^{i-1})$ for $i\geq 0$ and $(c(Y))^{-1}=\emptyset$. The cone on the empty set is defined to be the cone vertex, which we will usually denote $\{v\}$. The suspension $SY$ is stratified analogously so that $(SY)^i=SY^{i-1}$ for $i\geq 0$,  $S\emptyset$ is the disjoint union of two points, and $(SY)^{-1}=\emptyset$. 

The definition of PL stratified pseudomanifold is given by induction on the
dimension. 

\begin{definition}\label{D: pseudomanifold}
A $0$-dimensional stratified pseudomanifold $X$ is a  discrete set of points 
with the trivial filtration $X=X^0\supseteq X^{-1}=\emptyset$.

An $n$-dimensional \emph{PL stratified  pseudomanifold}
$X$ is a filtered PL space of dimension $n$  
such that
\begin{enumerate}
\item $X-X^{n-1}$ is dense in $X$, and
\item for each point $x\in X^i-X^{i-1}$, there exists a neighborhood
$U$ of $x$ for which there is a  \emph{compact} $n-i-1$ dimensional 
PL stratified    pseudomanifold  $L$ and a  PL homeomorphism
\begin{equation*}
\phi: \R^i\times cL\to U
\end{equation*}
that takes $\R^i\times c(L^{j-1})$ onto $X^{i+j}\cap U$. A neighborhood $U$ with
this property is called {\it distinguished} and $L$ is called a {\it link} of
$x$.
\end{enumerate}
\end{definition}

The $X^i$ are called \emph{skeleta}. We write $X_i$ for $X^i-X^{i-1}$; this 
is a PL $i$-manifold that may be empty. We refer to the connected components of 
the various $X_i$ as  \emph{strata}. If a stratum  is a subset of $X_n=X^n-X^{n-1}$ it 
is called a \emph{regular stratum}; otherwise it is called a \emph{singular 
stratum}. The union of singular strata is $X^{n-1}$, which we also denote $\Sigma$ or $\Sigma_X$.  The \emph{depth} of a stratified pseudomanifold is the number of 
distinct non-empty skeleta it possesses minus one. Note that non-empty strata of codimension one are allowed, though we will explicitly forbid them after Section \ref{S: bordisms}.

We will show below in Lemma \ref{L: unique link} the well-known fact that the PL homeomorphism type of a  link of a point in a PL stratified pseudomanifold depends only on the stratum containing it. 

\begin{remark}\label{R: link link link}
Another point worth observing is that if $L$ is a link of a point in a PL stratified pseudomanifold $X$, i.e. if $x$ has a distinguished neighborhood $\R^i\times cL$, and if $z$ is a point in $L$ with its own distinguished neighborhood $\R^j\times c \ell$ in $L$ (which exists because $L$ is itself a stratified pseudomanifold), then $\ell$ is itself a link of a point in $X$, i.e. ``the link of a link is a link.'' This follows by observing that the open subset  $\R^i\times (cL-\{v\})\cong \R^{i+1}\times L$ of $X$ has an open subset $\R^{i+1}\times  \R^j\times c \ell\cong \R^{i+j+1}\times c\ell$, consistently stratified to be a distinguished neighborhood of the point in $X$ corresponding to $(\vec 0,\{w\})\in \R^{i+j+1}\times c\ell$ in the product, letting $w$ be the vertex of $c\ell$. 
\end{remark}

\begin{definition}
We say that a PL stratified pseudomanifold is a \emph{classical} PL stratified pseudomanifold if it possesses no strata of codimension one. 
\end{definition}

\begin{definition}\label{def boundary}
An $n$-dimensional PL
\emph{$\bd$-stratified pseudomanifold} is a pair $(X,B)$ together with a
filtration on $X$ such that
\begin{enumerate}
\item $X-B$, with the induced filtration $(X-B)^i=X^i-(B\cap X^i)$, is an $n$-dimensional PL stratified
pseudomanifold,
\item $B$, with the induced filtration $B^i=X^{i+1}\cap B$, is an $n-1$ dimensional PL stratified
pseudomanifold,
\item\label{I: collar} $B$ has an open  stratified collar neighborhood in $X$, that is there exists a
neighborhood $N$ of $B$ with a homeomorphism of filtered spaces $N\to
[0,1)\times B$ that takes
$B$ to $B\times \{0\}$; here $[0,1)$ is given the trivial filtration, so that the $j+1$ skeleton of $[0,1)\times B$ has the form $[0,1)\times B^j$, where $B^j$ is a skeleton of $B$. 
\end{enumerate}
$B$ is called the 
\emph{boundary} of $X$ and may be denoted by $\bd X$.  

We will generally abuse notation by referring to the ``$\bd$-stratified 
pseudomanifold $X$,'' leaving $B$ tacit.

The strata of a $\partial$-stratified pseudomanifold $X$ are the 
components of the spaces $X^i-X^{i-1}$; these may be PL $\bd$-manifolds.
A PL stratified pseudomanifold $X$ is a PL $\bd$-stratified pseudomanifold with $\bd X=\emptyset$.

We say that a PL $\bd$-stratified pseudomanifold is a \emph{classical} PL stratified pseudomanifold if it possesses no strata of codimension one. 
\end{definition}

\begin{remark}\label{R: class}
As we allow codimension one strata, it is critical to note that, for the same underlying PL space, there are subsets that might be considered as boundaries or that might be considered as unions of strata, depending upon the particular choice of stratification. Hence, caution is urged. For more details and examples, see \cite{GBF25, GBF35}.
\end{remark}

\begin{definition}\label{D: pseudo}
A PL space will be called simply a \emph{PL pseudomanifold} if it can be given a stratification making it a PL stratified pseudomanifold. Given a PL stratified pseudomanifold  $X$, we use the notation $|X|$ to refer to the underlying PL pseudomanifold  without its stratification\footnote{This is similar to the PL notation in which one uses $K$ to denote a simplicial complex and $|K|$ its underlying space. From here on we will be consistent with this notation in that $|X|$ for us will be an underlying PL space with no fixed triangulation or stratification. However, when we write simply $X$, we intend to imply a stratification with no particular choice of triangulation unless otherwise noted.}. We sometimes abuse this notation, referring, for example, to ``the pseudomanifold $|X|$'' even when we have no particular starting stratification in mind.

Similarly, a PL space will be called a \emph{PL $\bd$-pseudomanifold} if it can be given a stratification making it a PL $\bd$-stratified pseudomanifold. 
In this  case, however, we always assume that the underlying space of the boundary is part of the given information, even though we do not include it in the notation. When necessary, we use the  notation  $|\bd X|$ for this tacitly-given boundary. 
 So,  a ``PL $\bd$-pseudomanifold $|X|$'' is a PL space that can be given a stratification making it a PL stratified $\bd$-pseudomanifold $X$ such that the underlying space of $\bd X$ agrees with the tacitly chosen subspace $|\bd X|$ of $|X|$.   
If we did not fix the boundary of the underlying space, some ambiguity would result   due to the previous observation in Remark \ref{R: class} that the boundary of a stratified $\bd$-pseudomanifold depends on the stratification and not just the underlying PL space.

We call a PL pseudomanifold or $\bd$-pseudomanifold a  \emph{classical} PL pseudomanifold  or $\bd$-pseudomanifold if it can be stratified without codimension one strata.

\end{definition}

Stratified pseudomanifolds and $\bd$-stratified pseudomanifolds are the setting for intersection homology theory. We will not review the basic definitions here; instead we refer the reader to the various expository sources such as \cite{GBF35, Bo, KirWoo, BaIH}.

\paragraph{Intrinsic stratifications.}
Every PL pseudomanifold possesses an intrinsic stratification\footnote{It might perhaps be more correct to say that ``$|X|$ possesses a filtration that gives it the structure of a PL stratified pseudomanifold,'' but we'll often use the words ``filtration'' and ``stratification'' interchangeably, as each determines the other for all cases we will consider.} as a PL stratified pseudomanifold. This is defined using equivalence classes in which two points $x_1,x_2\in |X|$ are equivalent if they possess respective neighborhoods $N_1, N_2$ such that $(N_1,x_1)\cong (N_2,x_2)$. These equivalence classes are unions of manifolds in $|X|$, and the $i$-skeleton of the intrinsic stratification is the union of all equivalence classes of dimension $\leq i$. 
Given a PL stratified pseudomanifold $X$ (or a PL pseudomanifold $|X|$), we will let  $X^*$ denote $|X|$ with its intrinsic stratification. The existence of intrinsic stratifications is a classical result (see, e.g., \cite{Ak69}); a more recent treatment commensurate with our point of view can be found in \cite{GBF35}.  
The intrinsic stratification of a pseudomanifold coarsens every other stratification of $X$, meaning that given a PL stratified pseudomanifold $X$, every stratum of $X$ is contained within some stratum of $X^*$ (equivalently, every stratum of $X^*$ is a union of strata of $X$). 

The following lemma concerning intrinsic stratifications will be useful.

\begin{lemma}\label{L: int prod}
Let $X$ be a PL stratified pseudomanifold. Then, as stratified spaces, $((0,1)\times X)^*=(0,1)\times X^*$. 
\end{lemma}

The proof uses some elements of PL topology concerning polyhedral links and suspensions that will be reviewed more thoroughly below in Section \ref{S: susps}. Thus, we defer the proof to the end of that section.

\begin{corollary}\label{C: int prod}
For any non-negative integer $k$, $(\R^k\times X)^*=\R^k\times X^*$. 
\end{corollary}
\begin{proof}
This follows by induction from Lemma \ref{L: int prod}.
\end{proof}

Since it will not be needed below, we avoid discussing intrinsic stratifications of $\bd$-pseudomanifolds, which involves extra technicalities due to the boundaries. 

\paragraph{Orientations.}

An $n$-dimensional PL $\bd$-stratified manifold is called \emph{orientable} (respectively, \emph{oriented}) if its regular strata are orientable (respectively, \emph{oriented}).

Let $\Sigma$ denote the singular set $X^{n-1}$ of the $n$-dimensional stratified pseudomanifold $X=X^n$, and let $\Sigma^* \subset X^*$ denote the singular set of $X^*$. Since $X^*$ is stratified more coarsely than $X$, we have $X-\Sigma\subset X^*-\Sigma^*$. If $X^*$ is oriented, meaning that an orientation is chosen for the manifold $X^*-\Sigma^*$, then the orientation restricts to an orientation of $X-\Sigma$. We say that this gives $X$ an orientation compatible with the orientation of $X^*$, or simply that $X$ and $X^*$ are \emph{compatibly oriented}.

More generally, if $X$ and $X'$ are two oriented stratified pseudomanifolds with $|X|=|X'|$, we will say that $X$ and $X'$ are compatibly oriented if there is an orientation of $X^*$ that restricts to both the given orientations of $X$ and $X'$. 

\begin{lemma}\label{L: orient}
If $X$ is a classical PL stratified pseudomanifold, then for any orientation $\mc O$ of $X$, there is a unique orientation $\mc O^*$ of $X^*$ such that $\mc O$ is compatible with $\mc O^*$. 
\end{lemma}
\begin{remark}
The lemma is not true if $X$ possesses codimension one strata. For example, let $X$ be the real line $X=\R$ filtered as $\{0\}\subset \R$, and let $X^*$ be $\R$ with the trivial stratification. If we orient $X$ with the orientation $\mc O$ such that the orientation of each ray points away from $0$, then clearly no orientation of $X^*$ restricts to $\mc O$. 
\end{remark}

\begin{remark}\label{R: induced o}
It follows from Lemma \ref{L: orient} that if $X$ and $X'$ are two classical PL pseudomanifold stratifications of the same underlying space  and $X$ is given an orientation $\mc O$, then there is induced a unique compatible orientation $\mc O'$ on $X'$ as the restriction of the unique extension of $\mc O$ to $X^*$. Conversely, we see that beginning with $\mc O'$  on $X'$, applying Lemma \ref{L: orient} with $X'$ in the role of $X$, and then restricting to $X$ must recover  $\mc O$. Thus, in this situation, there is a bijection between orientations on $X$ and $X'$, so given an orientation of $X$ there exists a unique compatible orientation for $X'$. Via these compatible orientations, we can thus consider the underlying pseudomanifold $|X|$ to have a well-defined orientation. We also remark that this discussion extends to $\bd$-pseudomanifolds since the orientation of a PL $\bd$-stratified pseudomanifold is determined by the orientation on $X-\bd X$. 
\end{remark}

\begin{proof}[Proof of Lemma \ref{L: orient}]
The orientation $\mc O$ on $X$ is defined on $X-\Sigma$, where, by assumption, $\Sigma$ has codimension at least $2$ in $X$. Furthermore, $X-\Sigma$ is a submanifold of the manifold $X^*-\Sigma^*$, and since $\Sigma$ has codimension at least $2$, the complement of  $X-\Sigma$ in $X^*-\Sigma^*$  must have codimension at least $2$ in $X^*-\Sigma^*$. 

Recall that we can think of an orientation of an $n$-manifold as an isomorphism between the constant sheaf $\mc Z$ with stalks $\Z$ and the sheaf determined by the presheaf $U\to H_n(U,U-x)$, which we shall denote $\ms H$. The assumption is that we have such an isomorphism  on $X-\Sigma$. Since the complement of $X-\Sigma$ in $X^*-\Sigma^*$ has codimension at least  $2$ and since $X-\Sigma$ is dense in $X^*-\Sigma^*$ (since it's dense in $X$), the isomorphism of sheaves extends uniquely over $X^*-\Sigma^*$ by \cite[Lemma V.4.11.a]{Bo}. This provides the necessary orientation $\mc O^*$ on $X^*$. 
\end{proof}

\section{Some PL topology}\label{S: susps}

In this section, we recall some needed basic results from PL topology. The standard  references for PL topology include \cite{RS, HUD, Z, STALL}. We will also refer often to \cite{Ak69}. 

Every point $x$ in a PL space $|X|$ has a neighborhood $|N|$ consisting of a cone on a compact PL space $|\Lk(x)|$. The space $|\Lk(x)|$ is called the polyhedral link\footnote{N.B. We will always refer to this space as the ``polyhedral link.'' The word ``link'' by itself will always refer to the  link $L$ of a stratum in a PL $\bd$-stratified pseudomanifold as in Definition \ref{D: pseudomanifold}.} of $x$. The polyhedral link is defined uniquely up to PL homeomorphism; see \cite[Lemma 2.19]{RS} and the discussion preceding it or \cite[Corollary 1.15]{HUD}. If $|X|$ is a PL $n$-manifold, then $|\Lk(x)|\cong|S^{n-1}|$ \cite[Corollary 1.16]{HUD}.  If $|S^kX|$ denotes\footnote{For situations in which we are interested in the underlying space of a construction such as a suspension, cone, or product,  we will put just one set of bars on the outside of the expression, e.g. $|SX|$, even if we begin with an unstratified space $|X|$. There should be  no ambiguity in the resulting unstratified  space. } the $k$-fold suspension of the compact PL space $|X|$, then $|S^rX|\cong |S^kY|$ for $r\leq k$ implies that $|X|\cong |S^{k-r}Y|$; see \cite[Lemma 9]{Ak69}. In particular, any compact PL space that suspends to a sphere is a sphere. 

If $x\in |X|$ has polyhedral link $|\Lk(x)|$ and $y\in |Y|$ has polyhedral link $|\Lk(y)|$, then the link of $(x,y)$ in $|X\times Y|$ is PL homeomorphic to the join $|\Lk(x)*\Lk(y)|$; see \cite[Exercise 2.24(3)]{RS} or the argument on  \cite[page 419]{Ak69}. In particular, this implies that if $X$ is a compact PL space, then $|\R\times cX|\cong |c(SX)|$ and  $|I\times \bar cX|\cong |\bar c(SX)|$, where $I$ is the closed interval and $\bar c(X)$ denotes the closed cone.

The facts of the preceding paragraph imply the well-known statement that the links of  PL stratified pseudomanifold are determined uniquely by their strata: 

\begin{lemma}\label{L: unique link}
Let $X$ be a PL stratified pseudomanifold, and let $\mc S$ be a stratum of $X$. Then any two links of any two  points in $\mc S$ are PL homeomorphic. 
\end{lemma}

\begin{proof}
If $X$ is a PL stratified pseudomanifold and $x\in X$ lives in an $i$-dimensional stratum and has a neighborhood PL homeomorphic to $|\R^i\times cL|$, then the polyhedral link of $x$ is PL homeomorphic to $|S^{i-1}*L|\cong |S^iL|$ (letting $|S^{-1}*L|=|L|$). If $L'$ were another possible link for $x$ in $X$, then we would have $|S^{i}L|\cong |S^{i}L'|$, but then $|L|\cong |L'|$. 

It is also true that the links of any two points in the same stratum are PL homeomorphic. Since strata are connected, it suffices to show that   the set of points in a stratum $\mc S$ with links homeomorphic to the link at a given point $z\in \mc S$ is both open and closed in $\mc S$. So let $|L|$ be the link of $z$, let $A$ be the set of points in $\mc S$ with link PL homeomorphic to $|L|$, 
and suppose $x\in A$. Then $x$ has a neighborhood in $X$ that is PL stratified homeomorphic to $\R^i\times \ell$, where $|\ell|\cong |L|$ and  where $|\R^i\times \{v\}|$ is taken by the homeomorphism to a neighborhood of $x$ in $\mc S$. Clearly any point in this neighborhood of $x$ in $\mc S$ also has a neighborhood in $X$ that is PL stratified homeomorphic to $|\R^i\times \ell|$. Therefore, $A$ is open in $\mc S$. Now, suppose $x\in \bar A$, the closure of $A$ in $\mc S$. Then $x$ again has a neighborhood PL stratified homeomorphic to $\R^i\times \ell$ for some $\ell$. But since $x$ is in the closure of $A$, there is a point $y\in A$ that is in the image of $|\R^i\times \{v\}|$ under the stratified homeomorphism. Hence the link of $y$ is both $\ell$ and PL homeomorphic to $|L|$, so $|\ell|\cong |L|$, and $x\in A$. Thus $A$ is closed and open and so must be all of $\mc S$.  
\end{proof}

The first part of the following lemma is utilized by Siegel in \cite{Si83}. 

\begin{lemma}\label{L: int link}
Let $|X|$ be a pseudomanifold and $x\in |X|$. Then the polyhedral link $|\Lk(x)|$ has the form $|S^j\ell|$ for a unique compact pseudomanifold $|\ell|$ that cannot be written as a suspension of a compact PL space. The pseudomanifold $|\ell|$ is the link of $x$ in $X^*$, the intrinsic stratification of $|X|$.
\end{lemma}
\begin{proof}
Let $|\ell|$ be the link of $x$ in $X^*$, and suppose $x$ is contained in an $i$-dimensional stratum of $X^*$. Then $|\ell|$ is a pseudomanifold and $x$ has a neighborhood in $X^*$ stratified homeomorphic to $\R^i\times c\ell$. The polyhedral link $\Lk(x)$ of $x$ is therefore PL homeomorphic to $|S^{i}\ell|$. We first claim that $\ell$ cannot itself be a suspension. If it were, then $|\ell|\cong |S\ell'|$ for some compact $|\ell'|$. But then $x$ has a neighborhood PL homeomorphic to 
$$|\R^i\times c\ell|\cong |\R^i\times c(S\ell')|\cong |\R^i\times \R^1\times c\ell'|\cong |\R^{i+1}\times c\ell'|.$$
Then if $w$ is the cone vertex of $|c\ell'|$, all the points in $|\R^{i+1}\times \{w\}|$, including $x$, have homeomorphic neighborhoods, contradicting that $x$ is contained in an $i$-dimensional stratum of $X^*$. Thus $|\ell|$ is not a suspension. 

We also see that $|\ell|$ is uniquely determined by $x$: by uniqueness of polyhedral links, any other polyhedral link $|\Lk(x)'|$ of $x$ is PL homeomorphic to $|S^{i}\ell|$. So, if $|\Lk(x)'|\cong |S^k\ell'|$ for some $|\ell'|$, then either 

\begin{enumerate}
\item $k>i$ and $|\ell|\cong |S^{k-i}\ell'|$, which would be a contradiction of the last paragraph, or \item $k<i$ and $|\ell'|\cong |S^{i-k}\ell'|$, in which case $|\ell'|$ is a suspension, or
\item $k=i$, in which case $|\ell|\cong |\ell'|$.
\end{enumerate}
Therefore, $|\ell|$, the link of $x$ in $X^*$, is the unique compact PL space with the given properties. 
\end{proof}

\subsection[Links in $\R\times cX$]{Links in $\mathbf{\R\times cX}$}\label{S: susp strat}

If $|SX|$ is the suspension of a compact PL space $|X|$, we use the interval $I=[-1,1]$ as the suspension parameter  so that each point of $|SX|$ can be described by a pair $(t,x)$ with $t\in [-1,1]$ and $x\in |X|$. This description is not unique when $t=-1$ or $t=1$. For the purposes of $PL$ topology, it is not usually quite correct to think of a  suspension as a quotient, but this notation still makes sense if, for example, we think of $|X|$ as PL embedded in $\{0\}\times \R^K\subset \R^{K+1}$ for some $K$ and then of $|SX|$ in $\R^{K+1}$ as a join with the points $(1,\vec 0)$ and $(-1,\vec 0)$. 

Similarly, we can form closed cones $|\bar cX|$ on compact PL spaces $|X|$, using the parameter $[0,1]$, but in this case letting the class of the pairs $(0,x)$ stand for the cone vertex, which we often denote by $v$. The notation for an open cone will be $cX$, in which case the parameter is chosen from $[0,1)$. Note that $|cX|$ is a PL space as an open subset of $|\bar cX|$. 
The exception to our parameter rule for cones will be when we want to think of $|SX|$ as the union of two cones, in which case we shall denote them by  $|\bar c_+X|$ and $|\bar c_-X|$ with respective cone parameters in $[0,1]$ and $[-1,0]$. 

In what follows, we will need to consider the following construction. If we begin with a compact stratified space $X$, consider the stratified space $\R\times cX$. The stratum $\mc S$ of $X$ contributes a stratum of the form $\R\times (0,1)\times \mc S$ to $\R\times  cX$, where the middle factor is the cone parameter. There is also a stratum $\R\times \{v\}$, where $v$ is the cone vertex. Since the polyhedral link of $v$ in $|cX|$ is $|X|$ and the polyhedral link of $0$ in $|\R|$ is two points, the polyhedral link of $(0,v)$ is the suspension $|SX|$, so $(0,v)$ has a neighborhood $|N|\cong |c(SX)|$. Notice that if we give $N$ the stratification inherited from $\R\times  cX$, i.e. $N^i= N\cap (\R\times cX)^i$, then $N$ and $c(SX)$ are \emph{not} homeomorphic as stratified spaces if we use the standard cone and suspension stratification for the latter. In fact,  $c(SX)$ has a $0$-dimensional stratum at its cone vertex, while $N$ with the inherited stratification has no $0$-dimensional strata. However, the following is true: Let $N_+=N\cap ((0,\infty)\times cX)$ so that $N_+$ is the portion of $N$ ``above'' $\{0\}\times cX$. Then the stratification that $N_+$ inherits as a subspace of $\R\times cX$ does agree with the stratification of the corresponding subspace of $c(SX)$, which is $c(SX)-c(\bar c_-X)=c(SX-\bar c_-X)-\{w\}$, where we recall that $\bar c_-X$ denotes the closed ``southern'' cone of $SX$ and where $w$ here denotes the vertex of the cone $c(SX)$.  This equivalence of stratifications can be seen by observing that 

\begin{enumerate}
\item $N^+$ has a  stratum PL homeomorphic to $(0,1)\times \{v\}$ corresponding to the cone (with its vertex removed) of the north pole of $SX$

\item for each stratum $\mc S$ of $X$, the stratum $N^+\cap (\R\times (0,1)\times \mc S)$ corresponds to the stratum $(0,1)\times (0,1)\times \mc S$ that arises in  $c(SX)-c(\bar c_-X)$. 
\end{enumerate}

\subsection{Proof of  Lemma \ref{L: int prod}.}

We end this section with the deferred proof of  Lemma \ref{L: int prod}, which stated that, for a  PL stratified pseudomanifold $X$, we have $((0,1)\times X)^*=(0,1)\times X^*$ as stratified spaces.

\begin{proof}[Proof of Lemma \ref{L: int prod}]

Since the intrinsic stratification of a PL stratified pseudomanifold is the coarsest stratification, and since $((0,1)\times X)^*$ and $(0,1)\times X^*$ are both PL stratified pseudomanifolds, it follows that the former must be a coarsening of the latter. Suppose that $(0,1)\times X^*$ is a strictly finer stratification than $((0,1)\times X)^*$. This implies that there must be two points, say $(t,x)$ and $(s,y)$ that are in different strata of $(0,1)\times X^*$ (and so $x$ and $y$ are in different strata of $X^*$) but the same stratum of $((0,1)\times X)^*$. Since $(t,x)$ and $(s,y)$ are in the same stratum of $((0,1)\times X)^*$, they have PL homeomorphic star neighborhoods. Owing to the product structure on 
$(0,1)\times X^*$, the point $(t,x)$ has a neighborhood of the form $|c(S\Lk(x))|$, where $|\Lk(x)|$ is the polyhedral link of $x$ in $X$ (ignoring stratification). Similarly, $(s,y)$ has a neighborhood of the form $|c(S\Lk(y)|$. Thus $|S\Lk(x)|$ and $|S\Lk(y)|$ are the respective polyhedral links of $(t,x)$ and $(s,y)$ in $|(0,1)\times X|$, and by uniqueness of polyhedral links, we must have $|S\Lk(x)|\cong |S\Lk(y)|$. But then again by PL topology, we must have $|\Lk(x)|\cong |\Lk(y)|$, but this implies that $x$ and $y$ must in fact have homeomorphic relative neighborhoods in $X$, a contradiction to the claim that they lie in different intrinsic strata of $|X|$. Thus $(0,1)\times X^*$ is not strictly finer than $((0,1)\times X)^*$, and since it cannot be strictly coarser, the two stratifications must agree. 
\end{proof}

\section{Bordisms}\label{S: bordisms}

In this section, our main goal is to construct stratified bordisms, first by showing that there is a stratified bordism between any two stratifications of the same pseudomanifold.  
For our construction, we will need to consider a new space $\mf SX$, which we call the \emph{half-intrinsic suspension} of the PL stratified pseudomanifold $X$.  It is
built as follows: The underlying space $|\mf SX|$ of $\mf SX$ will be the suspension $|SX|$.   The stratification of $\mf SX$ will be determined as follows:
\begin{enumerate}
\item the ``north pole'' suspension vertex at suspension parameter $1$ will be a $0$-dimensional stratum of $\mf SX$,

\item if $\mc S$ is a \emph{singular} stratum of $X$, then $\{0\}\times \mc S$ and $(0,1)\times \mc S$ will be strata of $\mf SX$,

\item the restriction of the stratification of $\mf SX$ to the open ``south cone'' $|c_-X|$ will be the intrinsic stratification $(cX)^*$,

\item the regular strata of $\mf SX$ will be the connected components of the union of the regular strata of the south cone (with the intrinsic stratification) and the subsets $(-1,1)\times \mc R$ as $\mc R$ ranges over the regular strata of $X$.  
\end{enumerate}

So, in particular, the open north and south cones of $\mf SX$ inherit stratifications that make them  stratified homeomorphic to $cX$ and $(cX)^*$, respectively. The closed subset $\{0\}\times X$ with its inherited stratification is stratified homeomorphic to $X$; however, a set $\{0\}\times \mc S$, for $\mc S$ a stratum of $X$, is a stratum of $\mf SX$ if and only if $\mc S$ is a singular stratum of $X$. The regular strata of $\mf SX$ span both the north and south closed cones; in the open northern cone, they restrict to the form $(0,1)\times \mc R$, but as they run into the southern cone they may merge with larger regular strata of the form $(-1,0)\times\mc R^*$, where $\mc R^*$ denotes a regular stratum of $X^*$. These are the only strata that intersect both the northern and southern open cones. Note that the south pole of $|SX|$ might or might not be a stratum of $\mf SX$, depending on the particulars of the intrinsic stratification $(cX)^*$. 

\begin{figure}[h!]\label{F: 1}
  \centering
\scalebox{.5}{\includegraphics{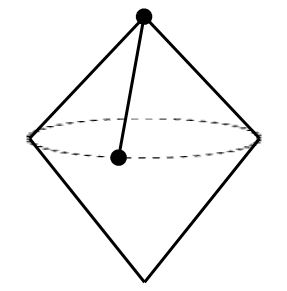}}
\caption{The half-intrinsic suspension of the stratified pseudomanifold $S^1\supset \{pt\}$. The $0$-dimensional strata are the north pole and $\{0\}\times \{pt\}$. There is a one-dimensional stratum $(0,1)\times \{pt\}$, and the rest of the space is one $2$-dimensional regular stratum. Notice that the south pole is not a stratum.}
\end{figure}

\begin{lemma}\label{L: half in}
If $X$ is a compact PL stratified pseudomanifold, then so is $\mf SX$. 
\end{lemma}

\begin{proof}
We begin by noticing that it follows from the preceding discussion that  all the strata of $\mf SX$ must be PL manifolds. It is also clear that the union of the regular strata of $\mf SX$ is dense in $\mf SX$; in fact, as $X$ is a stratified pseudomanifold, the union of the sets $(-1,1)\times \mc R$ over the regular strata $\mc R$ of $X$ must be dense in $\mf SX$. 

The described stratification is consistent with a PL filtration of $\mf SX$ whose $i$-skeleton is the union of the strata of $\mf SX$ of dimension $\leq i$; this can be seen by noticing that the closure of every stratum is the union of strata of lower dimension. This is clear for any regular strata and also for all strata within the closed northern cone because for any singular stratum $\mc S$ of $X$, the closure of $\{0\}\times \mc S$ will be $\{0\}\times \bar{\mc S}$ and the closure of $(0,1)\times \mc S$ will be the closed cone on $\bar {\mc S}$. If $\mc T$ is a stratum in the open southern cone, then since the southern cone is stratified as a pseudomanifold, the closure of $\mc T$ in $(cX)^*$ must be a union of lower dimensional strata. To determine which points of $\{0\}\times X$ lie in the closure of $\mc T$, we notice that the restriction of $\mc T$ to $|(-1,0)\times X|$ must be a stratum of the product stratification $(-1,0)\times X^*$ by Lemma \ref{L: int prod} and using that open subsets of intrinsically stratified space are intrinsically stratified (since intrinsic stratifications are determined by local conditions). Thus on $|(-1,0)\times X|$, $\mc T$ restricts to a stratum of the form $(-1,0)\times \mc U$, where $\mc U$ is a stratum of $X^*$, and  the intersection of the closure of $\mc T$ with $|(-1,0)\times X|$ must be the closure of $\{0\}\times \mc U$, which again is a union of lower-dimensional strata of $X$. 

It remains to verify that the links of $\mf SX$ are themselves stratified pseudomanifolds. This is immediate for points in the open northern and southern cones, where the stratifications reduce to the stratifications of the known stratified pseudomanifolds $cX$ and $(cX)^*$. It is also immediate at all points of the regular strata. Therefore, we must look at the links of points in the strata of the form $\{0\}\times \mc S$ for $\mc S$ a singular stratum of $X$. We claim that all such links have the form $\mf SL$, where $L$ is the link of the corresponding point in $X$. Since $L$ must have depth less than that of $X$, we can reduce the argument to a proof by induction with the induction assumption being that the lemma is true for all spaces of depth less than that of $X$. The base case consists of the situation when the depth of a compact PL stratified pseudomanifold $Z$ is $0$. In this case, $Z$ is a manifold, and $\mf SZ$ is either the suspension of $Z$ with its usual stratification (if $Z$ is not a sphere), or (if $Z$ is a sphere) a sphere with the north pole as the lone singular stratum. Thus, assuming the claim that the links of points in $\{0\}\times \mc S$ have the given form, the proof of the lemma will be complete by a strong induction.

\begin{figure}[h!]
  \centering
\scalebox{.5}{\includegraphics{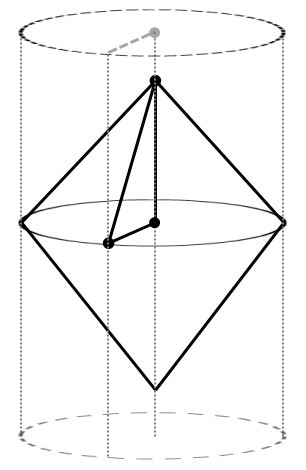}}
\caption{This figure is a piece of $\mf SX$ surrounding a point $z$ of $\{0\}\times X$ that corresponds to a $0$-dimensional stratum of $X$ and whose link in $X$ is the stratified pseudomanifold $S^1\supset \{pt\}$. The suspension of the circle in the middle of the picture is the boundary of a (closed) neighborhood of $z$ in $\mf SX$. The suspension itself can be identified with the link of $z$ in $\mf SX$. This link is the same stratified pseudomanifold as in Figure \ref{F: 1}. Notice the different stratifications of $\{t\}\times |S^1|$ in the northern and southern hemispheres of $\mf SX$.}
\end{figure}

So consider $z=(0,x)\in \mf SX$ such that $x\in \mc S$ with $\mc S$ a singular stratum of $X$. Since $X$ is a PL stratified pseudomanifold, $x$ has a distinguished neighborhood $N$ in $X$ stratified homeomorphic to
 $\R^i\times cL$, assuming $\dim(\mc S)=i$ and with $L$ a compact stratified pseudomanifold. To simplify the notation in the following discussion, we can denote $\{0\}\times X\subset \mf SX$ simply by $X$ and let the same letter $\mc S$ denote the corresponding singular stratum $\{0\}\times \mc S\subset \mf SX$. we will use the homeomorphism of the distinguished neighborhood to identify the distinguished neighborhood of $x$ in $X$ as $\R^i\times cL$, identifying $L$ with the image of $\{0\}\times (\{1/2\}\times L$. 
 
In $\mf SX$, the $i$-dimensional stratum $\mc S$ contains $z$. Furthermore, $|(-1,1)\times N|\subset |\mf SX|$ is a neighborhood of  $z$. We have 
 \begin{align*}
|(-1,1)\times N|&\cong |N\times (-1,1)|\\
&\cong |\R^i\times (-1,1)\times  cL|\\
&\cong |\R^i\times c(SL)|.
\end{align*} 
 
Here, if $w$ denotes the cone vertex of $|c(SL)|$, then $|\R^i\times \{w\}|$ still corresponds to a neighborhood of $z$ in its stratum $\mc S$ in $\mf SX$. We claim that if we stratify $|SL|$ as $\mf SL$, which is a PL stratified pseudomanifold by the induction hypothesis,  then the homeomorphism $|(-1,1)\times N|\cong |\R^i\times c(SL)|$ will give a stratified homeomorphism $(-1,1)\times N\cong \R^i\times c(\mf SL)$, where $(-1,1)\times N$ is given its  stratification as a subspace of $\mf SX$. This will show that $\mf SL$ is a link of $z$ in $\mf SX$. From the construction, we can assume that $|SL|\cap |X|=|L|$, and, by assumption, $L$ is stratified consistently with being a subspace of $X$. Furthermore, by construction, $|\{0\}\times N|\subset |(-1,1)\times N|$, is stratified in $\mf S X$ consistently with being identified as $N\subset X$.

Next, we bring in our discussion from Section \ref{S: susp strat} concerning the stratification of spaces of the form $\R\times cX$, though here we consider, equivalently, $(-1,1)\times cL$. Recall that the northern cone of $\mf SX$ is stratified as the cone on $X$. Therefore, in the stratification coming from $\mf SX$, $|(0,1)\times N|\subset |(-1,1)\times N|$ has the product stratification   $(0,1)\times N$. 
 The  discussion from Section \ref{S: susp strat} then demonstrates that  the stratification  in $(0,1)\times N \cong (0,1)\times (\R^i\times cL)$ of $|(0,1)\times (\{\vec 0\}\times cL)|\cong |(0,1)\times cL|$ is consistent with the the stratification  $c(SL)- c(\bar c_-(L))$. In other words, as a subspace of $\mf S X$, $|(0,1)\times N|$ is stratified consistently with having $|(-1,1)\times N|$ stratified as $\R^i\times c(\mf S L)$.  
 
On the other hand, the intersection of $|(-1,1)\times N|$ with the open southern cone of $\mf SX$ similarly has underlying space homeomorphic to $|\R^i\times (c(SL)- c(\bar c_+L))|$. Since the open southern cone of $\mf SX$ is intrinsically stratified, so will be its open subspace $|(-1,0)\times N|$ in the stratification inherited from $\mf SX$. But, $|\R^i\times (c(SL)- c(\bar c_+L))|\cong |\R^i\times (0,1)\times c_-(L)|$, where the middle factor represents the cone parameter in $|c(SL)- c(\bar c_+L)|$.
By  Corollary \ref{C: int prod}, the intrinsic stratification is $\R^i\times (0,1)\times (c_-(L))^*$. 
So, as a subspace of $\mf S X$, $|(-1,0)\times N|$ is stratified consistently with having $|(-1,1)\times N|$ stratified as $\R^i\times c(\mf S L)$.  

So, we have seen that if we restrict the neighborhood $|(-1,1)\times N|$ of $z$ to the top or bottom open cone of $\mf SX$ or to the ``center'' $X\subset \mf SX$ and consider the stratification induced from $\mf SX$, then these stratification are compatible with having $|(-1,1)\times N|$ stratified a $\R^i\times c(\mf SL)$. We also know from the construction that the singular strata of $N$ will be singular strata of $|\{0\}\times N|$ in  the stratification of $|(-1,1)\times N|$ inherited from $\mf SX$. Finally, we observe that the union of the regular strata of $|(-1,1)\times N|$ in the stratification inherited from $\mf SX$ must be the unions of the products $(-1,1)\times \mc R$, as $\mc R$ ranges over the regular strata of $N$, with the regular strata of the intrinsic stratification of $(-1,0)\times N$. But the regular strata of $N$ have the form $\R^i\times (c(R)-\{v\})$, where $R$ is a regular stratum of $L$. 
This implies that  that the intersection of $|(-1,1)\times N|$ with the regular strata of $\mf SX$ must also be the regular strata of $\R^i\times c(\mf SL)$.

We conclude that $z\in \mf SX$ has a neighborhood stratified homeomorphic to $\R^i\times c(\mf SL)$, so  $\mf SL$ is the link of $z$ in $\mf SX$. 
\end{proof}

\paragraph{Bordism.}
We can now use Lemma \ref{L: half in} to construct bordisms.

\begin{definition} We say that $Y$ is an \emph{oriented stratified pseudomanifold bordism} between the compact oriented PL stratified pseudomanifolds $X$ and $Z$ if $Y$ is a compact oriented PL stratified pseudomanifold with $\bd Y=X\amalg -Z$.
\end{definition}

\begin{proposition}\label{P: bordism}
 If $X$ is a compact PL stratified pseudomanifold and if $X$ is  compatibly oriented with $X^*$, then there is an oriented stratified pseudomanifold bordism from $X$ to $X^*$.  If $X$ is a classical PL stratified pseudomanifold, the bordism $Y$ can be chosen to be a classical PL $\bd$-stratified pseudomanifold. The underlying space of the bordism can be taken to be $|I\times X|$.
\end{proposition}

\begin{proof}
Consider $\mf SX$, which is a PL stratified pseudomanifold by Lemma \ref{L: half in}. The top and bottom open cones of $\mf SX$ are stratified homeomorphic to $cX$ and $(cX)^*$, respectively. Therefore, if we remove the subset of $\mf SX$ corresponding to the suspension parameters $[-1,-1/2)$ and $(1/2,1]$, what is left will be a $\bd$-stratified pseudomanifold $Y$. We note that $Y$ will have  stratified collared boundaries, as required: 

\begin{itemize}
\item The subspace $|(0,1/2]\times X|\subset |\mf SX|$ is stratified as $(0,1/2]\times X$ by construction.

\item The open subspace $|(-1,0)\times X|\subset |\mf SX|$ has the intrinsic stratification, which is $(-1,0)\times X^*$ by Lemma \ref{L: int prod}. So when we remove $(-1,-1/2)\times X^*$, what remains is a collaring of $\{-1/2\}\times X^*$.
\end{itemize}
Notice that if  $X$ has no codimension one strata, then neither will $Y$ by the construction of $X^*$.  This is our desired bordism. 

It only remains to consider orientations. The regular strata of $Y$ have the form of the connected components of the unions of  the subsets $[-1/2,1/2]\times \mc R$, as $\mc R$ ranges over the regular strata of $X$, with the regular strata of $([-1/2,0)\times X)^*$.  By Lemma \ref{L: int prod}, these latter strata will have the form $[-1/2,0)\times \mc U$, as $\mc U$ ranges over the regular strata of $X^*$. 
Suppose we give  subsets of the real line their standard orientations in the increasing direction and we stratify the $\mc U$ by the given orientation on $X^*$. Then we have the product  orientation on the manifold subspace $[-1/2,1/2]\times \prod \mc U$, where the product is taken over all the regular strata $\mc U$ of $X^*$. The compatibility assumption between the orientations of $X$ and $X^*$ assures us that the restriction of this orientation to $[0,1/2]\times \prod \mc R$ (where the product is taken over the regular strata of $X$) agrees with its orientation induced by the orientation of $X$. Hence, we obtain an orientation on the regular strata of $Y$ that is consistent with the product orientation  of the standard interval with $X$ on the northern portion of $Y$ and with the product orientation  of the standard interval with $X^*$ on the southern portion of $Y$. This yields the desired orientations on the boundaries of $Y$. 
\end{proof}

\begin{corollary}\label{C: bordism}
If $X, X'$ are any two compact PL stratified pseudomanifolds with $|X|\cong|X'|$ and if there is an orientation of $X^*$ that is simultaneously compatible with given orientations on $X$ and $X'$, then there is an oriented stratified pseudomanifold bordism between $X$ and $X'$. If  $X$ and $X'$ are classical PL stratified pseudomanifolds, then in each of the above situations, $Y$ can be chosen to be a classical $\bd$-stratified pseudomanifold. Furthermore, the underlying space of each bordism can be taken to be $|I\times X|$.
\end{corollary}
\begin{proof}
By the proposition, there are such bordisms $Y$ between $X$ and $X^*$ and $Y'$ between $X'$ and $X^*$. So to obtain the desired bordisms, we glue the boundary component $-X^*$ of  $Y$ to the boundary component $X^*$ of  $-Y'$ to obtain a $\bd$-stratified pseudomanifold $W$ with $|W|\cong |Y \cup_{X^*} -Y'|$. Then $W$ is the desired stratified pseudomanifold bordism from $X$ to $X'$. 
\end{proof}

\begin{corollary}\label{C: given bord}
Suppose that $X$, $Z$ are two  compact orientable PL stratified pseudomanifolds and that there exist compatibly oriented PL stratified pseudomanifolds $X'$ and $Z'$ such that $|X|\cong|X'|$,  $|Z|\cong |Z'|$, and there exists  a PL $\bd$-stratified pseudomanifold $Y'$ such that $\bd Y'\cong X'\amalg -Z'$. Then there exists a stratification of $|Y'|$ as a   PL $\bd$-stratified pseudomanifold $Y$ such that $\bd Y\cong X\amalg -Z$. If none of $X, X', Z,Z', Y'$ have a codimension one stratum, then $Y$ can be chosen to have no codimension one strata. 
\end{corollary}

\begin{remark}\label{R: int bord}
The basic idea of Corollary \ref{C: given bord} is that if we know that we have a pseudomanifold bordism $|Y|$ without a specific stratification between the underlying spaces  $|X|$ and $|Z|$, then we would like to know that it is possible to stratify $|Y|$ to be compatible with the given stratifications of $X$ and $Z$. The reason we have phrased the corollary as above is that even when considering $|Y|$ as a $\bd$-pseudomanifold without a stratification,  the definition of a $\bd$-pseudomanifold  nonetheless assumes that $|Y|$ \emph{can} be given the structure of a PL $\bd$-stratified pseudomanifold for \emph{some} stratification for which the underlying spaces of the  boundaries are  $|X|$ and $|Z|$. We call this arbitrary stratification $Y'$, and then this induces some stratification on its boundary, yielding $X'$ and $Z'$. 

\begin{proof}[Proof of Corollary \ref{C: given bord}]
By the proofs of Proposition \ref{C: bordism} and Corollary \ref{C: bordism}, there are oriented bordisms between  $X$ and  $X'$ and  between $Z$ and  $Z'$ whose underlying spaces are homeomorphic to $|[0,1]\times X|$ and $|[0,1]\times Z|$. Adjoining these bordisms to the boundary of $Y'$ with the proper orientations, we obtain a new $\bd$-stratified pseudomanifold  $Y$ whose underlying space is homeomorphic to $|Y'|$ but whose stratification now provides a stratified pseudomanifold bordism between $X$ and $-Z$. It also follows from the previous constructions that $Y$ will have no codimension one strata if none of $X, X', Z,Z', Y'$ do. 
\end{proof}

\begin{remark}
Proposition \ref{P: bordism} and Corollaries \ref{C: bordism} and \ref{C: given bord} admit evident unoriented versions by neglect of structure. 
\end{remark}

\end{remark}

\section{Bordism groups}\label{S: bordism groups}

In this section, we consider stratified and unstratified bordism groups of pseudomanifolds. For various technical reasons, codimension one strata are inconvenient\footnote{As one example, if we allow codimension one strata in our bordisms, then every oriented stratified pseudomanifold $X$ has a stratified pseudomanifold bordism to $-X$: simply choose the stratification $Y$ of $|Y|=|I\times X|$ so that if $\mc S$ is \emph{any} stratum of $X$, then $[-1,0)\times \mc S$, $\{0\}\times \mc S$, and $(0,1]\times \mc S$ are strata of $Y$. The regular strata of $Y$ are $[-1,0)\times \mc R$ and $(0,1]\times \mc R$, for $\mc R$ among the regular strata of $X$. By choosing the ``opposite'' orientations on $[-1,0)\times \mc R$ and $(0,1]\times \mc R$, we get $\bd Y=X\amalg X$. So $X$ and $-X$ are oriented bordant.}, and they do not arise in any of the pseudomanifold bordism theories previously studied, such as Witt spaces and IP spaces. 
We therefore make the blanket assumption for the remainder of the paper that codimension one strata are not allowed, i.e. all of our PL $\bd$-stratified pseudomanifolds from here on  will be classical and hence so will be their underlying spaces as PL $\bd$-pseudomanifolds. Given this assumption, we therefore will tend to omit the word ``classical'' unless we particularly wish to emphasize this point. 

\begin{remark}\label{R: classical boundary}
An additional advantage of working with classical PL pseudomanifolds that we will utilize in various places  is that their boundaries (in any stratification) are determined uniquely as the points whose polyhedral links are classical $\bd$-pseudomanifolds but not classical PL pseudomanifolds; the polyhedral links of boundary points will be closed cones on classical PL pseudomanifolds and so these polyhedral links have boundaries. Thus specifying a classical $\bd$-pseudomanifold specifies its boundary without having to make any choices such as those discussed in Definition \ref{D: pseudo}. 
\end{remark}

\subsection{Stratified and unstratified bordism groups of IWS classes}\label{S: all abord}

\begin{definition}
Let $\Psi$ denote the class of  compact classical PL $\bd$-stratified pseudomanifolds, and let 
 $|\Psi|$ denote the class of  compact classical PL $\bd$-pseudomanifolds. 
Similarly, let\footnote{Note: when we place ``$S$'' in front of the symbol for a category, then it indicates that we are considering the  version of the category with oriented spaces. This should not be confused with placing ``$S$'' in front of the symbol for a specific space, in which case it indicates suspension. These are both common uses for the symbol ``$S$,'' and we hope context will keep the meaning clear.}
 $S\Psi$ and $|S\Psi|$ denote the respective classes of \emph{oriented} compact classical PL $\bd$-stratified pseudomanifolds  and \emph{oriented} compact classical PL $\bd$-pseudomanifolds.  
\end{definition}

\begin{definition}\label{D: weak strat class}
We will say that a subclass $S\mc C\subset S\Psi$ is an \emph{oriented weak stratified bordism class} if it satisfies the following axioms: 
\begin{enumerate}
\item If $X\in S\mc C$, then $-X\in S\mc C$. 
\item If $X\in S\mc C$ and $X$ is stratified \emph{oriented} PL homeomorphic to $Y$,  then $Y\in S\mc C$.
\item If $X\in S\mc C$, then $\bd X\in S\mc C$, where $\bd X$ is given the induced orientation.
\item If $X\in S\mc C$ and $\bd X=\emptyset$, then $I\times X\in S\mc C$, using the product orientation. 
\item If $W,Y\in S\mc C$ with $\bd Y\cong X\amalg -Z$ and $\bd W=Z\amalg -V$, then $Y\cup_Z W\in S\mc C$ for any gluing of $W$ and $Y$ along $Z$ compatible with the orientations. 
\item $S\mc C$ contains an empty stratified pseudomanifold of  each dimension; we label each of these $\emptyset$. We let $-\emptyset=\emptyset$.
\end{enumerate}
\end{definition}

\begin{remark}\label{R: SWS}
Similarly, one may define a \emph{weak stratified bordism class} $\mc C\subset \Psi$ by removing orientation considerations from the axioms. 
If $S\mc C$ is an oriented weak stratified bordism class, we can obtain a weak stratified bordism class by forgetting the orientation information. Conversely, if $\mc C$ is a weak stratified bordism class, we can form the corresponding
oriented weak stratified bordism class $S\mc C$ whose objects are the orientable objects of $\mc C$ with each of their orientations. 
\end{remark}

The above-defined classes are called \emph{weak} stratified bordism classes because they reflect a weaker version of the requirements of Akin's bordism homology theories  \cite[page 349]{Ak75}, though there is also an additional specialization in that Akin deals with compact polyhedra in general and not just pseudomanifolds. Although our axioms are not sufficient to yield a bordism homology theory, a topic that we will discuss below in Section \ref{S: bordism homology}, they are sufficient to yield bordism groups, as we discuss now.
We will treat the oriented case, though the unoriented case is completely analogous. 

We first show that, if $S\mc C$ is an oriented weak stratified bordism class, then oriented stratified pseudomanifold bordism in $S\mc C$ is an equivalence relation. 
 Recall that we say that $Y$ is an oriented stratified pseudomanifold bordism between the compact oriented PL stratified pseudomanifolds $X$ and $Z$ if $Y$ is a compact oriented PL stratified $\bd$-pseudomanifold with $\bd Y\cong X\amalg -Z$. We will say that $X, Z\in S\mc C$ are oriented stratified bordant in $S\mc C$, denoted $X\sim_{S\mc C} Z$,  if there exists an oriented stratified pseudomanifold bordism $Y$ between $X$ and $Z$ such that $Y\in S\mc C$. If the class $S\mc C$ is understood, we may omit the phrase ``in $S\mc C$'' and simply write $X\sim Z$ for the relation. 
Notice that for $X$ and $Z$ to be bordant in $S\mc C$, it is necessary that $\bd (X\amalg -Z)=\bd \bd Y=\emptyset$, and so $\bd X=\bd Z=\emptyset$. 
 
 \begin{lemma}\label{L: equiv}
 The relation $\sim_{S\mc C}$ is an equivalence relation on the subclass consisting of those $X\in S\mc C$ such that $\bd X=\emptyset$.
 \end{lemma}
 \begin{proof}
If $X\in S\mc C$ with $\bd X=\emptyset$, then $I\times X$ realizes $X\sim X$, and if $X\sim Z$ via the oriented stratified bordism $Y$, then $Z\sim X$ via $-Y$. Finally, if $X\sim Z$ via  $Y$ and $Z\sim V$ via $W$, then $X\sim V$ via $Y\cup_Z W$. 
\end{proof}

 We can now define the groups\footnote{\label{F: N}We write $\Omega^{\mc C}_n$ rather than $\Omega^{S\mc C}_n$ as we will be slightly old-fashioned and use $\Omega$ for oriented bordisms groups  and use $\mc N$ for unoriented bordism groups in the few cases they arise. This prevents some worse notational sins in later superscripts.}  $\Omega^{\mc C}_n$.
 
 \begin{definition}\label{D: groups}
Let $\Omega^{\mc C}_n$ be the group generated by the $n$-dimensional elements  $X$ of $S\mc C$ such that $\bd X=\emptyset$, with commutative group operation being disjoint union $\amalg$ and with relations given by oriented stratified pseudomanifold bordism in  $S\mc C$. The identity is $\emptyset$, and the inverse of $X$ is $-X$. 
Notice that if $X\cong Y$, then $X\sim Y$ as $\bd (I\times X)=X\amalg -X\cong X\amalg -Y$.

 We denote the group element in $\Omega^{\mc C}_n$ corresponding to the space $X$ by $[X]$. 
\end{definition}

\begin{definition}\label{D: underlying class}
We will call an oriented weak stratified bordism class $S\mc C$ an \emph{oriented intrinsic weak stratified bordism class (or oriented IWS class)} if $X\in S\mc C$ implies $X'\in S\mc C$ whenever $|X|\cong |X'|$. 

Suppose $S\mc C$ is an oriented intrinsic weak stratified bordism class (oriented IWS class). We let $|S\mc C|\subset |S\Psi|$ denote the class of compact PL $\bd$-pseudomanifolds $|X|$ such that $X\in S\mc C$ for some, and hence any, stratification $X$ of $|X|$. The notation is meant to suggest that in $S\mc C$ we care about stratifications but in $|S\mc C|$ we do not.
\end{definition}

\begin{lemma}\label{L: IWS}
If $S\mc C$ is an oriented  IWS class, then $|S\mc C|$ has the following properties: 
\begin{enumerate}

\item If $|X|\in |S\mc C|$, then $-|X|\in |S\mc C|$. 
\item If $|X|\cong |Y|$ and $|X|\in |S\mc C|$ then $|Y|\in |S\mc C|$.
\item If $|X|\in |S\mc C|$, then $|\bd X|\in |S\mc C|$.
\item If $|X|\in |S\mc C|$ and $|\bd X|=\emptyset$, then $|I\times X|\in |S\mc C|$, using the product orientation. 
\item \label{I: nonevident} If $|W|,|Y|\in |S\mc C|$ with $|\bd Y|\cong |X\amalg -Z|$ and $|\bd W|=|Z\amalg -V|$, then $|Y\cup_Z W|\in |S\mc C|$ for any gluing of $W$ and $Y$ along $Z$ compatible with the orientations. 
\item $|S\mc C|$ contains an empty pseudomanifold of each dimension.
\end{enumerate}

Conversely, any subclass $S\mc Z$ of $|S\Psi|$ possessing these properties has the form $|S\mc C|$, where $S\mc C$ is the subclass of $S\Psi$ consisting of all $\bd$-stratified pseudomanifolds whose underlying spaces are in $S\mc Z$. If $S\mc C$ and $S\mc C'$ are oriented IWS classes, then $S\mc C=S\mc C'$ if and only if $|S\mc C|=|S\mc C'|$. 
\end{lemma} 
\begin{proof}
Given an  oriented  IWS class $S\mc C$, the  only condition claimed for $|S\mc C|$ that is not immediate from the preceding definitions is \eqref{I: nonevident}. So let $|W|$ and $|Y|$ be as given. We need to demonstrate that $|Y\cup_Z W|\in |S\mc C|$. By definition, there are stratified $W, Y\in S \mc C$ with underlying spaces $|W|$ and $|Y|$. Let $\bd Y=X\amalg -Z$ and $\bd W=Z'\amalg -V$. We can assume that $X$, $Z$, $Z'$, and $V$ have the corresponding underlying spaces $|X|$, $|Z|\cong |Z'|$, and $|V|$ in $|\bd Y|$ and $|\bd W|$, but we cannot assume that $Z$ and $Z'$ are stratified homeomorphic. However, by Corollary \ref{C: bordism} and Remark \ref{R: induced o}, there is an oriented  stratified bordism $Y$ between $Z$ and $Z'$ whose underlying space is $|A|\cong |I\times Z|$. So let $Y'=Y\cup_Z A$. Then $\bd Y'\cong X\amalg -Z'$, and since $|Y'|\cong |Y|$, $Y'\in S\mc C$. Therefore, $Y'\cup_{Z'}W\in S\mc C$, by the axioms for $S\mc C$, and so $|Y'\cup_{Z'}W|\cong |Y\cup_Z W|\in |S\mc C|$. 

Conversely, suppose $S\mc Z\subset |S\Psi|$ has the given properties, and let $S\mc C$ be the subclass of $S\Psi$ consisting of all elements whose underlying spaces are in $S\mc Z$. Clearly, $|S\mc C|=S\mc Z$, provided $S\mc C$ is an  oriented  IWS class. It is straightforward to verify the conditions and the remaining claims of the lemma.
\end{proof}

\begin{definition}\label{D: weak bord}
We will call any subclass of $|S\Psi|$ possessing the properties of the lemma an \emph{(unstratified) oriented weak bordism class}. 
\end{definition}

If $S\mc C$ is an oriented IWS class, then the properties of $|S\mc C|$ demonstrated in Lemma \ref{L: IWS} are sufficient to show that unstratified bordism is an equivalence relation in $|S\mc C|$ among objects with empty boundary: 
We will say that $|Y|$ is an oriented (unstratified) pseudomanifold bordism between the compact oriented PL  pseudomanifolds $|X|$ and $|Z|$ if $|Y|$ is a compact oriented PL pseudomanifold with $|\bd Y|\cong |X\amalg -Z|$. We will say that $|X|, |Z|\in |S\mc C|$ are oriented (unstratified) bordant in $|S\mc C|$, denoted $|X|\sim_{|S\mc C|} |Z|$,  if there exists an oriented pseudomanifold bordism $|Y|$ between $|X|$ and $|Z|$ such that $|Y|\in |S\mc C|$. If the class $|S\mc C|$ is understood, we may omit the phrase ``in $|S\mc C|$'' and simply write $|X|\sim |Z|$ for the relation.  Oriented pseudomanifold bordism in $|S\mc C|$ is an equivalence relation; the proof is completely analogous to that of Lemma \ref{L: equiv}. There result bordism groups $\Omega^{|\mc C|}_n$.

There are evident forgetful maps $\mf s:\Omega^{\mc C}_n\to \Omega^{|\mc C|}_n$  that take a stratified pseudomanifold in $S\mc C$  to its underlying space. 

We can now prove our main theorem concerning bordism groups:

\begin{theorem}\label{T: bord group}
If $S\mc C$ is an oriented IWS class, the forgetful maps $\mf s:\Omega^{\mc C}_n\to \Omega^{|\mc C|}_n$  are well-defined isomorphisms. 
\end{theorem}
\begin{proof}
The map $\mf s$ is well-defined because if $X\sim_{S\mc C} Z$  via the bordism $Y$, then $|X|\sim_{|S\mc C|} |Z|$  via $|Y|$. The map is also clearly surjective, as every generator $|X|$ of $\Omega^{|\mc C|}_n$ is the underlying space of some object $X$ of $S\mc C$ by the construction of $|S\mc C|$. It is less obvious that $\mf s$ is also injective, but this  follows from Corollary \ref{C: given bord}: if $X$ and $Z$ represent elements of $\Omega^{\mc C}_n$, then $\mf s(X)=\mf s(Y)$  if and only if there is some $Y'\in S\mc C$ with $|\bd Y'|\cong |X\amalg -Z|$. Let $\bd Y'\cong X'\amalg -Z'$. By Corollary \ref{C: given bord}, this is sufficient to determine a stratification $Y$ on $|Y'|$ such that $\bd Y\cong X\amalg -Z$. Since $|Y|\cong |Y'|$, $Y\in S\mc C$. So $[X]=[Z]\in \Omega^{\mc C}_n$.
\end{proof}

It will follow from a bit more work in the following subsection that this theorem has the following corollary, which answers the motivation question of McClure's:

\begin{corollary}
The stratification-forgetting map $\mf s$ induces an isomorphism of bordism groups $\Omega^{G-\text{Witt}}_n\to \Omega^{|G-\text{Witt}|}_n$ from bordism groups of stratified $G$-Witt spaces to bordism groups of unstratified $G$-Witt spaces.

The equivalent statement holds for bordism of $IP$ spaces defined over a Dedekind domain $R$ and for the unoriented $G$-Witt and $R$-IP bordism groups.
\end{corollary}

\begin{remark}
The assumption that $S\mc C$ is an oriented IWS class is used in the proof of  Theorem \ref{T: bord group} to ensure that our constructed stratified bordisms delivered by Corollary \ref{C: given bord}
are contained within $S\mc C$. If we drop the intrinsic condition and work only with a weak stratified bordism class, injectivity of $\mf s$ may no longer hold. It is not difficult to construct simple cases where this happens. For example, let $X$ be an oriented $n$-sphere, $n>1$, stratified as $S^n\supset \{pt\}$, and let $S\mc C$ be the class consisting of $X$, $I\times X$, $S^n$, $I\times S^n\times X$, and  $\emptyset$, each with both orientations, as well as all the spaces stratified homeomorphic to these, and all disjoint unions of collections of these spaces of the same dimension. Then $S\mc C$ is an oriented weak stratified bordism class. But while $|S^n|$ and $|X|$, with compatible orientations, are unstratified bordant via $|I\times S^n|$, the spaces $X$ and $S^n$ are not stratified bordant in $S\mc C$. 
\end{remark}

\begin{remark}\label{R: N}
Again, all results of this section have obvious unoriented analogues. In particular, we can define intrinsic weak stratified bordism classes (IWS classes)  as in Definition \ref{D: underlying class} but omitting orientation information, and, if $\mc C$ is such a class, then the forgetful map of unoriented bordism groups, $\mf s: \mc N_n^{\mc C}\to \mc N_n^{|\mc C|}$, is an isomorphism. 
\end{remark}

\subsection{Constructing IWS classes from classes of stratified pseudomanifold singularities.}\label{S: sps}

How can we construct and recognize (oriented) IWS classes? In his study of $\Q$-Witt spaces \cite{Si83}, Siegel first defined Witt spaces in \cite[Definition I.2.1]{Si83} in terms of a local intersection homology vanishing property on what are essentially the links of the intrinsic stratification at a non-boundary point; in other words, if $|S^k\ell|$ is the polyhedral link of a point $x\in |X|$ not contained in $|\bd X|$ and $|\ell|$ is not a suspension, the vanishing property is assumed for  $\ell$ (see Lemma \ref{L: int link}). He then goes on in \cite[Proposition I.2.5]{Si83} to show that assuming such a condition at all non-boundary points of $|X|$ is equivalent to assuming the same vanishing condition on all the links of a stratification of $|X|$. Later in \cite[Section IV.1]{Si83}, Siegel observes that the polyhedral links that arise in  $\Q$-Witt spaces  constitute a ``class of singularities''  in the sense of Akin \cite{Ak75}, and this provides a way to develop an unstratified bordism theory of $\Q$-Witt spaces.

Since we want to study \emph{stratified} bordisms, it is convenient to reverse this process somewhat. We will first define  \emph{classes of stratified pseudomanifold singularities}, which will serve as links of $\bd$-stratified pseudomanifolds; this will provide a way to construct  IWS classes $\mc C$.  Interestingly, classes of stratified pseudomanifold singularities will require fewer conditions than Akin's classes of singularities. In Section \ref{S: classes}, below, we then show that a class of stratified pseudomanifold singularities can be used to construct a  \emph{class of pseudomanifold singularities}, which will be the polyhedral links of the spaces in $|\mc C|$. We will show that such a class constitutes a class of singularities in Akin's sense, and so generates unstratified bordism theories according to \cite{Ak75}. We will then relate these stratified and unstratified bordism theories. 

Even though we wish to construct IWS classes, which consist of stratified spaces, it is useful to proceed by putting conditions on the underlying spaces of the possible links. This is motivated, in part, by the known examples, such as Witt and IP spaces, for which the link conditions are stated in terms of vanishing properties of stratification-invariant intersection homology groups.  Further validation of this approach comes from Proposition \ref{P: pm links}, below, which relates such conditions on the underlying spaces of the links to conditions on the polyhedral links, which are naturally unstratified.  

\begin{definition}\label{D: sps}
 We define\footnotemark $\mc E\subset |\Psi|$, where $|\Psi|$ is the class of compact classical PL $\bd$-pseudomanifolds, to be a  \emph{class of stratified pseudomanifold singularities} if
\begin{enumerate}
\item if $|X|\in \mc E$, then $\dim(|X|)>0$ unless $|X|=\emptyset$,
\item $|\emptyset|\in \mc E$ and $|S^1|\in \mc E$,
\item \label{I: underly} if $|X|\in \mc E$ and $|X|\cong |Y|$, then $|Y|\in \mc E$,

\item if $|X|\in \mc E$ then $|\bd X|=\emptyset$,
\item \label{I: susp} if $|X|\in \mc E$ and $|X|\neq \emptyset$, then the suspension $|SX|\in \mc E$. 
\end{enumerate}
\end{definition}

\footnotetext{Even though all the spaces in $\mc E$ are unstratified, we use the notation $\mc E$ instead of $|\mc E|$. Despite the risk of confusion here, this will help later when we want to use the notation $\Omega^{\mc E}_*$ for stratified bordism groups of spaces with links in $\mc E$ and $\Omega^{|\mc E|}_*$ for the corresponding unstratified bordism groups.}
\begin{remark}
Akin's definition of a ``class of singularities'' does not require the spaces to be pseudomanifolds, only compact polyhedra, but he does require condition \eqref{I: susp} to be an ``if and only if'' condition. 
\end{remark}

\begin{remark}
The condition that $\dim(|X|)\neq 0$ corresponds to our desire to avoid codimension one strata in our stratified pseudomanifolds.
\end{remark}

\begin{remark}
Notice that even if our ultimate interest is in oriented pseudomanifolds, the spaces in $\mc E$ do not carry any orientation information. See also Footnote \ref{F: Siegel}. 
\end{remark}

\begin{example}
Below, we will construct classes of stratified pseudomanifold singularities corresponding to various known classes of pseudomanifolds whose bordism groups have been studied. As our primary example, we can let $\mc E=\mc E_{\Q-\text{Witt}}$ be those closed pseudomanifolds $|Z|$ of dimension $>0$ (or empty) such that $I^{\bar m}H_{\dim(|Z|)/2}(|Z|;\Q)=0$ if $\dim(|Z|)$ is even \cite{Si83}. We will show below that this is indeed a class of stratified pseudomanifold singularities and the corresponding IWS class $\mc C_{\mc E}$ is the class of (stratified) $\Q$-Witt spaces.

\end{example}

Next we show that a class of stratified pseudomanifold singularities $\mc E$ determines an IWS class $\mc C_{\mc E}$.

\begin{lemma}\label{L: prop P}
Let $\mc E$ be a class of stratified pseudomanifold singularities, and let  $\mc C_{\mc E}\subset \Psi$ be the class of PL $\bd$-stratified pseudomanifolds whose links all have underlying spaces that  are elements of $\mc E$. Then $\mc C_{\mc E}$ is an IWS class and $SC_{\mc E}$, obtained from $\mc C_{\mc E}$ as in Remark \ref{R: SWS}, is an oriented IWS class.
\end{lemma}
\begin{proof}
The definition of $\mc C_{\mc E}$ is given entirely in terms of a condition on links: whether or not their underlying spaces are contained in $\mc E$. If $X$ is a $\bd$-stratified pseudomanifold, then due to the collaring condition on the boundary, all the links of $\bd X$ are also links of $X$. Similarly, if $\bd X=\emptyset$, all the links of $I\times X$ are links of $X$. If $W,Y\in \Psi$ with $\bd Y\cong X\amalg Z$ and $\bd W=Z\amalg V$, then each $Y\cup_Z W$ also has links that are already in $W$ or  $Y$.  Thus the only part of the claim that is not immediate is that  if $|X|\cong|X'|$ and $X\in \mc C_{\mc E}$ then $X'\in \mc C_{\mc E}$. Since the links of boundary points of $\bd$-stratified pseudomanifolds are all also links of non-boundary points by the existence of a stratified collar of the boundary, a PL $\bd$-stratified pseudomanifold $X$ will be in $\mc C_{\mc E}$ if and only if $X-\bd X$ has all its links in $\mc E$. Furthermore, since the boundary of a classical stratified pseudomanifold is determined by its underlying space (see Remark \ref{R: classical boundary}), it suffices to show that if all the links of $X-\bd X$ are in $\mc E$, then the same is true of $X'-\bd X'$. For this, since $(X-\bd X)^*\cong (X'-\bd X')^*$, it  suffices to show that if $A$ is a classical PL stratified pseudomanifold then its links are in $\mc E$ if and only if the links of $A^*$ are in $\mc E$.

First, suppose all the links of  $A^*$ are in $\mc E$. Let $x\in A$, let $L$ be the link of $x$ in the stratification $A$,  let $\ell$ be the link of $x$ in the stratification $A^*$, and let $\ms L$ be the polyhedral link of $x$ in $|A|$. Then $|\ms L|\cong |S^i\ell|\cong |S^kL|$ for some $i\geq j$ (see Lemma \ref{L: int link}), and it follows  that $|L|\cong |S^{i-j}\ell|$. If $|L|$ is empty, then $x$ is in a regular stratum of $A$ and hence also of $A^*$, so $|L|=|\ell|=\emptyset$, which is in  $\mc E$. Otherwise, if $L$ is not empty, then $|L|\cong |S^{i-j}\ell|$ and since we assumed the links of  $A^*$ are in $\mc E$, we have $|\ell|\in \mc E$ and hence so is  $|L|$ by the axioms for $\mc E$. Note that $|\ell|$ may be empty, in which case $L$ is a sphere of dimension at least $1$, and so in this case $|L|\in \mc E$ by the assumptions concerning $|S^1|$ and suspensions; $|L|$ cannot be $|S^0|$ because $A$ is a classical stratified pseudomanifold.

Conversely, suppose the links of $A$ are in $\mc E$ and that $x\in A^*$. So $x$ has a neighborhood in $A^*$ stratified homeomorphic to some $\R^i\times c\ell$. Since each stratum of $A^*$ is a union of strata of $A$, the image of  $\R^i\times \{v\}$ in $A^*$ must intersect some $i$-dimensional stratum $\mc S$ of $A$  in an $i$-dimensional open subset of $\mc S$. But then  $|\ell|$ must be the underlying space of a link of a point in $\mc S$. But since the links of $A$ are in $\mc E$, $|\ell|$ must be in $\mc E$. So the links of $A^*$ are in $\mc E$.

For $S\mc C_{\mc E}$, it follows from Remark \ref{R: SWS} that  $S\mc C_{\mc E}$ is an oriented weak stratified bordism class. That it is an oriented IWS class follows from $\mc C_{\mc E}$ being an IWS class and from Remark \ref{R: induced o}.
\end{proof}

\begin{remark}\label{R: no}
We have seen here only that $\mc C_{\mc E}$ is an IWS class and that $S\mc C_{\mc E}$ is an oriented IWS class. From Section \ref{S: theory}, below, will follow the stronger fact that these classes of spaces can be used to construct  bordism homology theories. Hence they are examples of classes that probably deserve to be called something like  ``(oriented) intrinsic strong stratified bordism classes'' or simply ``(oriented) intrinsic stratified bordism classes.'' However, we will not need this  general notion below, so we do not attempt to define it here.  The main drawback of attempting to define 
intrinsic stratified bordism classes
following the pattern of Akin's construction of unstratified bordism classes  is that it would then take us relatively far afield to prove that classes of the form $\mc C_{\mc E}$ or $S\mc C_{\mc E}$ always are  intrinsic stratified bordism classes. As  the $\mc C_{\mc E}$ or $S\mc C_{\mc E}$ are the principal classes  with which we are concerned and, as we are about to see, contain all previously-studied examples, and as we are able to prove everything we want about them without introducing the larger machinery, we leave the venture of defining intrinsic stratified bordism classes for the future.  
\end{remark}

\begin{remark}\label{R: bad E}
By contrast with Akin's bordism theories in \cite{Ak75}, if $\mc C$ is an IWS class  and we let $\mc E_{\mc C}$ consist of the classical PL pseudomanifolds homeomorphic to the underlying spaces of the links of the $\bd$-stratified pseudomanifolds in $\mc C$,  it does not necessarily follow that $\mc E_{\mc C}$ is a class of stratified pseudomanifold singularities. For one thing, as we have defined them, IWS classes might have non-empty spaces only in dimensions $n$ and $n+1$ for some $n\geq 0$, while any class of stratified pseudomanifold singularities must contain spaces of arbitrarily large dimension. Thus it also makes no sense to ask questions about $\mc C_{\mc E_{\mc C}}$. 

On the other hand, even if we begin with a class of the form $\mc C_{\mc E}$, then it is not clear whether $\mc E_{\mc C_{\mc E}}$ is a class of stratified pseudomanifold singularities. Proving that it is would require showing that for any non-empty object  $|L|\in \mc E$ that occurs as a link of an object $X\in \mc C_{\mc E}$, there is an object $Y\in \mc C_{\mc E}$ that has $|SL|$ as a link. It is not obvious this can always be done in $\mc C_{\mc E}$. We do notice, however, that while it is clear that 
 $\mc E_{\mc C_{\mc E}}\subset \mc E$, it is definitely  possible for this inclusion not to be an equality.   
For example, let $\mc E$ consist of $\emptyset$ and all spaces homeomorphic either to  $|S^n|$, for $n\geq 1$, or to $S^kT^2=S^k(S^1\times S^1)$ for $k\geq 1$. This is a class of stratified pseudomanifold singularities. Let $X\in \mc C_{\mc E}$, and suppose $x\in X$ has a neighborhood stratified homeomorphic to $\R^i\times cL$, where $|L|\cong |ST^2|$. Since the suspension points of $|ST^2|$ are the only points of $|ST^2|$ that do not have Euclidean neighborhoods, in every stratification of $|ST^2|$, the suspension points must be $0$-dimensional strata, with links $|T^2|$. But since links  in links of $X$ are themselves links of $X$, $|T^2|$ must  be a link of $X$. But this would imply that $X\notin \mc C_{\mc E}$, so no space in $\mc C_{\mc E}$ can have $|ST^2|$ as a link.

We will see below that other classes of spaces have much better relationships. 
\end{remark}

\subsubsection{Examples}\label{E: examples}

Despite any trepidation the reader may feel as a result of Remark \ref{R: bad E}, we now show that 
Lemma \ref{L: prop P} does give us a way to construct and recognize IWS classes and that, in fact, all of the previously-studied bordism theories of pseudomanifolds of which the author is aware arise from IWS classes of the form $\mc C_{\mc E}$. The results of this section will not be needed for us to continue our theoretical discussions in Section \ref{S: bordism homology}.

\paragraph{Pseudomanifolds.} Of course, we can let $\mc E$ be the class of all compact pseudomanifolds of dimensions greater than $0$, plus the empty set. In this case $\mc C_\mc E$ will be the class of all compact classical $\bd$-stratified pseudomanifolds. Unfortunately, the bordism groups in this class are not very interesting: every compact classical pseudomanifold is null-bordant via its cone, whether we take stratifications into account or not. Similarly, if we additionally require the elements of $\mc E$ to be orientable, then $S\mc C_{\mc E}$ will consist of all orientable compact classical $\bd$-stratified pseudomanifolds (note that links of orientable pseudomanifolds must be orientable), and again the resulting bordism groups will be trivial.

\paragraph{Mod $2$ Euler spaces.} Let $\mc E_{\chi}$ consist of compact pseudomanifolds  whose mod $2$ Euler characteristics (computed either simplicially or with ordinary homology with $\Z_2$ coefficients) are $0$. This condition makes no reference at all to stratifications. Then $\mc E_{\chi}$ is a class of stratified pseudomanifold singularities since  $\chi(|S^1|)=0$ and $|Z|$ has vanishing mod $2$ Euler characteristic if and only if $|S^jZ|$ does for any $j\geq 0$,  by the standard computation for the homology of a suspension. 
The resulting IWS class $\mc C_{\mc E_{\chi}}$ consists of the $\bd$-pseudomanifolds among the mod $2$ Euler spaces of Sullivan \cite{Sul71, Ak75}.   
In \cite{Ak75}, unoriented bordism groups of arbitrary polyhedra satisfying this condition on their polyhedral links are computed, but the arguments of Proposition 11(a) and on page 359 of \cite{Ak75} apply just as well replacing Akin's class of polyhedral Euler spaces with our $\mc C_{\mc E_{\chi}}$. The result is that $\mc N^{\mc C_{\mc E_{\chi}}}_n\cong \Z_2$ for all $n$, recalling that $\mc N$ stands for unoriented bordism (see Footnote \ref{F: N} and Remark \ref{R: N}). Technically, Akin computes the unoriented bordism groups of a point, treating 
$\mc N^{\mc C_{\mc E_{\chi}}}_n$ as a homology theory, but we show below in Lemma \ref{L: akin pt is bordism} that this is equivalent to computing the bordism groups as defined above in Definition \ref{D: groups}.

\paragraph{Witt spaces.} $G$-Witt spaces (generalizing the original $\Q$-Witt spaces of Siegel \cite{Si83}) are classical $\bd$-stratified pseudomanifolds characterized by the property that if $L$ is an even-dimensional link  then the intersection homology group $I^{\bar m}H_{\dim L/2}(L;G)$ is equal to $0$. We claim that the corresponding class $\mc E_{G-\text{Witt}}$ of compact classical pseudomanifolds $|Z|$  of dimension $>0$ (if not empty) satisfying the property that $I^{\bar m}H_{\dim(|Z|)/2}(|Z|;G)=0$ if $\dim(|Z|)$ is even constitutes a class of stratified pseudomanifold singularities. 
 This condition is independent of the stratification of $|Z|$ because lower-middle-perversity intersection homology is a topological invariant of pseudomanifolds \cite{GM2}. Furthermore, the middle-dimensional lower-middle-perversity intersection homology of an even-dimensional suspension of a non-empty space is always trivial by basic computations (see \cite{Ki} or \cite{GBF35}). So $\mc E_{G-\text{Witt}}$ is indeed a class of stratified pseudomanifold singularities, and the class $\mc C_{\mc E_{G-\text{Witt}}}$ of spaces whose links have this property is precisely the class of $G$-Witt spaces \cite[Proposition I.2.5]{Si83}. The bordism groups of oriented $\Q$-Witt spaces were compute in \cite{Si83}, the bordism groups of oriented $\Z_2$-Witt spaces were computed in \cite[Section 10.5]{GP89}, and the bordism groups of oriented $K$-Witt spaces for all other fields $K$ were computed in\footnotemark \cite{GBF21, GBF33, GBF34}. Unoriented $K$-Witt bordism in characteristic $2$ is given in \cite{Go84}; see also \cite{GBF34}. I do not know of any computations of the unoriented $K$-Witt bordism groups for $\text{char}(K)\neq 2$.

\footnotetext{For the historical record, I would like to make clear the following chain of events: In \cite{GBF21}, among other results, I extended Siegel's computation of $\Q$-Witt bordism groups to fields of arbitrary characteristics, not realizing that this had already been done in characteristic 2 by Goresky in \cite{Go84} and Goresky-Pardon \cite{GP89}. To make matters worse,  my computations contained an  error in the characteristic 2 case. In an  attempt to fix this error, the corrigendum \cite{GBF33} was published, with further details provided in a separate paper \cite{GBF34}, in which the $4k+2$ dimensional case of oriented $\Z_2$-Witt bordism groups are left unresolved. In that paper, I also acknowledged Goresky's original calculation of unoriented $\Z_2$-Witt bordism in \cite{Go84}, which I had discovered by that point,  and provided some details of the computation not made explicit in \cite{Go84}. I was unaware, however, that the solution to the  $4k+2$ case of oriented $\Z_2$-Witt bordism, as well as the complete computation of $\Omega_*^{\Z_2-\text{Witt}}$, had been lurking in \cite{GP89} all along. I apologize for introducing this confusion into the literature.}

\paragraph{IP spaces.} $R$-IP spaces \cite{Pa90} are defined by the property that if a link $L$ is even-dimensional then $I^{\bar m}H_{\dim(L)/2}(L;R)=0$ and if $L$ is odd-dimensional then the $R$-torsion submodule of $I^{\bar m}H_{\frac{\dim(L)-1}{2}}(L;R)$ is trivial. Let $\mc E_{R\text{-IP}}$ be the class of all compact classical pseudomanifolds $|Z|$ of dimension $>0$  that satisfy these homological properties (plus the empty set).
Again, we note that these conditions are independent of stratification because these intersection homology groups do not depend on the stratification of $|Z|$ \cite{GM2}. Standard computations in intersection homology show that if $|Z|\neq \emptyset$ is compact and  $|SZ|$ is even-dimensional, then $I^{\bar m}H_{\dim(|SZ|)/2}(|SZ|;R)=0$  and if $|SZ|$ is odd-dimensional of dimension $>1$, then $I^{\bar m}H_{\frac{\dim(|SZ|-1)}{2}}(|SZ|;R)=I^{\bar m}H_{\frac{\dim(|Z|)}{2}}(|SZ|;R)\cong I^{\bar m}H_{\frac{\dim(|Z|)}{2}}(|Z|;R)$, which is $0$ by assumption if $|Z|\in \mc E_{R\text{-IP}}$. The only pseudomanifold suspension of dimension $1$ is $|S^1|$; the intersection homology groups of $|S^1|$
 agree with the ordinary homology groups of $|S^1|$ and so have no $R$-torsion. Therefore, the class of $\mc E_{R\text{-IP}}$ is a class of stratified pseudomanifold singularities and the corresponding IWS class is the class of $R$-IP spaces. In \cite{Pa90}, Pardon computed the oriented bordism groups of $\Z$-IP spaces. I do not know of any other computations of IP space bordism groups.

\paragraph{Locally-orientable pseudomanifolds and locally orientable $G$-Witt spaces.} The class $S\mc E$ of compact orientable pseudomanifolds of dimension $>0$ (or empty) is a class of stratified pseudomanifold singularities and the corresponding IWS class is the class of locally-orientable pseudomanifolds of \cite{GP89}. Similarly, the class of compact orientable pseudomanifolds satisfying the $G$-Witt condition provides  a class of stratified pseudomanifold singularities and, when $G=\Z_2$, the corresponding IWS class is the class of locally-orientable Witt spaces of  \cite{GP89}. The unoriented bordism groups of locally-orientable pseudomanifolds and the unoriented bordism groups of locally-orientable $\Z_2$-Witt spaces  were computed in \cite[Corollary 9.3 and Section 10.5]{GP89}. Every orientable pseudomanifold is locally-orientable by \cite[Section 8.3]{GP89}, so oriented bordism of locally-orientable pseudomanifolds is the same as oriented bordism of pseudomanifolds, which is a trivial theory by taking cones, and oriented bordism of locally-orientable $\Z_2$-Witt spaces is the same as oriented bordism of $\Z_2$-Witt space. The latter groups are computed in \cite[Section 10.5 Theorem A]{GP89}.

\paragraph{$\bar s$-duality spaces.} Goresky and Pardon \cite[Section 21.2]{GP89} define a stratified pseudomanifold $X$ to be an $\bar s$-duality space if 
\begin{enumerate}
\item its even dimensional links satisfy the $\Z/2$-Witt condition $I^{\bar m}H_{\dim(L)/2}(L;\Z/2)=0$, 
\item its $2k-1$ dimensional links satisfy $I^{\bar m}H_k(L;\Z/2)=I^{\bar m}H_{k-1}(L;\Z/2)=0$,
\item $X$ has no strata of codimensions $1$, $2$, $3$, or $4$. 
\end{enumerate}

To translate these conditions into a class of stratified pseudomanifold singularities, we  first note that the last conditions imply that  all $1$-, $2$-, or $3$-dimensional links of $|X|$ in any stratification must be  spheres:  Suppose $X$ (and hence $X^*$, or using $(X-\bd X)^*$ if $\bd X\neq \emptyset$) has no strata of codimensions $1$, $2$, $3$, or $4$, and suppose $X'$ is another classical PL pseudomanifold stratification of $X$ and that $x\in X'$ is contained in a stratum of codimension $\leq 4$.
Then $x$ must lie in a regular stratum of $X^*$, so $x$ has polyhedral link $|S^{n-1}|$. But in the stratification $X'$,
  $x$ has a distinguished neighborhood  stratified PL homeomorphic to some $\R^i\times cL$.   By uniqueness of polyhedral links, $|S^{n-1}|\cong |S^{i-1}L|$, but this then implies that $|L|$ is  PL homeomorphic to a sphere (see Section \ref{S: susps}). 
 Conversely, if $|X|$ has the property that the links of codimension $1$, $2$, $3$, or $4$ strata are all spheres for any stratification $X$ of $|X|$,  then points with these links will be in the regular strata of the intrinsic stratification and so the intrinsic stratification of $|X|$ will not possess any  strata of the forbidden dimensions. Therefore, since we hope to achieve an IWS class, we rewrite the  last property to  the condition that   all $1$-, $2$-, or $3$-dimensional links of $|X|$ in any stratification must be  spheres. 

So let $\mc E_{\bar s}$ be the compact classical PL pseudomanifolds of dimension $>0$ (plus the empty set) such that 
\begin{enumerate}
\item if $\dim(|Z|)\in \{1,2,3\}$ then $|Z|\in \{S^1, S^2,S^3\}$,
\item if $\dim(|Z|)=2k$, then $I^{\bar m}H_{k}(|Z|;\Z_2)=0$, 
\item if $\dim(|Z|)=2k-1$, $k>1$, then $I^{\bar m}H_k(|Z|;\Z/2)=I^{\bar m}H_{k-1}(|Z|;\Z/2)=0$. 
\end{enumerate}

Since the relevant intersection homology groups are stratification independent \cite{GM2}, so are these conditions. 
  We have already observed that if $|SL|$ is any even-dimensional suspension of a non-empty space then its middle-dimensional lower-middle-perversity intersection homology groups vanish for any coefficients and that if $|SZ|$ has dimension $2k-1$, $k>1$,  then $I^{\bar m}H_{k-1}(|SZ|;\Z/2)\cong I^{\bar m}H_{k-1}(|Z|;\Z/2)$,  which is $0$ if $|Z|\in \mc E_{\bar s}$. Similarly, the standard suspension calculations show that if $\dim(|SZ|)=2k-1$ then $I^{\bar m}H_{k}(|SZ|;\Z/2)=0$ always. The low-dimensional conditions are clear for $|SZ|$ if $|Z|\in \mc E_{\bar s}$, and $S^1$ is allowed by the definition.
  So, again, the given conditions determine a class of stratified pseudomanifold singularities $\mc E_{\bar s}$ and a resulting IWS class, which consists of exactly the $\bar s$-duality spaces of \cite{GP89}. The oriented bordism groups of $\bar s$-duality spaces are computed in \cite[Theorem 16.5]{GP89}.

\paragraph{LSF spaces.} In \cite{GP89}, Goresky and Pardon define a stratified $\bd$-pseudomanifold $X$ to be locally square free (LSF) if its even dimensional links satisfy the $\Z/2$-Witt vanishing condition and its $2k-1$ dimensional links satisfy the property that the map $Sq^1: I^{\bar m}H_k(L;\Z/2)\to I^{\bar m}H_{k-1}(L;\Z/2)$ is the $0$ map. Here $Sq^1$ is an intersection homology Bockstein constructed in \cite[Section 6.6]{GP89}. 
 The map $Sq^1$ is not always defined on a PL stratified pseudomanifold $Z$ but depends upon the existence of a certain map of Deligne sheaves in the derived category of sheaves over $Z$; this map exists if and only if $Sq^1$ is already defined recursively on the links of $Z$ and, for such links $\ell$, the maps $Sq^1:I^{\bar m}H_k(\ell;\Z/2)\to I^{\bar m}H_{k-1}(\ell;\Z/2)$ are trivial if $\dim(\ell)=2k-1$ and 
 the maps $Sq^1:I^{\bar m}H_{k+1}(\ell;\Z/2)\to I^{\bar m}H_{k}(\ell;\Z/2)$ are trivial if $\dim(\ell)=2k$
 \cite[Section 6.6]{GP89}.
It is not obvious that these conditions are stratification independent; however, up to isomorphism in the derived category, the Deligne sheaf does not depend on the stratification and so the existence of the requisite map of Deligne sheaves is stratification independent. It follows that the property of having a well-defined $Sq^1$  is a property of PL pseudomanifolds, independent of their stratification. 

For LSF spaces, the condition that $Sq^1$ is well-defined on the odd dimensional links is part of the definition. It is not obvious from the definition, but it turns out, as we will prove below in Lemma \ref{L: LSF}, that if $X$ is an LSF space then $Sq^1$ is also well-defined on all even dimensional links. So, we  make this part of the definition of the class $\mc E_{LSF}$. Thus, we let $\mc E_{LSF}$ consist of  the compact classical PL pseudomanifolds $|Z|$ of dimension $>0$ (plus the empty set) on which $Sq^1$ is defined and such that $I^{\bar m}H_{k}(|Z|;\Z/2)$ vanishes if $\dim(|Z|)=2k$ and $Sq^1: I^{\bar m}H_k(|Z|;\Z/2)\to I^{\bar m}H_{k-1}(|Z|;\Z/2)$ vanishes if $\dim(|Z|)=2k-1$. 

We have already observed that having $Sq^1$ well-defined is independent of the stratifications. Also, $\emptyset$ is in $\mc E_{LSF}$ trivially, as is  $|S^1|$ because  $Sq^1$ is trivial on $|S^1|$ (on a manifold, $Sq^1$ is the homology dual of the  standard Steenrod square, which vanishes in this case for dimensional reasons). 
We must check that if $|Z|\in \mc E$ then so is $|SZ|$. 
 As previously observed, the vanishing of middle-dimensional lower-middle-perversity intersection homology is automatically satisfied for any even-dimensional suspension of a non-empty space, and if $\dim(|SZ|)=2k-1$, $k>1$, then  $I^{\bar m}H_{k}(|Z|;\Z/2)=0$, also by the properties of suspensions. Thus if $|SZ|$, $\dim(|Z|)>0$, is odd-dimensional, $Sq^1: I^{\bar m}H_k(|Z|;\Z/2)\to I^{\bar m}H_{k-1}(|Z|;\Z/2)$ will be trivial if it is well-defined. To see that $Sq^1$ is well-defined on any $|SZ|$ if $|Z|$ satisfies the given conditions, it suffices by \cite[Section 6.6]{GP89} to check that $Sq^1$ is well-defined  on all the links of $SZ$, for some stratification $Z$ of $|Z|$, and that it vanishes in the appropriate dimensions. The  
 links of $SZ$ are the links of $Z$ together with $Z$ itself.  But for $Z$ to be in $\mc E_{LSF}$, we have assumed that $Sq^1$ is well-defined on $Z$ and that it vanishes in the appropriate dimensions. The well-definedness of $Sq^1$ on $Z$  then implies by \cite[Section 6.6]{GP89} that it is also well-defined on all the links of $Z$ and vanishes on these links in the appropriate dimensions. 

We have now shown that $\mc E_{LSF}$ is a class of stratified pseudomanifold singularities, and therefore the LSF spaces constitute an IWS class. Oriented bordism of LSF spaces is computed in \cite[Theorem 13.1]{GP89}.

We finish our discussion of this example with the following delayed lemma.

\begin{lemma}\label{L: LSF}
Let $X$ be an LSF space. Then $Sq^1$ is well-defined on all links of $X$.
\end{lemma}
\begin{proof}

Recall from \cite[Section 6.6]{GP89} that $Sq^1$ is defined on a stratified pseudomanifold $Z$ if and only if

\begin{enumerate}
\item it is defined on all links $L$ of $Z$,

\item the maps $Sq^1:I^{\bar m}H_k(L;\Z/2)\to I^{\bar m}H_{k-1}(L;\Z/2)$ are trivial on the links of $Z$ of dimension $\dim(L)=2k-1$, and 
\item  the maps $Sq^1:I^{\bar m}H_{k+1}(L;\Z/2)\to I^{\bar m}H_{k}(L;\Z/2)$ are trivial on the links of $Z$ of dimension $\dim(L)=2k$.
\end{enumerate}

That $Sq^1$ is well-defined on all odd links of $X$ and vanishes in the appropriate dimensions is part of the definition of $X$ being an LSF space. We will prove by induction on the dimensions of even-dimensional links that $Sq^1$ is well-defined on these links. The vanishing condition is automatic by the  Witt vanishing condition on even-dimensional links in  the definition of an LSF space.

Recall  that if $L$ is a link of a PL stratified pseudomanifold $X$, then the links of $L$ are also links of $X$.

Let $L$ be an even-dimensional link of $X$ of smallest dimension. Therefore, all links of $L$ are odd-dimensional (or empty) links of $X$ on which $Sq^1$ is defined and vanishes in the required dimension by the assumption that $X$ is LSF. Thus $Sq^1$ is well-defined on $L$; it vanishes due to the Witt  vanishing condition imposed on $L$ by $X$ being an LSF. Suppose now that we have shown the lemma for all even dimensional links of $X$ of dimension $<2k$ and that $L$ is a link of $X$ of dimension $2k$. Once again, all the links of $L$ are links of $X$ of dimension $<2k$, so $Sq^1$ is defined and vanishes in the correct dimensions  by the LSF conditions on the odd dimensional links of $X$ and by induction and the Witt vanishing condition on the even dimensional links of $X$. This completes the proof by induction.   
\end{proof}

\paragraph{Spaces with trivial perverse signatures.} This example has not previously been studied in the literature. Let $\bar p$ and $\bar q$ be complementary perversities satisfying the requirements of Goresky and MacPherson in \cite{GM1}; furthermore, suppose $\bar p\leq \bar q$. If $X$ is a closed orientable $4k$-dimensional PL stratified pseudomanifold, then there is defined a \emph{perverse signature} with respect to $\bar p$ and $\bar q$, which is the signature of the intersection pairing restricted to the image of the natural map $I^{\bar p}H_{2k}(X;\Q)\to  I^{\bar q}H_{2k}(X;\Q)$. This perverse signature was introduced by Hunsicker \cite{Hun07}; see also \cite{GBF27}. If we assume that  $X$ is a $\Q$-Witt space and that $\bar p=\bar m$ and $\bar q=\bar n$, then the perverse signature is the Witt signature of $X$. 
Since $\bar p,\bar q$ satisfy the Goresky-MacPherson conditions, the corresponding intersection homology groups, and hence the perverse signature, are topological invariants. Let $\mc E_{\bar p,\bar q}$ consist of the closed orientable  PL pseudomanifolds of dimension $>0$ (or empty) with vanishing perverse signature with respect to $\bar p$ and $\bar q$; if $\dim(|Z|)$ is not a multiple of $4$, its perverse signature is $0$ by definition. The remaining condition to verify to prove that $\mc E_{\bar p,\bar q}$ is a class of stratified pseudomanifold singularities is that the suspension of any $4k-1$ dimensional element of $\mc E_{\bar p,\bar q}$ is also in $\mc E_{\bar p,\bar q}$, but, as we have observed previously, if $|Z|$ is a suspension of positive,  even dimension (in this case,  dimension $4k$, $k>0$), then $I^{\bar m}H_{2k}(|Z|;\Q)=0$. The conditions on the perversities $\bar p$ and $\bar q$ imply that we must have $\bar p\leq \bar m\leq \bar n\leq \bar q$, so the map $I^{\bar p}H_{2k}(|Z|;\Q)\to  I^{\bar q}H_{2k}(|Z|;\Q)$
factors through  $I^{\bar m}H_{2k}(|Z|;\Q)=0$. Therefore, the perverse signature of a suspension is trivial. 

So far, no computations of the bordism groups associated to this class have been carried out. However, there is some evidence to believe that such an inquiry would be profitable: in \cite{Ba02}, Banagl studies a class of spaces (now called $L$-spaces; see \cite{ALMP-cheeger})  carrying self-dual sheaves compatible with intersection homology. One\footnotemark of the defining properties of these spaces is precisely the vanishing of the signatures of links; these signatures are defined with respect to the sheaf cohomology of the self-dual sheaves (restricted to the link). Bordism  of $L$-spaces is studied in \cite{Ba02}, and the associated bordism homology theory, dubbed ``signature homology,'' was introduced by Minatta in \cite{Min06}; see also \cite{Ba06a}. For more on the general philosophy of constructing bordism theories of this type, see Banagl's survey article \cite{Ba11}.

\footnotetext{The second defining characteristic of $L$-spaces, a monodromy-along-strata property on the Lagrangian subspaces associated with the vanishing signatures, has no useful analogy here.}

\section{Bordism homology theories}\label{S: bordism homology}

In this section, we generalize Theorem \ref{T: bord group}, which  concerned bordism \emph{groups}, to Theorem \ref{T: theory}, which concerns bordism \emph{homology theories}. The casual reader could easily jump at this point to Theorem \ref{T: theory} and have little trouble understanding either the statement or the idea of the proof. However, in order to develop these homology theories rigorously, we will need a deal of preliminary work, which is provided in the first few subsections of Section \ref{S: bordism homology}. We first provide an overview:

Siegel's construction of a bordism homology theory based on Witt spaces uses the machinery of Akin \cite{Ak75}. Akin defines quite general unoriented bordism theories of polyhedra, though we will show that in the relevant special cases these provide pseudomanifold bordism theories. As observed by Siegel \cite[Section IV.1]{Si83}, Akin's bordism constructions carry over directly to oriented bordism.

 One method given by Akin for generating specific bordism homology theories  is by first specifying a ``class of singularities'' $\mc D$ and then looking at spaces $\mc F=\mc F_{\mc D}$ whose polyhedral links (at non-boundary points) lie within $\mc D$. We then obtain a bordism homology theory $\Omega^{\mc F}_*(\cdot)$ such that $\Omega^{\mc F}_*(T)$ is generated, roughly speaking, by maps from spaces in $\mc F$ to $T$ and with relations given by bordisms between maps. Of course, more generally, the homology theory is defined on pairs $(T,T_0)$; we review the construction in more detail below. The main thing we wish to note for now is that Akin's bordism homology theory constructions (and Siegel's use of them) do not take stratifications into account at all. 

By contrast, we have been studying both stratified pseudomanifolds and their underlying (unstratified) pseudomanifolds. Thus, rather than begin with a class of singularities in Akin's sense, it made more sense for us to begin with  classes of stratified pseudomanifold singularities $\mc E$  in order to specify what can be the links of stratified pseudomanifolds in IWS classes $\mc C_{\mc E}$; but now we would like to show how a class of stratified pseudomanifold singularities gives rise to a class of singularities, and hence a bordism homology theory, in Akin's sense. Motivated by Siegel's work in \cite{Si83}, in which the stratifications do sometimes play a useful role, we must  first study relationships between  classes of links of stratified pseudomanifolds and classes of polyhedral links of unstratified pseudomanifolds, and then move on to develop related bordism homology theories based on both stratified and unstratified pseudomanifolds. The story will unfold through several subsections of this long Section \ref{S: bordism homology}; to better orient the reader, we provide a brief outline of these subsections.

We begin in Section \ref{S: classes} by defining \emph{classes of pseudomanifold singularities} $\mc G$. These will be our versions of Akin's classes of singularities; they are also the unstratified analogues of the classes of stratified pseudomanifold singularities $\mc E$. In fact, a class of stratified pseudomanifold singularities $\mc E$ determines a class of pseudomanifold singularities  $\mc G=\mc G_{\mc E}$, and we show in  Proposition \ref{P: pm links} that  the links of the strata of a $\bd$-stratified pseudomanifold $X$ are contained in $\mc E$ if and only if the polyhedral links of points of $|X|-|\bd X|$ are contained in $\mc G_{\mc E}$.

In  Section \ref{S: F}, we show that a class of pseudomanifold singularities $\mc G$ gives rise to a  \emph{pseudomanifold bordism class $\mc F_{\mc G}$} consisting of those $\bd$-pseudomanifolds such that the polyhedral links of points of $|X|-|\bd X|$ are contained in $\mc G$. We demonstrate in Lemma \ref{L: some lemma} that if $\mc G=\mc G_{\mc E}$, then $\mc F_{\mc G_{\mc E}}=|\mc C_{\mc E}|$.

In Section \ref{S: unstrat bord}, we consider the unstratified bordism homology theory $\Omega^{|\mc G|}_*(\cdot)$  determined by pseudomanifold bordism classes $\mc F_{\mc G}$. The existence of these homology theories follows directly from Akin \cite{Ak75}, and we connect this to our work in previous sections by showing in Lemma \ref{L: akin pt is bordism} that $\Omega^{\mc G}_n(pt)\cong \Omega_n^{|\mc F_{\mc G}|}$, where $\Omega_n^{|\mc F_{\mc G}|}$ is an  unstratified  bordism group in the sense of Section \ref{S: bordism groups}.

Section \ref{S: theory} then contains our investigation of stratified bordism as a homology theory, culminating in 
Theorem \ref{T: theory}, which says that a  stratified bordism homology theory we can construct based on an IWS class $\mc C_{\mc E}$ is isomorphic to the corresponding unstratified bordism homology theory based on  $|\mc C_{\mc E}|=\mc F_{\mc G_{\mc E}}$.

\subsection{Classes of pseudomanifold singularities}\label{S: classes}

 We begin with the instances of Akin's classes of singularities that will suit our needs here.

\begin{definition}
 We define\footnotemark $\mc G\subset |\Psi|$, where $|\Psi|$ is the class of compact classical PL $\bd$-pseudomanifolds, to be a  \emph{class of pseudomanifold singularities} if
\begin{enumerate}
\item if $|X|\in \mc G$, then $|\bd X|=\emptyset$,
\item $|\emptyset|\in \mc G$ ,
\item \label{I: underly2} if $|X|\in \mc G$ and $|X|\cong |Y|$, then $|Y|\in \mc G$,
\item \label{I: susp2}  $|X|\in \mc G$ if and only if $|SX|\in \mc G$. 
\end{enumerate}
\end{definition}

\footnotetext{As for $\mc E$, we use the notation $\mc G$ rather than $|\mc G|$ even though all the spaces in $\mc G$ are unstratified. }

\begin{remark}\label{R: s0}
The assumption that all elements of $\mc G$ are classical pseudomanifolds, together with the requirement on suspensions, implies that the only space in $\mc G$ of dimension $0$ is $S^0$.  
\end{remark}

\begin{lemma}\label{L: akin}
A class of pseudomanifold singularities is a class of singularities in the sense of Akin \cite[Definition 8]{Ak75}.
\end{lemma}
\begin{proof}
The only difference between our definition of a class of pseudomanifold singularities and Akin's is the requirement that the spaces be pseudomanifolds and not just arbitrary compact polyhedra. But clearly any collection of pseudomanifolds satisfying the requirements is also a collection of polyhedra satisfying the requirements. 
\end{proof}

The following  construction is motivated by Siegel's definition of Witt spaces in terms of polyhedral links by a process that really depends upon looking at the intrinsic link. 

\begin{lemma}\label{L: e to g}
Given a class of stratified pseudomanifold singularities $\mc E$, we obtain a class of pseudomanifold singularities $\mc G_{\mc E}$ as follows:
If $|Y|$ is a compact pseudomanifold, then there are unique $i\geq 0$ and pseudomanifold $|Z|$ (up to homeomorphism) such that $|Y|\cong |S^iZ|$ and $|Z|$ is not a suspension of a pseudomanifold. We declare $|Y|\in \mc G_{\mc E}$ if and only if $|Z| \in \mc E$. 
\end{lemma}

\begin{proof}
We show that $\mc G_{\mc E}$ really is a class of pseudomanifold singularities. First, we observe that there exists  a unique such $|Z|$ by the arguments in Section \ref{S: susps}. 

Clearly $\emptyset=S^0\emptyset\in \mc G_{\mc E}$. If $|X|\cong |Y|$, then both have the same $|Z|$ to check for inclusion in $\mc E$, so $|X|\in \mc G_{\mc E}$ if and only if $|Y|\in \mc G_{\mc E}$. If $|X|\in \mc G_{\mc E}$, then $|X|\cong |S^iZ|$ for some $|Z|\in \mc E$ and some $i\geq 0$, so since $|\bd Z|=\emptyset$, $|\bd X|=\emptyset$. 

If $|X|\in \mc G_{\mc E}$ and $|X|\cong |S^iZ|$ with $|Z|\in \mc E$ not a suspension of a pseudomanifold, then $|SX|\cong |S^{i+1}Z|$, so $|SX|\in \mc G_{\mc E}$. And, finally,  if $|SX|\in \mc G_{\mc E}$, then $|SX|\cong |S^iZ|$ with $|Z|\in \mc E$ and $|Z|$ not a suspension of a pseudomanifold. But this implies that we must have $i>0$ and $|X|\cong |S^{i-1}Z|$ by basic PL topology (see Section \ref{S: susps}). So also $|X|\in \mc G_{\mc E}$. 
\end{proof}

The next proposition is a generalization of Siegel's observation that the $\Q$-Witt spaces could be defined either in terms of a condition defined on links of strata in a stratification or in terms of conditions on the polyhedral links at points. 

\begin{proposition}\label{P: pm links}
Let $\mc E$ be a class of stratified pseudomanifold singularities, and let $X$ be a classical $\bd$-stratified pseudomanifold. Then the links of the strata of $X$ are contained in $\mc E$ if and only if the polyhedral links of points of $|X|-|\bd X|$ are contained in $\mc G_{\mc E}$.
\end{proposition}
\begin{proof}
The links in the stratification of $X$ are the same as the links in the stratification of $X-\bd X$ due to the stratified collar assumption on the boundary, so we can assume for the rest of the argument that $\bd X=\emptyset$. 

Now, suppose the links of $X$ are contained in $\mc E$ and that  
 $x\in |X|$ with polyhedral link $|\Lk(x)|$. By Lemma \ref{L: int link}, $|\Lk(x)|\cong |S^{i-1}\ell|$, where $|\ell|$ is the link of the $i$-dimensional stratum of $X^*$ containing $x$ and $|\ell|$ is not a suspension. Since $X^*$ coarsens all other stratifications of $X$, there are points in the stratum  of $X^*$ containing $X$ that are also in $i$-dimensional strata of $X$, and these points all have $|\ell|$ as their link by Lemma \ref{L: unique link}.  Therefore $|\ell|$ is a link in the stratification $X$, so $|\ell|\in \mc E$ and $|\Lk(x)|\in \mc G_{\mc E}$. 

Suppose now that the polyhedral links of $|X|$ are contained in $\mc G_{\mc E}$, and suppose $x\in X$. Let $L$ be the link of $x$ in $X$, and let $|\Lk(x)|$ be its polyhedral link. By assumption, $|\Lk(x)|\cong |S^iZ|$ for some $i\geq 0$ and some $|Z|$ that is not a suspension and $|Z|\in \mc E$. Furthermore, since $L$ is the link of $x$ in $X$, $|\Lk(x)|\cong |S^{j}L|$ for some $j$. Therefore, $|S^iZ|\cong |S^{j}L|$, and we must have $i\geq j$, as we know that $|Z|$ is the maximal desuspension of $|\Lk(x)|$. Therefore, $|L|\cong |S^{i-j}Z|$, and since $|Z|\in \mc E$, also $|L|\in \mc E$, as suspensions of elements of $\mc E$ are also in $\mc E$, at least assuming that $|Z|$ in non-empty. But if $|Z|$ is empty, then $|L|$ is also either empty or a sphere of dimension $>0$, as $X$ has no codimensions one strata; again this implies $|L|\in \mc E$.  
\end{proof}

We have seen that when given a class of stratified pseudomanifold singularities $\mc E$, we can construct a class of pseudomanifold singularities $\mc G_{\mc E}$. Conversely, it is clear that given a class of pseudomanifold singularities $\mc G$, we can obtain  a class of stratified pseudomanifold singularities $\mc E_{\mc G}$ simply by throwing away any $0$-dimensional spaces from $\mc G$ and observing that what remains satisfies the requirements to be a class of pseudomanifold singularities. However, as the following lemma shows, these procedures are not inverse to each other so that our classes of stratified pseudomanifold singularities and pseudomanifold singularities are not so trivially related to each other. 

\begin{lemma}\label{L: E and G}
$\mc G_{\mc E_{\mc G}}= \mc G$ and $\mc E_{\mc G_{\mc E}}\subset \mc E$, but it is not necessarily true that $\mc E\subset \mc E_{\mc G_{\mc E}}$. In particular, it is possible to have $\mc G_{\mc E}=\mc G_{\mc E'}$ even if $\mc E\neq \mc E'$. 
\end{lemma}
\begin{proof}

First we observe from the definitions that  $|X|\in \mc G_{\mc E_{\mc G}}$ if and only if $|X|\cong |S^jZ|$ for some $|Z|$ that is not a suspension and such that $|Z|\in \mc G$ and does not have dimension $0$. 
Thus, if  $|X|\in \mc G_{\mc E_{\mc G}}$, then $|X|$ is a suspension of $|Z|\in \mc G$, so $|X|\in \mc G$. Therefore, $\mc G_{\mc E_{\mc G}}\subset \mc G$.

 On the other hand, suppose $|X|\in \mc G$. Then we know $|X|\cong |S^jZ|$ for some $j\geq 0$ and some $|Z|$ that is not a suspension. By definition of $\mc G$, $|Z|\in \mc G$. Thus $|X|\in \mc G_{\mc E_{\mc G}}$ if $\dim(|Z|)\neq 0$. But if $\dim(|Z|)$ is $0$, then we must have $|Z|\cong |S^0|$ by Remark \ref{R: s0}, and $|S^0|$ is the suspension of the empty set. Thus it is impossible to have $|Z|\cong |S^jZ|$ with $|Z|$ both not a suspension and an element of $\mc G$ of dimension $0$. 
 Therefore, we have shown $\mc G\subset \mc G_{\mc E_{\mc G}}$.

Now,  suppose $|X|\in \mc E_{\mc G_{\mc E}}$. By definition, $|X|$ is not $0$-dimensional, and $|X|\cong |S^jZ|$ with $|Z|\in \mc E$. But if $|Z|\in \mc E$, so are all its suspensions unless $|Z|=\emptyset$. But if $|Z|=\emptyset$, $|X|$ is empty or a sphere of positive dimension and so  $|X|\in \mc E$. Thus $\mc E_{\mc G_{\mc E}}\subset \mc E$. 

Suppose $|X|\in \mc E$. Then we would have $|X|\in \mc E_{\mc G_{\mc E}}$ if  $|X|\cong |S^jZ|$ with $|Z|\in \mc E$. But from the definitions, there is no reason to assume that if $|S^jZ|$ is in $\mc E$ then so is $|Z|$. For example,  we can take $\mc E$ to be the class of pseudomanifolds $|X|$ such that either $|X|$ is empty or  there exists a $|Z|$ (depending on $|X|$) with $|X|\cong |SZ|$; i.e. $\mc E$ consists of all suspensions of pseudomanifolds. 
Then $\mc E$ satisfies the requirements to be a class of stratified pseudomanifold singularities, but $\mc G_{\mc E}$, by definition, becomes the spaces that are empty or suspensions of non-suspensions in $\mc E$. But since $\mc E$ consists entirely of suspensions or the empty set, any such object of $\mc G$ must then be a suspension of the empty set, i.e. a sphere.  Thus $\mc E_{\mc G_{\mc E}}$ is the set of spheres of dimensions $>0$ plus the empty set, and so $\mc E_{\mc G_{\mc E}}\neq \mc E$. If we let $\mc E'=\mc E_{\mc G_{\mc E}}$ in this particular example, then we observe $\mc G_{\mc E}=\mc G_{\mc E'}$ but $\mc E\neq \mc E'$.
\end{proof}

The first moral of  Lemma \ref{L: E and G} is that the class of pseudomanifold singularities and the class of stratified pseudomanifold singularities are closely related but not trivially so. However, Proposition \ref{P: pm links} tells us that when it comes to investigating classes of pseudomanifolds,
the links and polyhedral links are related by taking $\mc E$ to $\mc G_{\mc E}$, while Lemma \ref{L: E and G} tells us that every $\mc G$ is a $\mc G_{\mc E}$ for an $\mc E$ that depends rather simply on $\mc G$. In this sense, interesting classes of  pseudomanifolds can be described equivalently via an $\mc E$ or a $\mc G$. The advantage of working with classes of stratified pseudomanifold singularities is that there is one fewer condition to check as we require the class to be preserved under suspensions but not desuspensions. 

\begin{remark}\label{R: downside}
The downside is that the proposition and lemma together also imply that specifying a class $\mc E$ does not guarantee that all objects of $\mc E$ can occur as links of a stratified pseudomanifold. For example, consider the class  $\mc E$ mentioned in the proof of the lemma that is the class of pseudomanifolds that are either empty or can be written as $|SZ|$ for some compact PL space $|Z|$. We have seen that the associated $\mc G_{\mc E}$ consists just of spheres, so the class of pseudomanifolds with links in $\mc G_{\mc E}$ is the class of manifolds. If $|X|$ is a manifold with any stratification $X$, then every link suspends to a sphere, and so must be a sphere. Thus even though $\mc E$ includes objects that are not spheres, these objects cannot, in fact, occur as links in manifolds. So the disadvantage of  a class of stratified pseudomanifold singularities  is that it might not be easy to tell precisely which objects of $\mc E$ really do occur as links once we have restricted the class of links to $\mc E$. However, this is also an issue with $\mc G$, and with Akin's classes of singularities, in general. For example, we can define a class of pseudomanifold singularities $\mc G$ as the class of spaces homeomorphic to the empty set, spheres, or suspensions of $|Z|=|S^1\times ST^2|$. Since $|Z|$ is not a suspension, this is a class of pseudomanifold singularities, but any space $|X|$ with $|Z|$ as a polyhedral link contains points with neighborhoods homeomorphic to the cone on $|S^1\times ST^2|$. But then within such a neighborhood, there are points with neighborhoods homeomorphic to $|\R^2\times cT^2|$ and so a polyhedral link homeomorphic to $|S^1*T^2|\cong |S^2T^2|$. But this would imply that $|S^2T^2|$ and so $|T^2|$ are in $\mc G$, which is not the case. Therefore, $|Z|$ cannot be a link in the class of pseudomanifolds having objects of $\mc G$ as links. 
\end{remark}

\subsection{Pseudomanifold bordism classes}\label{S: F}

 By Lemma \ref{L: akin}, a class of pseudomanifold singularities is a class of singularities in the sense of 
 Akin \cite[Definition 8]{Ak75}, and so, continuing to follow \cite[Definition 8]{Ak75}, 
  we can define the class  $\mc F^n_{\mc G}$ to consist of $\emptyset$ and the \emph{totally $n$-dimensional} compact polyhedral pairs $(|X|,|X_0|)$ such that $|X_0|$ is collared in $|X|$ and such that $x\in |X|-|X_0|$ implies $|\Lk(x)|\in \mc G$. Here ``totally $n$-dimensional'' means that every simplex in some (and hence in every) triangulation of $|X|$ is a (not necessarily proper) face of an $n$-simplex; this is equivalent (see \cite[Definition 6]{Ak75}) to assuming that $|X|$ is the closure of a dense PL $n$-manifold.
  By \cite[Proposition 8]{Ak75},  $\mc F^n_{\mc G}$ is an example of what Akin calls  a \emph{(dimension graded) bordism sequence} associated to the class of singularities $\mc G$. 
   
    We let $\mc F_{\mc G}=\cup \mc F_{\mc G}^n$. This is a  different notation from Akin's, where an $\mc F$ without a superscript implies an ``ungraded bordism sequence.''  We will always care about dimensions of spaces for our bordisms, so we omit  this superscript unless we wish to emphasize consideration only of the $n$-dimensional spaces in the bordism sequence.
       Analogously as for weak bordism classes (see Remark \ref{R: SWS}), we let $S\mc F_{\mc G}$ denote the class consisting of the  orientable objects of $\mc F_{\mc G}$ with each of their orientations. 
  
  Our first order of business is to reinterpret these bordism sequences  in the language of pseudomanifolds. 

\begin{lemma}\label{L: akin is pm}
Let $\mc F_{\mc G}$ be a (graded) bordism sequence associated to a class of pseudomanifold singularities $\mc G$. Then the objects of $\mc F_{\mc G}$ are the pairs $(|X|,|\bd X|)$, where $|X|$ is a $\bd$-pseudomanifold  such that if $x\in |X|-|\bd X|$ then $|\Lk(x)|\in \mc G$. The analogous result holds for $S\mc F_{\mc G}$.
\end{lemma}
\begin{proof}
We prove the result in the unoriented case; the oriented case follows.

Suppose $|X|$ is a $\bd$-pseudomanifold $|X|$ such that if $x\in |X|-|\bd X|$ then $|\Lk(x)|\in \mc G$. Then $(|X|,|\bd X|)$ satisfies the conditions to be in $\mc F_{\mc G}$.

Conversely, suppose $(|X|,|X_0|)\in \mc F_{\mc G}$. Since every point of $x\in |Y|=|X|-|X_0|$ must  have an $n$-dimensional  neighborhood homeomorphic to a cone on a classical PL pseudomanifold, $|Y|$ is locally an $n$-dimensional pseudomanifold in a neighborhood of every point. In particular, this implies that $|Y|$ itself contains a dense $n$-manifold $|M|$ and that  $|Y-M|$ has dimension $\leq n-2$. So, as shown in \cite{GBF35}, $|Y|$ has an intrinsic stratification making it a PL stratified pseudomanifold and so $|Y|$ is a pseudomanifold. Furthermore, by the axioms of $\mc F_{\mc G}$ \cite[page 349]{Ak75},  $(|X_0|,\emptyset)\in \mc F_{\mc G} $, so $|X_0|$ is an $n-1$ dimensional pseudomanifold. Since $|X_0|$ is assumed to be collared in $|X|$, there is a neighborhood of $|X_0|$ of the form $[0,1)\times |X_0|$ with $|X_0|\subset |X|$ corresponding to $\{0\}\times |X_0|$. As shown in \cite{GBF35}, the intrinsic PL stratification on $(0,1)\times |X_0|$, which must coincide with the restriction of the intrinsic stratification of $|Y|$, must have the form $(0,1)\times X_0^*$, where $X_0^*$ is the intrinsic stratification on $|X_0|$. It follows that $|X|\cong |Y|\cup |X_0|$ can be stratified as a PL $\bd$-stratified pseudomanifold by gluing $Y^*$ and $[0,1)\times X_0^*$ compatibly along $(0,1)\times X_0^*$, and with this stratification, $\bd X\cong X_0$.  Thus $|X|$ is a  $\bd$-pseudomanifold with the desired properties. 
\end{proof}

\begin{remark}\label{R: pairs}
By Akin's definitions, elements of $\mc F_{\mc G}$ or $S\mc F_{\mc G}$ should really be considered pairs of spaces, but since we have shown in Lemma \ref{L: akin is pm} that the subspace in the pair must correspond to the unique classical pseudomanifold boundary in the cases we will consider, we are justified in thinking of the objects of $\mc F_{\mc G}$ as pseudomanifolds and using the notation $|X|$ for such objects.  

Notice that since all spheres are contained in $\mc G$, all manifolds will be contained in $\mc F_{\mc G}$ and all oriented manifolds in $S\mc F_{\mc G}$.
\end{remark}

\begin{definition}\label{D: pbc}
Although the classes $\mc F_{\mc G}$ are examples of bordism sequences in the sense of Akin, for the purposes of consistency within the present paper, we will call $\mc F_{\mc G}$ a \emph{pseudomanifold bordism class}. 
\end{definition}

\begin{remark}\label{R: PBC}
Although we have only considered classes of the form $\mc F_{\mc G}$ for some $\mc G$ and do not provide an independent more general definition of  ``pseudomanifold bordism class,'' it turns out that there is no loss of generality, as Akin shows in \cite[Proposition 9]{Ak75} that every bordism sequence is determined by a class of singularities. 
\end{remark}

If $\mc G$ has the form\footnotemark $\mc G_{\mc E}$, then we abbreviate $\mc F_{\mc G_{\mc E}}$ by $\mc F_{\mc E}$ and $S\mc F_{\mc G_{\mc E}}$ by $S\mc F_{\mc E}$.  In this case, the class $\mc F_{\mc E}$ consists of the underlying spaces in the IWS class associated to $\mc E$:

\footnotetext{Technically, every $\mc G$ has this form for some $\mc E$ by Lemma \ref{L: E and G}, so the choice of notation is more a matter of the emphasis within a given discussion.}

\begin{lemma}\label{L: some lemma}
If $\mc E$ is a class of stratified pseudomanifold singularities, then\footnote{Recall Definition \ref{D: underlying class}.} $|\mc C_{\mc E}|=\mc F_{\mc E}$ and $|S\mc C_{\mc E}|=S\mc F_{\mc E}$. 
\end{lemma}
\begin{proof}
Let $\mc C_{\mc E}$ be the IWS class with links in $\mc E$, and suppose $X\in \mc C_{\mc E}$ so that $|X|\in |\mc C_{\mc E}|$. Then by Proposition \ref{P: pm links}, $X$ has its polyhedral links of non-boundary points in $\mc G_{\mc E}$, so $|X|\in \mc F_{\mc E}$ by Lemma \ref{L: akin is pm}. Conversely, suppose $|X|\in \mc F_{\mc E}$. Then, by Lemma \ref{L: akin is pm}, the polyhedral links of non-boundary points of $|X|$ are in $\mc G_{\mc E}$, so by Proposition \ref{P: pm links}, any links of any stratification $X$ of $|X|$ are in $\mc E$. So $|X|\in |\mc C_{\mc E}|$. 

The oriented case follows. 
\end{proof}

\begin{remark}
It is not true that if $\mc C$ is any IWS class then $|\mc C|$ must correspond to some pseudomanifold bordism class $\mc F$. As observed in Remark \ref{R: PBC}, every pseudomanifold bordism class $\mc F$ has the form $\mc F=\mc F_{\mc G}$, and since $\mc G$ must contain spaces of arbitrarily large dimension, so must $\mc F$; for example, as $\mc G$ must contain all spheres, $\mc F$ must contain all manifolds.  We have already seen in Remark \ref{R: bad E} that this need not be true of an IWS class. 

This failure of general IWS classes to correspond to pseudomanifold bordism classes is in fact a feature of IWS classes, not a bug, as IWS classes were defined to provide stratified bordism \emph{groups} with certain properties, not stratified bordism homology theories. Development of bordism homology theories requires additional properties reflected in Akin's classes $\mc F$ and in our $\mc C_{\mc E}$; see Remark \ref{R: no}.
\end{remark}

\subsection{Unstratified bordism homology}\label{S: unstrat bord}  

Next, we turn to unstratified bordism as a homology theory based on classes $S\mc F_{\mc G}$.

By \cite[Propositions 7 and  8]{Ak75}, the pseudomanifold bordism class $S\mc F_{\mc G}$  yields  an unstratified bordism homology theory that we will denote $\Omega^{|\mc F_{\mc G}|}_*(\cdot)$, or simply $\Omega^{|\mc G|}_*(\cdot)$. 
Akin works only with unoriented bordism, but as observed by Siegel \cite[Section IV.1]{Si83}, it is easy to modify Akin's definition to get oriented theories. We will recall the definitions of these homology theories below. If $\mc G$ has the form $\mc G_{\mc E}$,  
we denote the resulting unstratified bordism homology theory by   $\Omega^{|\mc E|}_*(\cdot)$.
Note that even though the classes  $S\mc F_{\mc G}$ and $S\mc F_{\mc E}$ contain only unstratified pseudomanifolds, we use the notations such as  $\Omega^{|\mc E|}_*(\cdot)$ here to emphasize that we have an unstratified bordism theory. Later, we will use  notations such as $\Omega^{\mc E}_*(\cdot)$ for a stratified bordism homology theory.

Now, suppose $\mc G$ is a class of pseudomanifold singularities, so that $\mc F_{\mc G}$ is a pseudomanifold bordism class in the sense of Akin or that $S\mc F_{\mc G}$ is the corresponding oriented pseudomanifold bordism class. By Lemma \ref{L: E and G}, $\mc G$ has the form $\mc G_{\mc E}$, in this case for $\mc E=\mc E_{\mc G}$, so by Lemma \ref{L: some lemma}, $S\mc F_{\mc G}=|S\mc C_{\mc E}|$. 
Since $S\mc C_{\mc E}$ is an IWS class by Lemma \ref{L: prop P}, it follows from the constructions in Section \ref{S: bordism groups} that we have bordism groups  $\Omega_n^{|\mc F_{\mc G}|}$. The next lemma shows that these bordism groups correspond to evaluating the corresponding bordism homology theories at a point.

\begin{lemma}\label{L: akin pt is bordism}
If $\mc G$ is a class of pseudomanifold singularities and $\Omega^{|\mc G|}_*(\cdot)$ is the dimension-graded bordism homology theory of Akin \cite{Ak75} associated to $\mc G$, then $\Omega^{|\mc G|}_n(pt)\cong \Omega_n^{|\mc F_{\mc G}|}$, where $\Omega_n^{|\mc F_{\mc G}|}$ is an unstratified bordism group as defined in Section \ref{S: bordism groups}. In particular, $\Omega^{|\mc E|}_n(pt)\cong \Omega_n^{|\mc C_{\mc E}|}$.
\end{lemma}
\begin{proof}
Elements of $\Omega^{|\mc G|}_n(pt)$ are\footnote{This follows by taking  \cite[Definition 5]{Ak75}, restricting it to the case where Akin's $(T,T_0)$ is $(pt, \emptyset)$, applying Lemma \ref{L: akin is pm}, and  including orientations; note that $(T,T_0)\cong (pt, \emptyset)$ implies that Akin's space $W_1$ is empty.} equivalence classes of maps $f:|X|\to pt$, for $|X|$ an $n$-dimensional oriented pseudomanifold in $\mc F_{\mc G}$ with $|\bd X|=\emptyset$.
 Two such maps  $f:|X|\to pt$ and $g:|Y|\to pt$ are equivalent if there is an $n+1$ dimensional oriented pseudomanifold $|W|$ in $\mc F_{\mc G}$ with a PL orientation-preserving homeomorphism $i_{|X|}\amalg i_{|Y|}:|X|\amalg -|Y|\to |\bd W|$ and a map $F:|W|\to pt$ such that $f=Fi_{|X|}$ and $g=Fi_{|Y|}$. Since every space has a unique map to a point, the key thing for us to look at will be  how to interpret $i_{|X|}$ and $i_{|Y|}$ to fit our preceding definitions. 

We first observe that there is a homomorphism $\phi:\Omega^{|\mc G|}_n(pt)\to \Omega_n^{|\mc F_{\mc G}|}$ that takes $\{f:|X|\to pt\}\in \Omega^{|\mc G|}_n(pt)$ to the class in $\Omega_n^{|\mc F_{\mc G}|}$ represented by $|X|$. To see that this is well-defined, suppose $[f:|X|\to pt]=[g:|Y|\to pt]\in \Omega^{|\mc G|}_n(pt)$. By definition, this relation provides a bordism $|W|$ in $\mc F_{\mc G}$ such that $|\bd W|=|X|\amalg -|Y|$. Hence, $\phi$ is well defined. It is also certainly surjective since, given any $|X|\in \mc F_{\mc G}$, we can construct a unique $f:|X|\to pt$. Lastly, suppose $\phi([f:|X|\to pt])=\phi([g:|Y|\to pt])$, meaning that $|X|$ and $|Y|$ are bordant. Thus there is a bordism $|W|\in \mc F_{\mc G}$ with $|\bd W|\cong |X|\amalg -|Y|$. Implicit in this statement is the existence of some PL orientation-preserving homeomorphism $i_{|X|}\amalg i_{|Y|}:|X|\amalg -|Y|\to |\bd W|$. But this is sufficient to show that $[f:|X|\to pt]=[g:|Y|\to pt]\in \Omega^{|\mc G|}_n(pt)$ since if $F:|W|\to pt$ is the unique map, then $f=Fi_{|X|}$ and $g=Fi_{|Y|}$ for \emph{any} choices of embeddings $i_{|X|}$ and $i_{|Y|}$. 

The final claim of the lemma now follows from Lemma \ref{L: some lemma}.
\end{proof}

\begin{remark}
Once again, an unoriented version of Lemma \ref{L: akin pt is bordism} follows by ignoring orientations.
\end{remark}

\subsection{Stratified bordism as a homology theory}\label{S: theory}
We turn to stratified bordism as a homology theory. As noted in Remark \ref{R: no}, it seems likely that one could follow Akin and define a notion of stratified bordism homology theories directly and then try to verify that corresponding stratified and unstratified bordism homology theories are isomorphic for  IWS classes. The following approach seems much simpler and will suffice for the classes $S\mc C_{\mc E}$: Given a class of stratified pseudomanifold singularities $\mc E$, there is a bordism homology theory $\Omega^{|\mc E|}_*(\cdot)$  as constructed by Akin \cite{Ak75} in the unoriented case and easily generalized to oriented bordism. We will introduce $\Omega_*^{\mc E}(\cdot)$ as a  functor that admits boundary maps and long exact sequences, construct a natural transformation of functors  $\mf s: \Omega_*^{\mc E}(\cdot)\to \Omega_*^{|\mc E|}(\cdot)$,  and show that it is an isomorphism of functors. Thus $\Omega_*^{\mc E}(\cdot)$  must itself  be a  homology theory via the isomorphism. 

We must first show that there is some functor $\Omega_*^{\mc E}(\cdot)$ built upon the collection $S\mc C_{\mc E}$ of oriented PL stratified $\bd$-pseudomanifolds with links in $\mc E$. 
 As above for $\bd$-pseudomanifolds (see Remark \ref{R: pairs}), we replace Akin's pairs $(X,X_0)$ with pairs $(X,\bd X)$, which we will denote by $X$ alone. In order to define a bordism theory from such a class of spaces, we can follow \cite[Definition 5]{Ak75}, suitably modified to include orientation and stratification information. We note that only Axioms 1, 2, and 4 on page 349 of \cite{Ak75} (and which we will provide below) are needed in order to construct a  bordism functor $\Omega^{\mc E}_n(\cdot)$ on pairs of spaces $(T,T_0)$  
such that

\begin{enumerate}
\item  there is a map $\bd: \Omega_n^{\mc E}(T,T_0)\to \Omega_{n-1}^{\mc E}(T_0)$ such that if  $g:(T,T_0)\to (R,R_0)$, the following diagram commutes:
\begin{diagram}
\Omega_n^{\mc E}(T,T_0)&\rTo^{\bd}& \Omega_{n-1}^{\mc E}(T_0)\\
\dTo^{g_*}&&\dTo^{(g|_{T_0})_*}\\
\Omega_n^{\mc E}(R,R_0)&\rTo^{\bd}& \Omega_{n-1}^{\mc E}(R_0),
\end{diagram}
\item there are long exact sequences
\begin{diagram}
&\rTo&\Omega_n^{\mc E}(T_0)&\rTo& \Omega_n^{\mc E}(T)&\rTo &\Omega_n^{\mc E}(T,T_0)&\rTo^{\bd} &\Omega_{n-1}^{\mc E}(T_0)&\rTo&
\end{diagram}
with $\Omega_n^{\mc E}(T_0)\to \Omega_n^{\mc E}(T)$ and $\Omega_n^{\mc E}(T)\to \Omega_n^{\mc E}(T, T_0)$ induced by the inclusions $(T_0,\emptyset)\into (T,\emptyset)$ and $(T,\emptyset)\into (T,T_0)$.
\end{enumerate}

We now state versions of  Akin's axioms 1, 2, and 4, modified to take into account orientations and our focus on stratified pseudomanifolds\footnotemark. We will denote a class of spaces satisfying these axioms by $S\mc B\subset S\Psi$. We could refer to $S\mc B$ as an \emph{oriented stratified bordism functor class.}:

\footnotetext{Akin's axiom 3, which we will not need a version of, concerns cutting out a regular neighborhood of a subspace. It is needed to prove the excision property for bordism homology theories. In our case, excision will follow from the isomorphism between  $\Omega^{\mc E}_*(\cdot)$ and $\Omega^{|\mc E|}_*(\cdot)$, which possesses excision by Akin's theory.}

\begin{description}

\item[Axiom 1.]  If $X\in S\mc B$ and\footnote{Throughout these axioms, $\cong$ denotes \emph{stratified} homeomorphism.} $Y\cong X$ or $Y\cong -X$, then $Y\in S\mc B$. 

\item[Axiom 2.] If $X\in S\mc B$, then $\bd X\in S\mc B$ and $I\times X\in S\mc B$.  

\item[Axiom 4.] Suppose that $X,Y\in S\mc B$ are such that $\bd X\cong Z\cup X_1$,   $\bd Y\cong -(Z\cup Y_1)$, and  $X_1,Y_1,Z\in S\mc B$, and that, with appropriate orientations that we suppress,  $Z\cap X_1\cong \bd Z\cong \bd X_1$ and $Z\cap Y_1\cong \bd Z\cong \bd Y_1$. Then  $X\cup_Z Y\in S\mc B$ for any gluings of $X$ and $Y$ along $Z$ compatible with the orientation information. 
\end{description}

\begin{remark}
The difference between our versions of the axioms and Akin's is only that Akin does not take into account stratification or orientation information, and he allows general pairs of compact polyhedra $(|X|,|X_0|)$ rather than our $\bd$-stratified pseudomanifolds $(X,\bd X)$.  
\end{remark}

Given a class of  spaces $S\mc B\subset S\Psi$ that satisfying these versions of Akin's axioms 1, 2, and 4, we recall, following \cite[Definition 5]{Ak75}, how to construct a bordism functor $\Omega_n^{\mc B}(\cdot)$ satisfying the properties required above. We do not claim that $\Omega_n^{\mc B}(\cdot)$ is a homology theory, though this will follow later for functors of the form $\Omega_n^{\mc E}(\cdot)=\Omega_n^{\mc C_{\mc E}}(\cdot)$.

\begin{definition}
\begin{enumerate}
\item Let $X,Y\in S\mc B^n$, where $S\mc B^n$ comprises those spaces of $S\mc B$ of dimension $n$. Then a \emph{cobordism} between\footnote{Notice that we do not require here that $\bd X$ or $\bd Y$ is empty, so this is a more general notion than that of Section \ref{S: all abord}.} $X$ and $Y$ is a quadruple\footnote{Akin's bordisms in \cite{Ak75} are quintuples, but recall that for classical $\bd$-pseudomanifolds we have $W_0=\bd W$ implicit in $W$.} $(W, W_1; i_X, i_Y)$ with $W_1\subset \bd W$ and $i_X:X\to \bd W$ and $i_Y:Y\to \bd W$ PL stratified embeddings such that 

\begin{enumerate}
\item $i_X(X)\cap W_1=i_X(\bd X)$, $i_Y(Y)\cap W_1=i_Y(\bd Y)$, $i_X(X)\cap i_Y(Y)=\emptyset$, and $\bd W=i_X(X)\cup W_1\cup -i_Y(Y)$, and
\item $W\in \mc B^{n+1}$, $W_1\in S \mc B^n$, and $\bd W_1=-i_X(\bd X)\cup i_Y(\bd Y)$.
\end{enumerate}

\item  Let $X,Y\in S\mc B^n$ and $f:(X,\bd X)\to (T,T_0)$ and $g:(Y,\bd Y)\to (T,T_0)$ be continuous maps to a topological pair $(T,T_0)$. Then a cobordism between $f$ and $g$ consists of a cobordism $(W, W_1; i_X, i_Y)$ between $X$ and $Y$ and a continuous map $F:(W,W_1)\to (T,T_0)$ that  $Fi_X=f$ and $Fi_Y=g$. 
\end{enumerate}
\end{definition}

The cobordism relations are reflexive by Axiom 2, transitive by Axiom 4, and symmetric by reversing the orientation of a bordism. We let  $\Omega_n^{\mc B}(T, T_0)$ be the set of bordism equivalence classes of continuous maps $f:(X,\bd X)\to (T,T_0)$. We can denote equivalences classes by $[f:(X,\bd X)\to (T,T_0)]$, or simply $[f]$. 
The set  $\Omega_n^{\mc B}(T, T_0)$ is a group under disjoint union, with the identity
being represented by a unique empty map  $\emptyset:(\emptyset, \emptyset)\to (T,T_0)$. The inverse of the class $[f:(X,\bd X)\to (T,T_0)]$  is   $[-f:(-X,-\bd X)\to (T,T_0)]$, where $-f$ is the map obtained from $f:(X,\bd X)\to (T,T_0)$ by reversing the orientation of $X$; a bordism $F: (I\times X, ((\bd I)\times X)\cup (I\times \bd X))\to (T,T_0)$ is obtained by composing the projection of $I\times X$ to $X$ with $f$. The assignment $(T,T_0)\to \Omega_n^{\mc B}(T, T_0)$ is a functor by composition of maps, i.e. if $g: (T,T_0)\to (R,R_0)$, then $g_*$ takes $[f:(X,\bd X)\to (T,T_0)]$ to $[gf: (X,\bd X)\to (R,R_0)]$. The map $\bd :\Omega_n^{\mc B}(T, T_0)\to \Omega_n^{\mc B}(T_0)$ is defined by restricting $[f:(X,\bd X)\to (T,T_0)]$ to $[f|_{\bd X}:(\bd X, \emptyset)\to (T_0, \emptyset)]$. Evidently $\bd g_*=(g|_{T_0})_*\bd$. 
The existence of long exact sequences is proven as in the  proof of  \cite[Proposition 7]{Ak75} using the expected arguments; note that Axiom 3 is not needed.

Next, we show that $S\mc C_{\mc E}$ satisfies the axioms necessary to be considered a class $S\mc B$.

\begin{lemma}\label{L: axioms}
Each $S\mc C_{\mc E}$  satisfies our axioms 1, 2, and 4 to be an $S\mc B$.  
Each $S\mc F_{\mc E}=|SC_{\mc E}|$ satisfies all four of the original axioms of  Akin's from \cite{Ak75}, suitably modified to include orientations. 
\end{lemma}
\begin{proof}
$S\mc F_{\mc E}$ is defined from a class of pseudomanifold singularities, which, by Lemma \ref{L: akin}, is a class of singularities in the sense of \cite{Ak75}. So $S\mc F_{\mc E}$ satisfies Akin's four axioms by oriented versions of \cite[Proposition 8]{Ak75} and the arguments on page 352 and 353 of \cite{Ak75} regarding collars. 

For $S\mc C_{\mc E}$, we check directly that the axioms are satisfied:

Axiom 1 calls for $Y\in S\mc C_{\mc E}$ if $X\in S\mc C_{\mc E}$ and $Y\cong X$ or $Y\cong -X$. This follows from $S\mc C_{\mc E}$ being an oriented IWS class.

Axiom 2 requires that $X\in S\mc C_{\mc E}$ implies $\bd X\in S\mc C_{\mc E}$ and $I\times X\in  S\mc C_{\mc E}$. The first of these is clear because any link in $\bd X$ must also be a link in $X$. Similarly, the links in $I\times X$ are the same as the links in $X$. We must still show, however, that $\bd(I\times X)= ((\bd I)\times X)\cup (I\times X_0))$ is stratified collared in $I\times X$. We discuss this below.

For Axiom 4, translated into our language, we must suppose we have two $n$-dimensional oriented PL stratified $\bd$-pseudomanifolds $X,Y\in S\mc C_{\mc E}$ such that $\bd X\cong Z\cup X_1$,   $\bd Y\cong -(Z\cup Y_1)$,  $X_1,Y_1,Z\in S\mc C_{\mc E}$, and, with appropriate orientations that we suppress,  $Z\cap X_1\cong \bd Z\cong \bd X_1$ and $Z\cap Y_1\cong \bd Z\cong \bd Y_1$. The axiom then states that $X\cup_Z Y\in SC_{\mc E}$. There is no trouble here verifying the link condition in the new space $X\cup_Z Y$, again assuming that all PL homeomorphisms are stratified homeomorphisms. However, again, we must take care to show the boundary $X_1\cup_{\bd Z}Y_1$ of $X\cup_Z Y$ is collared. 

We now turn to a discussion of collars. Let us begin with the space $I\times X$ whose boundary is the union of $\{0\}\times X$, $I\times \bd X$, and $\{1\}\times X$. 
Since $\bd X$ is collared in $X$, there is a neighborhood $U$ of $\bd X$ in $X$ that we can identify (including stratification) with $[0,1/4)\times \bd X$. Technically, there is a PL homeomorphism involved, but, via that homeomorphism, we will treat $U$ as if it is identically $[0,1/4)\times \bd X$. So $\bd(I\times X)$ has a neighborhood $V$ in $I\times X$ stratified homeomorphic to 
$$V=([0,1/4)\times X)\cup (I\times U)\cup ((3/4,1]\times X).$$ We will show that we can unfold $V$ to a collar of $\bd(I\times X)$ via a PL homeomorphism.

\begin{figure}[h!]
  \centering
\scalebox{.25}{\includegraphics{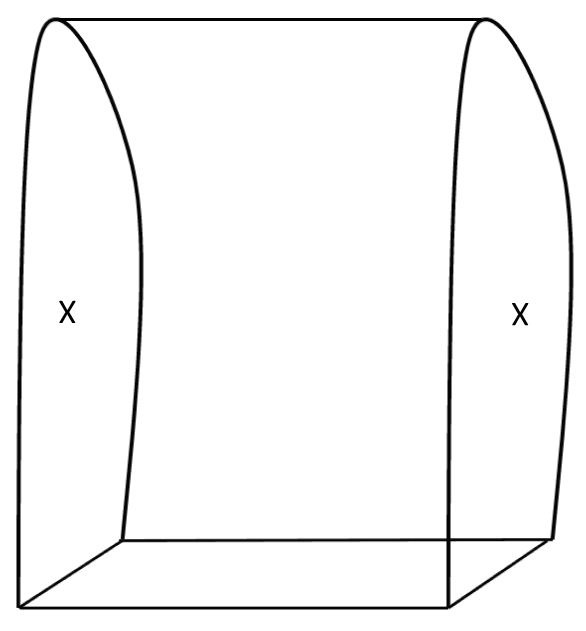}}
\caption{$I\times X$}
\end{figure}

\begin{figure}[h!]
\label{F: unfold}  \centering
\scalebox{.25}{\includegraphics{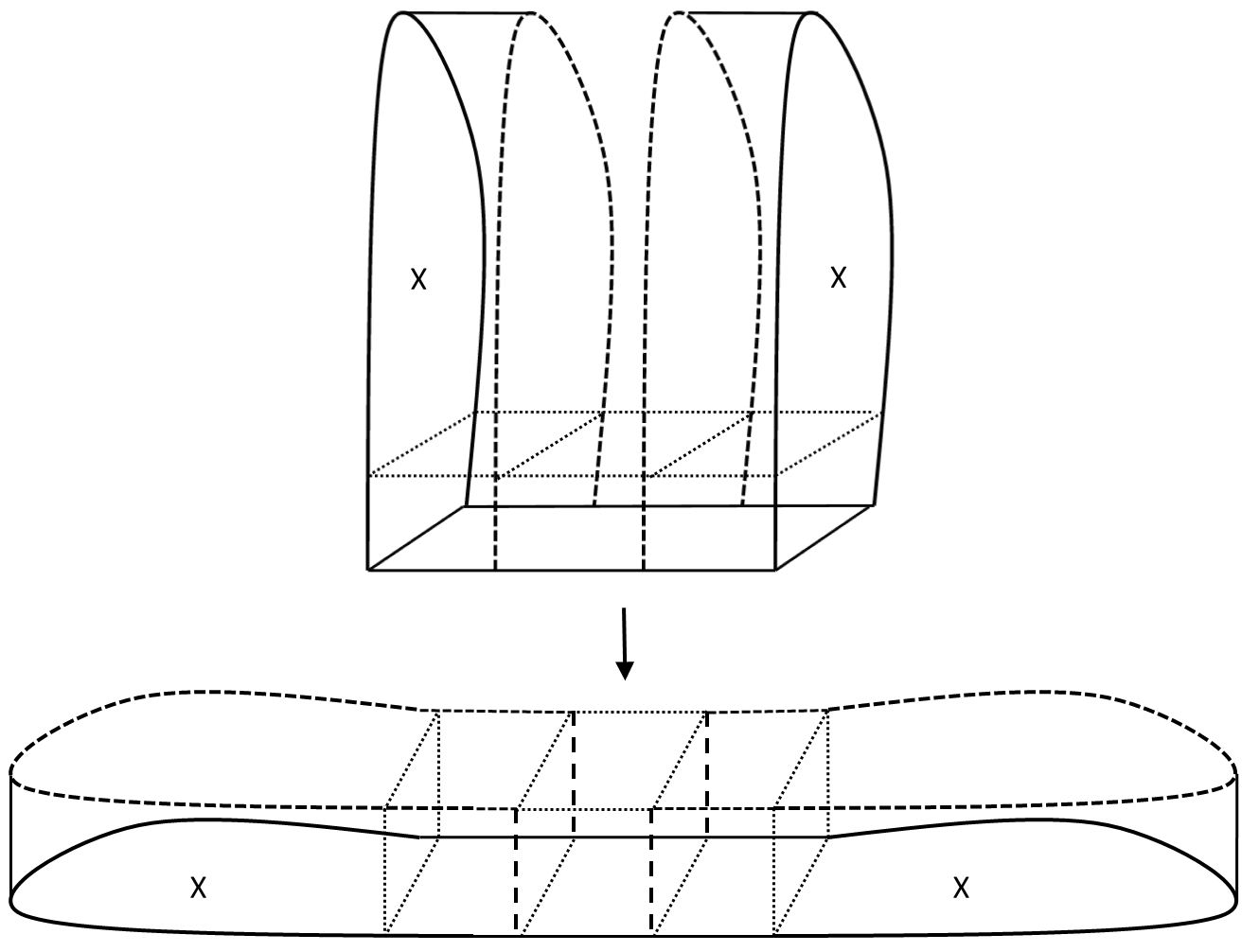}}
\caption{Unfolding a neighborhood of $\bd(I\times X)$}
\end{figure}

Notice that   $$W:=([0,1/4)\times X)\cap (I\times U)\cong [0,1/4)\times U\cong [0,1/4)\times [0,1/4)\times \bd X,$$
which can be seen in the bottom left corner of the top part of Figure \ref{F: unfold}.
The closure of $W$ in $V$ is the product of $\bd X$ with $([0,1/4]\times [0,1/4])-\{(1/4,1/4)\}$. But there is a PL homeomorphism 
$$([0,1/4]\times [0,1/4])-\{(1/4,1/4)\}\cong[-1,1]\times  [0,1/4)$$
that takes $\{1/4\}\times [0,1/4)$ to $\{1\}\times [0,1/4)$ and $[0,1/4)\times \{1/4\}$ to $\{-1\}\times [0,1/4)$. Such a PL homeomorphism is given by the simplicial map suggested in the following figure:

\begin{figure}[h!]
  \centering
\scalebox{.25}{\includegraphics{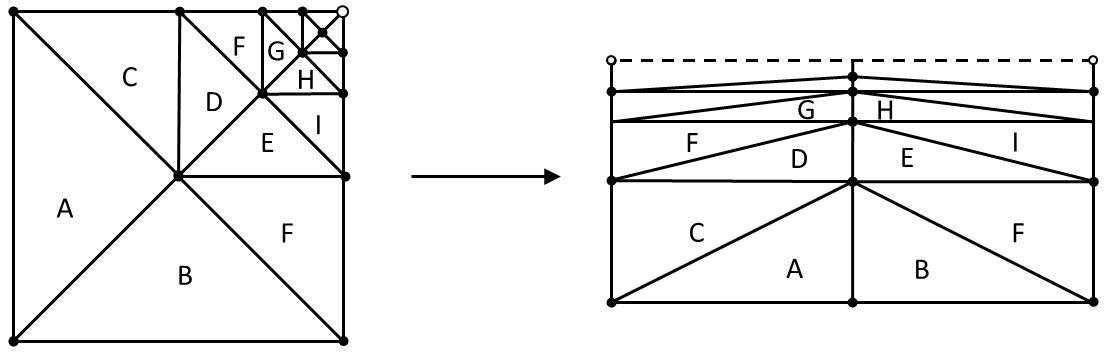}}
\caption{A simplicial map indicating the unfolding of a square with the upper right corner deleted. Labels indicate corresponding $2$-simplices under the map.}
\end{figure}

Thus, by a PL homeomorphism,  we can ``straighten out'' the closure of $W$ in $V$  to  $[-1,1]\times [0,1/4) \times \bd X$. We can then glue 
$$(([0,1/4)\times X)\cup (I\times U))-W\cong ([0,1/4)\times (X-U))\amalg ([1/4,1]\times U)$$ to $ [-1,1]\times[0,1/4)\times \bd X$ along $(\{-1\}\times[0,1/4)\times  \bd X)\amalg(\{1\}\times[0,1/4)\times  \bd X)$, thus ``straightening the corner'' of the collar. The same procedure can be applied at $((3/4,1]\times X)\cap (I\times U)$, so that $\bd (I\times X)$ is collared.

Next, consider $X\cup_Z Y$ as in Axiom 4. By the given assumptions (and ignoring orientations in the notation), $\bd Z=\bd X_1=\bd Y_1$ is bicollared in $\bd X$ and in $\bd Y$. Let $U$ be the collar of $\bd Z$ in $Z$; we are free to assume that $ U\cong [0,1)\times \bd Z$. Let $N$ be the restriction of the collar of $\bd X$ in $X$ to $X_1\cup  U$, so $N\cong [0,1)\times (X_1\cup U)$. Then using the  bicollar on $\bd Z$, $\bd Z$ has a neighborhood in $X$ PL homeomorphic to  $[0,1)\times (-1,1]\times \bd Z$. Reversing the PL homeomorphism of the last paragraph (and omitting one side of the box), we get a PL homeomorphism from $N$ to $[0,1)\times X_1$, with $U$ being taken to $[0,1)\times \bd X_1$. Performing a similar homeomorphism on a neighborhood of $Y_1$ in $Y$ and then gluing these two homeomorphisms together shows that $X_1\cup_{\bd Z} Y_1$ has a stratified collaring in $X\cup_Z Y$.

\begin{figure}[h!]
  \centering
\scalebox{.35}{\includegraphics{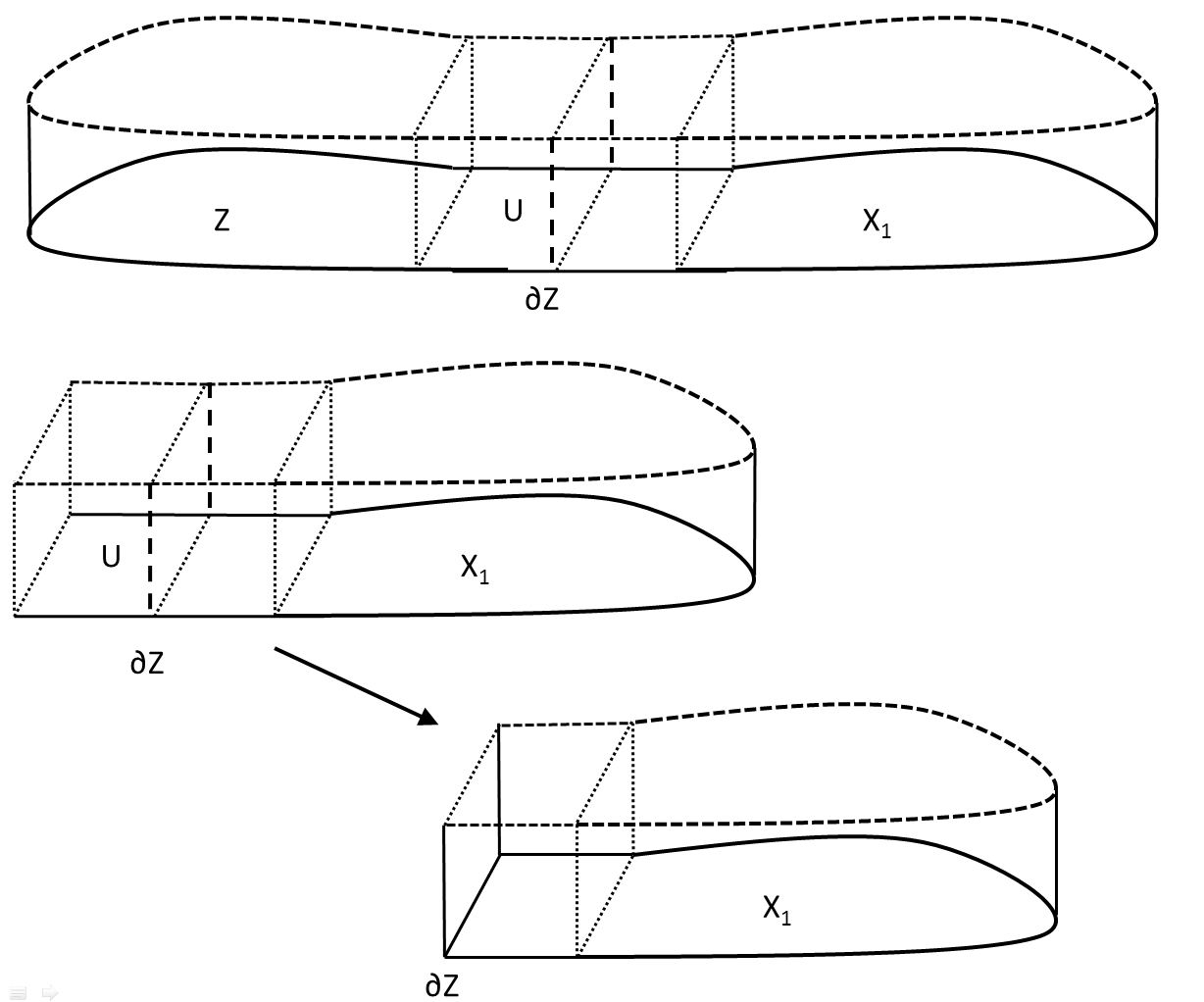}}
\caption{The top figure represents a neighborhood of $\bd X$ in $X$. Note that $U\subset Z$ and the unlabeled rectangle in $\bd X$ is part of $X_1$. The bottom figures illustrate taking $N$ and folding a corner to form a collar of $X_1$. This collar is then glued to a similar collar of $Y_1$ to form a collar neighborhood of $\bd(X\cup_Z Y)$.  }
\end{figure}
\end{proof}

Lemma \ref{L: axioms}, together with our construction of bordism functors based on classes $S\mc B$, provides a stratified bordism functor based on $S\mc C_{\mc E}$ that we can denote $\Omega^{\mc E}_*(\cdot)$. The analogous $\Omega^{|\mc E|}_*(\cdot)$  follows from Akin \cite{Ak75}. 
Furthermore, when $S\mc B=S\mc C_{\mc E}$, there is  a natural  transformation of functors $\mf s: \Omega_*^{\mc E}(\cdot)\to \Omega_*^{|\mc E|}(\cdot)$ defined by taking $[f:(X,\bd X)\to (T,T_0)]$ to the element of  $\Omega_*^{|\mc E|}(T,T_0)$ represented by the underlying map $[f: (|X|,|\bd X|)\to (T,T_0)]$. Stratified bordisms have underlying unstratified bordisms, so this map is well-defined; it is evidently a natural transformation  and commutes with $\bd_*$. We show that $\mf s$ is an isomorphism:

\begin{theorem}\label{T: theory}
The natural transformation $\mf s: \Omega_*^{\mc E}(\cdot)\to \Omega_*^{|\mc E|}(\cdot)$  is an   isomorphism of homology theories.  
\end{theorem}
\begin{proof}
The natural transformation $\mf s$ commutes with the $\bd$ maps, and so induces maps of long exact sequences. 
We will show that $\mf s$ induces isomorphisms $\mf s:\Omega_*^{\mc E}(T)\to \Omega_*^{|\mc E|}(T)$ for any $T$, and it will follow that $\mf s$ induces isomorphisms on bordism groups of pairs by the five lemma. The other properties of a homology theory will follow for $\Omega_*^{\mc E}(\cdot)$ via the homeomorphisms and the naturality of $\mf s$. 

To see that $\mf s:\Omega_*^{\mc E}(T)\to \Omega_*^{|\mc E|}(T)$ is an isomorphism, first suppose $[f:|X|\to T]\in \Omega_*^{|\mc E|}(T)$. Since $|\bd X|$ maps to the empty set, $|\bd X|=\emptyset$, so we can choose the intrinsic stratification $X^*$ of $|X|$ to obtain an element $[f:X^*\to T]\in \Omega_*^{\mc E}(T)$, represented by the same topological map $f$,  such that $\mf s([f:X^*\to T])=[f:|X|\to T]$.  Therefore, $\mf s$ is surjective. 

Next, suppose $\mf s([f:X \to T])=\mf s([g: Y\to T])\in  \Omega_*^{|\mc E|}(T)$. Then we must have $\bd X=\bd Y=\emptyset$ as above and  there is some bordism $F: |W|\to T$, with $|W|\in \mc F_{\mc E}$, between $f:|X|\to T$ and $g:|Y|\to T$. By Corollary \ref{C: given bord}, it is possible to stratify $|W|$ to some $W$ to obtain a stratified bordism with the same underlying topological map $F:W\to T$ so that $F$ is a  stratified bordism between $f$ and $g$. Since $|W|\in \mc F_{\mc E}$, $W\in \mc C_E$ by Proposition \ref{P: pm links}. Therefore, $[f:X \to T]=[g: Y\to T]\in \Omega_*^{\mc E}(T)$, so $\mf s$ is injective.
\end{proof}

\begin{remark}
Once more, an unoriented version of Theorem \ref{T: theory} follows by ignoring orientations in the preceding arguments.
\end{remark}

\section{Siegel classes}\label{S: Siegel}

In this final section, motivated again by \cite{Si83}, we explore one further aspect of  the construction of classes of spaces suitable for bordism theories. In particular, we  examine how our use of Akin's bordism homology groups in our constructions of the previous section relate to Siegel's construction of Witt bordism as a homology theory in \cite{Si83}, which is somewhat different. This leads us to the notion of Siegel classes, which are defined in terms of the  ``second order link properties'' mentioned in the introduction to this paper but which provide, from one point of view, a more efficient way of generating pseudomanifold bordism classes.

In \cite[Section IV]{Si83}, Siegel defines oriented $\Q$-Witt bordism as a homology theory to be the oriented version of Akin's \cite{Ak75} bordism theory based on the class of singularities $\mc L$ such that $|Z|\in \mc L$ if
\begin{enumerate}
\item $I^{\bar m}H_{\dim(|Z|)/2}(|Z|;\Q)=0$ if $\dim(|Z|)$ is even dimensional,

\item $|Z|$ is a compact orientable\footnote{\label{F: Siegel}In \cite[Section IV]{Si83}, Siegel uses  ``oriented'' rather than ``orientable,'' but we argue that, even when considering oriented bordism theories, specific orientations do not need to be assumed at the level of links. We first observe that it is true that if $X$ is an oriented (or orientable) pseudomanifold, then the links of any stratification of $X$, and hence the polyhedral links, which are the suspensions of links of the intrinsic stratification, must be orientable. However, choosing a specific orientation of $X$ itself, without necessarily choosing any orientations of the singular strata of $X$, does not lead to a natural choice of orientation for the links (or polyhedral links) of $X$. But this is not a problem in altering Akin's theory to account for oriented bordism, as it is only the orientations of the global spaces that need to be accounted for in the definitions of oriented bordism given above. For the purposes of specifying as precisely as possible the allowable links for, say, oriented $\Q$-Witt spaces, it is only necessary to specify the vanishing condition and the orientability, not specific link orientations. As we hope the reader will agree, the justifiability of this modification will be borne out in the remainder of this section.} $\Q$-Witt space (without boundary\footnotemark).
\end{enumerate}

\footnotetext{Siegel does not explicitly say ``without boundary'' but this is implicit from the context.}

Siegel claims without proof  that $\mc L$ is a class of singularities in the sense of Akin, and so yields a pseudomanifold bordism theory with such spaces as the polyhedral links of non-boundary points. There turns out to be a minor error in this claim. Siegel's $\mc L$ does not contain $|S^0|$, for which $I^{\bar m}H_{0}(|S^0|;\Q)=\Q\oplus \Q$, but any Akin class of singularities must contain all spheres, including $|S^0|$. However, if we modify $\mc L$ to a class $\mc L'=\mc L\cup\{|S^0|\}$, then we do get a class of spaces that is in fact a class of pseudomanifold singularities such that $\Omega_*^{|\mc L'|}$ is $\Q$-Witt bordism, as we will show below. We can also remove Siegel's orientability assumption and obtain a class of pseudomanifold singularities for unoriented Witt bordism.

It is interesting that the conditions on $\mc L'$ are somewhat different from the defining conditions of $\Q$-Witt spaces that are used throughout  \cite{Si83} prior to Section IV. 
Letting $\mc E=\mc E_{G-\text{Witt}}$, notice that  being in $\mc F_{\mc E}$ depends on what might be called a ``first order link property'' of polyhedral links: by definition, $|X|\in \mc F_{\mc E}$ if and only if its polyhedral links are in $\mc G_{\mc E}$. By contrast, the conditions of $\mc L'$  involve what we might call ``second order link properties'' of polyhedral links: a space $|X|$ is contained in $\mc F_{\mc L'}$ if the polyhedral links $|\Lk(x)|$ of dimension $\neq 0$ are in $\mc G_{\mc E}$ (by Siegel's first condition) \emph{and} the polyhedral links of $|Lk(x)|$ are in $\mc G_{\mc E}$ (by Siegel's second condition, which is itself a condition on links of $|Lk(x)|$!). Such second order link properties arise elsewhere in the literature; for example we have already seen such recursive issues arise in our discussion of LSF spaces in Section \ref{E: examples}. This is perhaps not too surprising as it is well known that ``a link of a link is a link,'' i.e. if $X$ is a stratified pseudomanifold and $L$ is one of its links, then any link of a stratum in $L$ is also a link in $X$ by Remark \ref{R: link link link}.

Emulating Siegel's definition of $\mc L$, we make the following definition:

\begin{definition}\label{D: siegel}
Given a class of stratified pseudomanifold singularities $\mc E$, we define the  \emph{Siegel class}  $\mc S_{\mc E}\subset |\Psi|$ to be the class of all $|X|\in |\Psi|$ that  satisfy the following conditions
\begin{enumerate}
\item $|\bd X|=\emptyset$,

\item   if $\dim(|X|)=0$, then $|X|\cong |S^0|$, 

\item \label{I: 1}  if $\dim(|X|)>0$, then $|X|\in \mc E$, 

\item \label{I: 2} for any stratification of $|X|$ as a classical stratified pseudomanifold, its links are in $\mc E$ (or equivalently, by Proposition \ref{P: pm links}, all of the polyhedral links of $|X|$ are in $\mc G_{\mc E}$).
\end{enumerate}
\end{definition}

The last condition of the definition is what we have referred to as a ``second order linking property'' since, if $\mc S_{\mc E}$ is meant to be a class of polyhedral links of pseudomanifolds (see the following lemma), then condition \eqref{I: 2} is a condition on ``links of links.'' 

\begin{lemma}\label{L: siegel}
If $\mc S_{\mc E}$ is a Siegel class, then

\begin{enumerate}
\item $\mc S_{\mc E}$ is a class of pseudomanifold singularities, and
\item $\mc S_{\mc E}\subset \mc G_{\mc E}$.
\end{enumerate}
\end{lemma}
\begin{proof}
We will show that $\mc S=\mc S_{\mc E}$ is a class of pseudomanifold singularities; the verification that $\mc S_{\mc E}\subset \mc G_{\mc E}$ will be included as part of this argument. By assumption, if $|X|\in \mc S$, then $|\bd X|=\emptyset$, and $\emptyset$ itself satisfies the conditions to be in $\mc S$. The conditions of the definition are also stratification-independent, so, if $|X|\in \mc S$ and $|Y|\cong |X|$, then $|Y|\in \mc S$. 

Now, suppose $|X|\in \mc S$ and consider $|SX|$. First suppose $\dim(|X|)>0$. As $|X|\in \mc E$, also $|SX|\in\mc E$. Now, each polyhedral link of $|SX|$ is either  homeomorphic to $|X|$ or homeomorphic to $|S\Lk|$, where $|\Lk|$ is a polyhedral link in $X$. By assumption, each $|\Lk|$ is in $\mc G_{\mc E}$, and hence so is $|S\Lk|$ by the properties of $\mc G_{\mc E}$.  For $|X|$ itself, we can write $|X|\cong |S^kZ|$ for some $k\geq 0$ and some pseudomanifold $|Z|$ that is not a suspension of a  pseudomanifold. If $k=0$, then $|X|\cong |Z|\in \mc E$, so $|X|\in \mc G_{\mc E}$ by Lemma \ref{L: e to g}. If $k>0$ and we choose any stratification $Z$ of $|Z|$, then $Z$ is a link of $S^kZ$, which is a stratification for $|X|$. Therefore, $|Z|\in \mc E$ (as $|X|\in \mc S$), so again $|X|\in \mc G_{\mc E}$ by Lemma \ref{L: e to g}. It follows that all of the polyhedral links of $|SX|$ are in $\mc G_{\mc E}$. So $|SX|\in \mc S$. If $\dim(|X|)=0$, then $|X|\cong |S^0|$, so $|SX|\cong |S^1|$, which also satisfies all the properties to be in $\mc S$ (and in $\mc G_{\mc E}$). So we have shown that for $|X|\in \mc S$, we have $|SX|\in \mc S$. We have also shown in this paragraph that if $|X|\in \mc S$ then $|X|\in \mc G_{\mc E}$, which demonstrates the second claim of the lemma.

Next, suppose $|SX|\in \mc S$ for a compact classical pseudomanifold $|X|$. Then we must have $|\bd X|=\emptyset$. If $\dim(|X|)=0$, then $|SX|$ must be $|S^1|$, since no other suspension of a $0$-dimensional pseudomanifold is a classical pseudomanifold. This implies that $|X|\cong |S^0|\in \mc S$. Now assume $\dim(|X|)>0$. Since $X$ is a link of $SX$, $|X|\in \mc E$ by property \eqref{I: 2} of the definition of $\mc S$. Furthermore, if $L$ is a link in a stratification of $X$, then $L$ is also a link in the stratification of $SX$, so each such link is in $\mc E$. Therefore, $|X|\in \mc S$. 

This concludes the verification that $\mc S$ is a class of pseudomanifold singularities.
\end{proof}

\begin{remark}\label{R: g not equal s}
Lemma \ref{L: siegel} shows that $\mc S_{\mc E}\subset \mc G_{\mc E}$. In general, it will not be the case that $\mc S_{\mc E}= \mc G_{\mc E}$. In particular, given $\mc E$, it is quite reasonable for there to be  a compact classical pseudomanifold $|X|$ that is not a suspension of a pseudomanifold such that $|X|\in \mc E$ (and so $|X|\in \mc G_{\mc E}$) but not all the links of $X$ are in $\mc E$. For example, let $\mc E=\mc E_{\Q-\text{Witt}}$, and let $X=S^1\times S^{2k}\times S\C P^2$ for some large value of $k$. This $X$ is not a suspension, and, applying the appropriate K\"unneth theorem (see \cite{Ki, GBF35}), the middle-dimensional lower-middle perversity intersection homology of $X$ vanishes. Thus  $X\in \mc G_{\mc E}$. But $X$ has $\C P^2$ as a link, and $\C P^2\notin \mc E$. 
\end{remark}

Given Remark \ref{R: g not equal s}, the following proposition  is somewhat surprising. Since, in general, $\mc S_{\mc E}\subsetneqq \mc G_{\mc E}$, this proposition demonstrates that a Siegel class $\mc S_{\mc E}$ can be in some sense more efficient than the class of pseudomanifold singularities $\mc G_{\mc E}$ at generating a bordism class.

\begin{proposition}\label{L: F}
If $\mc E$ is a class of stratified pseudomanifold singularities, then $\mc F_{\mc S_{\mc E}}=\mc F_{\mc G_{\mc E}}$ and $S\mc F_{\mc S_{\mc E}}=S\mc F_{\mc G_{\mc E}}$. 
\end{proposition}
\begin{proof}
By Lemma \ref{L: akin is pm}, if $\mc G$ is a class of pseudomanifold singularities, $\mc F_{\mc G}$ is the class of compact classical $\bd$-pseudomanifolds whose polyhedral links of non-boundary points are in $\mc G$. Since $\mc S_{\mc E}\subset \mc G_{\mc E}$, it follows immediately that $\mc F_{\mc S_{\mc E}}\subset \mc F_{\mc G_{\mc E}}$.
To show that $\mc F_{\mc G_{\mc E}}\subset \mc F_{\mc S_{\mc E}}$, we will show that, if $|X|\in \mc F_{\mc G_{\mc E}}$, then the polyhedral links of $|X-\bd X|$ are in $\mc S_{\mc E}$. 

So, suppose $|X|\in \mc F_{\mc G_{\mc E}}$ so that the polyhedral links of $|X|-|\bd X|$ are in $\mc G_{\mc E}$, or, equivalently by Proposition \ref{P: pm links}, the links of any stratification of $|X|$ are in $\mc E$. Let $|\Lk(x)|$ be a polyhedral link of a point in $|X|-|\bd X|$. We always have $|\bd \Lk(x)|=\emptyset$, and, if $\dim(|\Lk(x)|)=0$, then $|\Lk(x)|$ must be $|S^0|$, as otherwise $X$ could not be a classical pseudomanifold. So suppose $\dim(|\Lk(x)|)>0$. 
By Lemma \ref{L: int link},  $|\Lk(x)|\cong |S^j\ell|$, where $\ell$ is the link of $x$ in the intrinsic stratification $X^*$. So $|\ell|\in \mc E$ by the definition of $\mc G_{\mc E}$, and thus so is the suspension $|\Lk(x)|\cong |S^j\ell|$, unless $|\ell|=\emptyset$, by the properties of $\mc E$.  If $|\ell|=\emptyset$, every $|S^j\emptyset|$ is a sphere, and these are all in $\mc E$ except $|S^0|$. This shows that $|\Lk(x)|$ satisfies the first three properties to be in $\mc S_{\mc E}$.  

Now we must consider the links of $|Lk(x)|$. Suppose  we stratify $|\Lk(x)|$ as the iterated suspension $S^j\ell$ (or as $S^j$ if $\ell=\emptyset$). Then its links are either empty or links of $\ell$ or  suspensions $S^k\ell$, $0\leq k<j$. We know $\emptyset\in \mc E$, and since the links of $\ell$ are also links of $X$, they are in $\mc E$. Finally, since $|\ell|\in \mc E$, so then all the $|S^k\ell|$ are in $\mc E$, unless $|\ell|$ is empty. But if $|\ell|=\emptyset$, then the suspensions are all spheres, which are all in $\mc E$ except $|S^0|$. And since $S^j\ell$ is a classical stratified pseudomanifold, it cannot have $|S^0|$ as a link. So all links of $|\Lk(x)|$, stratified as $S^j\ell$ are contained in $\mc E$, so  all  polyhedral links of $|\Lk(x)|$ are in $\mc G_{\mc E}$ by Proposition \ref{P: pm links}. 

So if $|X|\in \mc F_{\mc G_{\mc E}}$, its polyhedral links of non-boundary points are in $\mc S_{\mc E}$, so $|X|\in \mc F_{\mc S_{\mc E}}$. 

We have shown that $\mc F_{\mc S_{\mc E}}=\mc F_{\mc G_{\mc E}}$. In Section \ref{S: F}, we defined $S\mc F_{\mc G}$ to be  the class consisting of the  orientable objects of $\mc F_{\mc G}$ with each of their orientations. Thus the claim that $S\mc F_{\mc S_{\mc E}}=S\mc F_{\mc G_{\mc E}}$ follows.
\end{proof}

\begin{remark}
How can we reconcile Proposition \ref{L: F} with Remark \ref{R: g not equal s}? For example, why doesn't our space $|X|=|S^1\times S^{2k}\times S\C P^2|$ from that remark, which we have shown is contained in $\mc G_{\mc E}$ for $\mc E=\mc E_{\Q-\text{Witt}}$ but not $\mc S_{\mc E}$, contradict Proposition \ref{L: F}? The answer is that even though this $|X|$ is in $\mc G_{\mc E}$, it cannot arise as a link in $\mc F_{\mc G_{\mc E}}$. The reason is that links of links are links, and so $\C P^2$ would also have to be a link of a space in $\mc F_{\mc G_{\mc E}}$ (i.e. in a $\Q$-Witt space); but $\C P^2$ cannot be a link in a $\Q$-Witt space. In fact, this argument concerning links of links arises directly in the next to last paragraph of the proof of Proposition \ref{L: F}. This shows that there was already second order linking information coming into the formation of $\mc F_{\mc G_{\mc E}}$.

Taken together, Proposition \ref{L: F} and Remark \ref{R: g not equal s} also show that, 
if we let $\mc G_{\mc F}$ be the class of polyhedral links actually occurring in some pseudomanifold bordism class $\mc F$, then it is not always the case that $\mc G_{\mc F_{\mc G}}=\mc G$. This is consistent with our observations in Remark \ref{R: downside} that specifying $\mc E$ does not guarantee that all objects of $\mc E$ can occur as links of stratifications of manifolds in $\mc F_{\mc G_{\mc E}}$ (though we did not yet use the notation $\mc F_{\mc G}$ at that point). 

However, this nice property \emph{will} hold if $\mc G$ is a Siegel class, as we show in the next lemma.
\end{remark}

\begin{lemma}
If $\mc G=\mc S_{\mc E}$ is a Siegel class, then $\mc G_{\mc F_{\mc G}}=\mc G$. If $\mc S_{\mc E}$ is a Siegel class in which all spaces are orientable, then $\mc G_{S\mc F_{\mc G}}=\mc G$.
\end{lemma}
\begin{proof}
It is trivial from the definitions that $\mc G_{\mc F_{\mc G}}\subset \mc G$ and $\mc G_{S\mc F_{\mc G}}\subset \mc G$ for any class of pseudomanifold singularities $\mc G$. 

Now suppose $|Z|\in \mc S_{\mc E}=\mc G$ is an object of a Siegel class, and suppose $|Z|$ has a stratification $Z$. Then $|SZ|\in \mc G$ by the properties of $\mc G$. If $Z\neq S^0$, we can stratify $|SZ|$ as $SZ$, whose links are $Z$ and the links of $Z$. In particular, $|Z|$ is the polyhedral link of the suspension points. Since we assume $\mc G$ is a Siegel class, we must have that $|Z|\in \mc E$ and the links of $Z$ are in $\mc E$. Thus every link of $SZ$ is in $\mc E$, so $SZ\in \mc C_{\mc E}$ and $|SZ|\in \mc F_{\mc G_{\mc E}}$ by Lemma \ref{L: some lemma}. If $Z=S^0$, then also $|SZ|\cong |S^1|\in \mc F_{\mc E}$ (which contains all manifolds). Therefore, all elements of a Siegel class $\mc S_{\mc E}$ actually occur as polyhedral links of spaces in $\mc F_{\mc E}$, which is equal to $\mc F_{\mc S_{\mc E}}$ by Proposition \ref{L: F}. So if $\mc G=\mc S_{\mc E}$ is a Siegel class, then $\mc G_{\mc F_{\mc G}}=\mc G$. 

If we assume that elements of $\mc S_{\mc E}$ are also all orientable and $|Z|\in \mc S_{\mc E}$, then $|SZ|$ is orientable, so in fact $|SZ|\in S\mc F_{\mc E}$. Hence $|Z|\in \mc G_{S\mc F_{\mc G}}$. 
\end{proof}

Finally, let us return to Siegel's class $\mc L'$. Suppose we let $\mc E=S\mc E_{\Q\text{-Witt}}$ be the subclass of $\mc E_{\Q\text{-Witt}}$ consisting of pseudomanifolds in $\mc E_{\Q\text{-Witt}}$ that are also orientable. It is easy to observe that $\mc E$ is a class of stratified pseudomanifold singularities, using that we already know that $ \mc E_{\Q\text{-Witt}}$ is one by Section \ref{E: examples}. 
We claim that $\mc S_{\mc E}=\mc L'$. Indeed, suppose $|Z|\in \mc L'$. By construction  $|S^0|$ is in every Siegel class. If $\dim(|Z|)>0$ is even, then $|Z|$ is orientable and satisfies the Witt vanishing condition and so is in $\mc E$. Furthermore, $|Z|$ is a Witt space, so its polyhedral links must be in $\mc G_{\mc E}$. So $|Z|\in \mc S_{\mc E}$. Conversely, suppose that $|Z|\in \mc S_{\mc E}$ with $\dim(|Z|)>0$. Then the assumption that the polyhedral links of $|Z|$ are in $\mc G_{\mc E}$ makes $|Z|$ a $\Q$-Witt space, and the assumption that $|Z|\in \mc E$ makes $|Z|$ orientable. Furthermore, if $\dim(|Z|)$ is even, then $|Z|\in \mc E$ implies that  the required intersection homology group vanishes, by definition of $\mc E$. So $|Z|\in \mc L'$.  If $\dim(|Z|)=0$, then, since $\mc S_{\mc E}$ is a class of pseudomanifold singularities, $|Z|\cong |S^0|$, which is in $\mc L'$ by definition.  Finally, we observe that $\mc L'$ and $\mc S_{\mc E}$ both contain the empty set. 
 Thus,  $\mc S_{\mc E}=\mc L'$. 

Therefore, putting together Proposition \ref{L: F} and Lemma \ref{L: some lemma}, an orientable compact  classical $\bd$-pseudomanifold whose polyhedral links of non-boundary points are in $\mc L'$ is precisely an element of $|S\mc C_{S\mc E_{\Q-\text{Witt}}}|$, i.e. it is an orientable $\Q$-Witt space. So, in particular, Siegel's definition of oriented Witt bordism  does agree with the oriented Witt bordism  treated here. 
 
Curiously,  
Siegel does not explain in \cite{Si83} why he chose to use $\mc L'$ to describe the bordism theory in \cite[Section IV.1]{Si83} and not $\mc G_{\mc E_{\Q-\text{Witt}}}$, which is essentially the class of pseudomanifold singularities he works with throughout the earlier sections of \cite{Si83}. Nonetheless, as it is a Siegel class, we see that $\mc L'$ provides an efficient way to characterize Witt spaces, at the expense of using second order link information.

\section{Questions}\label{S: Q}

Although we have provided more general definitions where convenient, our principal focus throughout has been upon classes of pseudomanifolds and stratified pseudomanifolds determined by conditions on the underlying spaces of their links or polyhedral links. This includes the previously-studied pseudomanifold bordism groups and homology theories, including Witt bordism and IP space bordism. For such classes, we have shown that the stratified and unstratified bordism groups and homology theories agree. However, we have also defined various other classes of spaces and bordisms that we have not investigated as thoroughly. Here, we collect some questions for future exploration.

In Section \ref{S: all abord}, we defined oriented weak stratified bordism classes $S\mc C$ and showed that they yield bordism groups $\Omega^{\mc C}_*$. A general orieted weak stratified bordism class that is not an oriented IWS class could have bordism groups for which different stratifications of the same underlying space are not bordant. It would be interesting to have non-trivial examples of such groups, though they might be difficult to compute, as computation would likely need to depend on invariants that are sensitive to stratifications. Such invariants could conceivably come from intersection homology groups with perversities that do not meet Goresky and MacPherson's perversity requirements of \cite{GM1}. One possible such invariant would be the perverse signatures of Hunsicker \cite{Hun07} (see also \cite{GBF27}); another possibility would be the  signatures of $L$-spaces \cite{Ba02, ALMP-cheeger}. 

Passing to bordism homology theories, one could more generally define stratified bordism functors for classes of spaces satisfying  stratified versions of all four of Akin's axioms, including the cutting axiom, Axiom 3, which we omitted. Such axioms would directly yield a bordism functor by the same discussion given here, and it should not require much more work to show that this is a bordism homology theory. We have already seen that the $\Omega^{\mc E}_*(\cdot)$ are homology theories using their relationship with the $\Omega^{|\mc E|}_*(\cdot)$ (and similarly in the unoriented cases), but arbitrary stratified theories will not have such corresponding unstratified theories for comparison and so the homology theory axioms would need to be verified directly. Again, it would be interesting to find nontrivial examples of such stratified bordism homology theories. 

In our study of bordism groups, the IWS classes formed something of an intermediate between the weak stratified bordism classes and the classes of the form $\mc C_{\mc E}$, the IWS classes being particular examples of the former and having the latter as particular instances. For bordism groups, Theorem \ref{T: bord group} showed that if $\mc C$ is an IWS class then $\Omega^{\mc C}_*\cong \Omega^{|\mc C|}_*$, and similarly for the unoriented groups. One could imagine defining versions of stratified bordism functors (satisfying versions of Akin's axioms 1, 2, and 4) and stratified bordism homology theories (satisfying versions of all  four of Akin's axioms) based on intrinsic classes of spaces such that, if $X$ is in the class, then so is $X'$ for any other stratification of $|X|$, but without requiring such classes to have the form $\mc C_{\mc E}$. The resulting functors and homology theories would then have well-defined forgetful natural transformations  to corresponding unstratified theories, but it would take some additional work (requiring relative versions of our results of Section \ref{S: bordisms}) to determine whether or not  these  yield isomorphisms of homology theories. Such an undertaking does not seem much beyond the methods of the present paper, but given the already-lengthy volume of this work and the lack of pre-existing motivating examples (a role played by Witt and IP spaces in the classes we have studied in detail), we defer the problem for now.

A question in a different direction would be to classify those classes of stratified pseudomanifold singularities $\mc E$ such that is  true that  $\mc E= \mc E_{\mc G_{\mc E}}$ (see Lemma \ref{L: E and G} and Remark \ref{R: downside}). Similarly, Siegel classes remain somewhat mysterious. For example, for what classes of stratified pseudomanifold singularities $\mc E$ is it true that $\mc G_{\mc E}=\mc S_{\mc E}$? Furthermore, we know from Definition \ref{D: pbc} that every pseudomanifold bordism class $\mc F$ has the form $\mc F_{\mc G}$ for some $\mc G$ and, by Lemma \ref{L: E and G}, that every $\mc G$ has the form $\mc G_{\mc E}$ for some $\mc E$. We also know by Proposition \ref{L: F} that for such an $\mc E$,  $\mc F_{\mc S_{\mc E}}=\mc F_{\mc G_{\mc E}}$. But in general, we know by Remark \ref{R: g not equal s} that  $\mc S_{\mc E}$ and $\mc G_{\mc E}$ are not necessarily equal (though we do have   $\mc S_{\mc E}\subset \mc G_{\mc E}$ by Lemma \ref{L: siegel}). So, a natural question is the following: given an $\mc F$, is there an $\mc E$ such that $\mc F=\mc F_{\mc E}$ and such that $\mc S_{\mc E}=\mc G_{\mc E}$? 
 In some sense, such an $\mc E$ would generate $\mc F$ most efficiently. 

Finally, there is work to be done even among the classes of the form $\mc C_{\mc E}$. Among the examples of Section \ref{E: examples}, we noted that some bordism groups of known classes of pseudomanifolds remain to be computed. Given the remarkable properties of Witt and IP spaces already established in \cite{Si83} and \cite{Pa90}, what other pseudomanifold bordism theories of the form $\Omega^{\mc E}_*(\cdot)\cong \Omega^{|\mc E|}_*(\cdot)$ can be computed, and, conversely, what other extraordinary homology theories can be realized as bordism theories of this type? And how do such geometric homology theories relate to the geometric homology theories of Buoncristiano, Rourke, and Sanderson of \cite{BRS}?

\bibliographystyle{amsplain}
\bibliography{./bib}

\end{document}